\documentclass[10pt]{article}

\usepackage[margin=2cm]{geometry}
\usepackage{mathtools}
\usepackage{latexsym}
\usepackage{amssymb,amsmath,amsthm}
\usepackage[svgnames,x11names,table]{xcolor}
\usepackage{graphicx}
\usepackage{calc}
\usepackage{hyphenat}

\graphicspath{{fig/}}

\newcommand\xqed[1]{%
  \leavevmode\unskip\penalty9999 \hbox{}\nobreak\hfill
  \quad\hbox{#1}}
\newcommand*{\exend}{\xqed{$\triangle$}}

\newcommand*{\NN}{\ensuremath{\mathbb{N}}}
\newcommand*{\QQ}{\ensuremath{\mathbb{Q}}}
\newcommand*{\ZZ}{\ensuremath{\mathbb{Z}}}
\newcommand*{\RR}{\ensuremath{\mathbb{R}}}
\newcommand*{\CC}{\ensuremath{\mathbb{C}}}
\newcommand*{\FF}{\ensuremath{\mathbb{F}}}

\newcommand*{\PP}{\ensuremath{\mathcal{P}}}

\newcommand*{\LL}{\ensuremath{\mathcal{L}}}
\newcommand*{\indx}{\mathcal{I}}
\newcommand*{\iters}{\mathcal{Z}}
\newcommand*{\niters}{|\mathcal{Z}|}

\newcommand*{\from}{\colon}
\newcommand*{\ceq}{\coloneqq}
\newcommand*{\qec}{\eqqcolon}
\newcommand*{\eps}{\varepsilon}
\newcommand*{\set}[1]{\{{#1}\}}
\newcommand*{\dset}[1]{\{{#1}\}}
\newcommand*{\tup}[1]{({#1})}
\newcommand*{\dtup}[1]{({#1})}
\newcommand*{\norm}[1]{ \|  #1 \|}
\newcommand*{\indicator}{1}
\newcommand*{\plusequals}{\mathbin{\text{+=}}}

\newcommand*{\subim}{\le}
\newcommand*{\subpr}{<}

\newcommand{\lt}{\left}
\newcommand{\rt}{\right}

\newcommand*{\restr}[2]{#1\big|_{#2}}

\newcommand*{\barshbl}{{\bar{s}_\mathrm{HBL}}}
\newcommand*{\shbl}{{s_\mathrm{HBL}}}
\newcommand*{\shbla}[1]{s_{\mathrm{HBL}, #1 }}
\newcommand*{\Mat}[2]{M_{#1}({#2})}

\DeclareMathOperator{\opDim}{dim}
\newcommand*{\Dim}[1]{\opDim({#1})}

\DeclareMathOperator{\opRank}{rank}
\newcommand*{\Rank}[1]{\opRank({#1})}

\DeclareMathOperator{\opKernel}{ker}
\newcommand*{\Kernel}[1]{\opKernel({#1})}

\newcommand*{\ForStep}[4]{\ensuremath{\text{\bf{for} } #1=#2:#3:#4}}
\newcommand*{\For}[3]{\ensuremath{\text{\bf{for} } #1=#2:#3}}

\newcommand*{\If}[1]{\ensuremath{\text{\bf{if} } #1}}

\newcommand*{\sectn}[1]{Section~{#1}}
\newcommand*{\sectns}[1]{Sections~{#1}}
\newcommand*{\Sectn}[1]{Section~{#1}}

\newcommand*{\brk}{\allowbreak}

\newcommand{\ignore}[1]{}

\numberwithin{equation}{section}

\makeatletter
\def\th@plain{%
  \thm@notefont{}% same as heading font
  \itshape % body font
}
\def\th@definition{%
  \thm@notefont{}% same as heading font
  \normalfont % body font
}
\makeatother

\theoremstyle{plain}
\newtheorem{theorem}{Theorem}[section]
\newtheorem{definition}[theorem]{Definition}
\newtheorem{notation}[theorem]{Notation}
\newtheorem{remark}[theorem]{Remark}
\newtheorem{lemma}[theorem]{Lemma}
\newtheorem{proposition}[theorem]{Proposition}

\theoremstyle{definition}
\newtheorem*{exMatmul}{Example: Matrix Multiplication}
\newtheorem*{exAk}{Example: Matrix Powering}
\newtheorem*{exNbody}{Example: $N$-Body Simulation and Database Join}
\newtheorem*{exMatvec}{Example: Matrix-Vector Multiplication}
\newtheorem*{exTensor}{Example: Tensor Contraction}
\newtheorem*{exComplicated}{Example: Complicated Code}

\title{% 
Communication Lower Bounds and Optimal Algorithms for \\
Programs That Reference Arrays --- Part 1}

\author{%
Michael\,Christ\thanks{%
Mathematics Dept., Univ.\ of California, Berkeley ({\tt mchrist@math.berkeley.edu}).},%
\,James\,Demmel\thanks{%
Computer Science Div.\ and Mathematics Dept., Univ.\ of California, Berkeley ({\tt demmel@cs.berkeley.edu}).},%
\,Nicholas\,Knight\thanks{%
Computer Science Div., Univ.\ of California, Berkeley ({\tt knight@cs.berkeley.edu}).},%
\,Thomas\,Scanlon\thanks{%
Mathematics Dept., Univ.\ of California, Berkeley ({\tt scanlon@math.berkeley.edu}).},%
\,and\,Katherine\,Yelick\thanks{%
Computer Science Div., Univ.\ of California, Berkeley and Lawrence Berkeley Natl.\ Lab.\ ({\tt yelick@cs.berkeley.edu}).}}
\date{\today}

\begin{document}
\maketitle
\begin{abstract}
Communication, i.e., moving data, between levels of a memory hierarchy or between parallel processors on a network, can greatly dominate the cost of computation, so algorithms that minimize communication can run much faster (and use less energy) than algorithms that do not.
Motivated by this, attainable communication lower bounds were established in \cite{hongkung,ITT04,BallardDemmelHoltzSchwartz11} for a variety of algorithms including matrix computations.
The lower bound approach used initially in \cite{ITT04} for $\Theta(N^3)$ 
matrix multiplication, and later in \cite{BallardDemmelHoltzSchwartz11} 
for many other linear algebra algorithms, depended on a geometric result by 
Loomis and Whitney \cite{LW49}: this result bounded the volume of a 3D set 
(representing multiply-adds done in the inner loop of the algorithm) 
using the product of the areas of certain 2D projections of this set 
(representing the matrix entries available locally, i.e., without communication). 
Using a recent generalization of Loomis' and Whitney's result, we generalize 
this lower bound approach to a much larger class of algorithms, 
that may have arbitrary numbers of loops and arrays with arbitrary dimensions, as long as the index expressions are affine combinations of loop variables.
In other words, the algorithm can do arbitrary operations on any number of 
variables like $A(i_1,i_2,i_2-2i_1,3-4i_3+7i_4,\ldots)$.
Moreover, the result applies to recursive programs, irregular iteration spaces, sparse matrices, 
and other data structures as long as the computation can be logically 
mapped to loops and indexed data structure accesses. 
% essentially including any algorithm that accesses arrays indexed by arbitrary 
% linear combinations of the loop variables.
%
%
We also discuss when optimal algorithms exist that attain the lower bounds; 
this leads to new asymptotically faster algorithms for several problems.
\end{abstract}

\section{Introduction}
\label{sec:intro}
Algorithms have two costs: computation (e.g., arithmetic) and 
communication, i.e., moving data between levels of a memory hierarchy or between processors over a network. 
Communication costs (measured in time or energy per operation) already greatly exceed computation costs, 
and the gap is growing over time following technological trends \cite{FOSC,ComputerPerformanceNRC}.
Thus it is important to design algorithms that minimize communication, and if possible attain communication lower bounds.
In this work, we measure communication cost in terms of the number of words moved (a bandwidth cost), and will not discuss other factors like per-message latency, congestion, or costs associated
with noncontiguous data.
Our goal here is to establish new lower bounds on the communication cost of a much broader class of algorithms than possible before, and when possible describe how to attain these lower bounds.

Communication lower bounds have been a subject of research for a long time. 
Hong and Kung \cite{hongkung} used an approach called {\em pebbling} to establish lower bounds for $\Theta(N^3)$ matrix multiplication and other algorithms.
Irony, Tiskin and Toledo \cite{ITT04} proved the result for $\Theta(N^3)$ matrix multiplication in a different, geometric way, and extended the results both to the parallel case and the case of using redundant copies of the data. 
In \cite{BallardDemmelHoltzSchwartz11} this geometric approach was further generalized to include any algorithm that ``geometrically resembled'' matrix multiplication in a sense to be made clear later, but included most direct linear algebra algorithms (dense or sparse, sequential or parallel), and some graph algorithms as well. 
Of course lower bounds alone are not algorithms, so a great deal of additional work has gone into developing algorithms that attain these lower bounds, resulting in many faster algorithms as well as remaining open problems (we discuss attainability in \sectn{\ref{sec:attain}}).

Our geometric approach, following \cite{ITT04,BallardDemmelHoltzSchwartz11}, works as follows. 
We have a set $\iters$ of arithmetic operations to perform, and the amount of data available locally, i.e., without any communication, is $M$ words. 
For example, $M$ could be the cache size. 
Suppose we can upper bound the number of (useful) arithmetic operations that we can perform with just this data; call this bound $F$.
Letting $|\cdot|$ denote the cardinality of a set, if the total number of arithmetic operations that we need to perform is $\niters$ (e.g., $\niters=N^3$ multiply-adds in the case of dense matrix multiplication on one processor), then we need to refill the cache at least $\niters/F$ times in order to perform all the operations. 
Since refilling the cache has a communication cost of moving $O(M)$ words (e.g., writing at most $M$ words from cache back to slow memory, and reading at most $M$ new words into cache from slow memory), the total communication cost is $\Omega(M\cdot\niters/F)$ words moved.
This argument is formalized in \sectn{\ref{sec:lb}}. 
% which also extends it to parallel algorithms, 
% where each processor only does a fraction of the work, and even more general architectures.

The most challenging part of this argument is determining $F$.
Our approach, described in more detail in \sectns{\ref{sec:model}--\ref{sec:hbl}},  
builds on the work in \cite{ITT04,BallardDemmelHoltzSchwartz11}; in those papers, the algorithm is modeled geometrically using the iteration space of the loop nest, as sketched in the following example.
\begin{exMatmul}[Part~1/5\footnote{This indicates that this is the first of 5 times that matrix multiplication will be used as an example.}]
For $N$--by--$N$ matrix multiplication, with 3 nested loops over indices $(i_1,i_2,i_3)$, any subset $E$ of the $N^3$ inner loop iterations $C(i_1,i_2) \plusequals A(i_1,i_3)\cdot B(i_3,i_2)$ can be modeled as a subset of an $N$--by--$N$--by--$N$ discrete cube $\iters=\set{1,\ldots,N}^3\subset\ZZ^3$.
The data needed to execute a subset of this cube is modeled by its projections
onto faces of the cube, i.e., $(i_1,i_2)$, $(i_1,i_3)$ and $(i_3,i_2)$ for each $(i_1,i_2,i_3)\in E$. $F$ is the maximum number
of points whose projections are of size at most $O(M)$. In \cite{ITT04,BallardDemmelHoltzSchwartz11}
the authors used a geometric theorem of Loomis and Whitney (see Theorem~\ref{thm_LoomisWhitney}, a special case of \cite[Theorem~2]{LW49})
to show $F$ is maximized when $E$ is a cube with edge length $M^{1/2}$,
so $F = O(M^{3/2})$, yielding the communication lower bound
$\Omega(M\cdot\niters/F) = \Omega(M\cdot N^3/M^{3/2}) = \Omega(N^3/M^{1/2})$.
\exend
\end{exMatmul}
\noindent Our approach is based on a major generalization of \cite{LW49} in \cite{BCCT10} 
that lets us geometrically model a much larger class of algorithms with an 
arbitrary number of loops and array expressions involving affine functions of indices.

We outline our results:
\begin{description}
\item[\Sectn{\ref{sec:model}}] introduces the geometric model that we use to compute the bound $F$
(introduced above). We first describe the model for matrix multiplication
(as used in \cite{ITT04,BallardDemmelHoltzSchwartz11}) and apply Theorem~\ref{thm_LoomisWhitney} to obtain a bound of the
form $F = O(M^{3/2})$. 
Then we describe how to generalize the geometric model
to any program that accesses arrays inside loops with subscripts that are affine 
functions of the loop indices, such as $A(i_1,i_2)$, $B(i_2,i_3)$, 
$C(i_1+3i_2,1-i_1+7i_2+8i_3,\ldots)$, etc.
(More generally, there do not need to be explicit loops or array references; see \sectn{\ref{sec:lb}}
for the general case.) 
In this case, we seek a bound of the form $F = O(M^{\shbl})$ for some
constant $\shbl$.
\item[\Sectn{\ref{sec:hbl}}] 
takes the geometric model and in Theorem~\ref{thm:1} proves a bound yielding 
the desired $\shbl$.
HBL stands for H\"{o}lder-Brascamp-Lieb, after the authors of precedents and generalizations 
of \cite{LW49}.
In particular, Theorem~\ref{thm:1} gives the constraints for 
a linear program with integer coefficients whose solution is $\shbl$.
\item[\Sectn{\ref{sec:lb}}] formalizes the argument that an upper bound on $F$ 
yields a lower bound on communication of the form $\Omega (M\cdot\niters/F)$. 
Some technical assumptions are required for this to work. Using the same approach
as in \cite{BallardDemmelHoltzSchwartz11}, we need to eliminate
the possibility that an algorithm could do an unbounded amount of work on a fixed
amount of data without requiring any communication. We also describe how the
bound applies to communication in a variety of computer architectures (sequential,
parallel, heterogeneous, etc.).
\item[\Sectn{\ref{sec:undecidability}}.] Theorem~\ref{thm:1} proves the existence of a certain set of linear constraints, but does not provide an algorithm for writing it down. 
Theorem~\ref{thm:decision} shows we can always compute a linear program with the same feasible region; i.e., the lower bounds discussed in this work are decidable. Interestingly, it may be undecidable to compute the exact set of constraints given by Theorem~\ref{thm:1}:
Theorem~\ref{thm4.1} shows that having such an effective algorithm for an arbitrary 
set of array subscripts (that are affine functions of the loop indices)
is equivalent to a positive solution to Hilbert's Tenth Problem over $\QQ$,
i.e., deciding whether a system of rational polynomial equations has a rational solution. 
Decidability over $\QQ$ is a longstanding open question; this is to be contrasted 
with the problem over $\ZZ$, proven undecidable by 
Matiyasevich-Davis-Putnam-Robinson \cite{Matiyasevich93}, 
or with the problem over $\RR$, which is Tarski-decidable \cite{Tarski-book}.
\item[\Sectn{\ref{sec:bound}}.]
While Theorem~\ref{thm:decision} demonstrates that our lower bounds are decidable, the algorithm given is quite inefficient. 
In many cases in many cases of practical interest we can write down an equivalent linear program in fewer steps; in \sectn{\ref{sec:bound}} we present several such cases 
(Part~2 of this paper will discuss others). For example, when every
array is subscripted by a subset of the loop indices, 
Theorem~\ref{thm:prod} % Theorem~\ref{thm6.1} 
shows we can obtain an equivalent linear program immediately.
For example, this includes
matrix multiplication, which has three loop indices $i_1$, $i_2$, and $i_3$, and
array references $A(i_1,i_3)$, $B(i_3,i_2)$ and $C(i_1,i_2)$.
Other examples include tensor contractions, the direct N-body algorithm,
and database join.
\item[\Sectn{\ref{sec:attain}}] considers when the communication lower bound for an algorithm
is attainable by reordering the executions of the inner loop; we call such an algorithm {\em communication optimal}. 
For simplicity, we consider the case where all reorderings are correct.
Then, for the case discussed in Theorem~\ref{thm:prod}, i.e., when every array is subscripted by a subset
of the loop indices,
Theorem~\ref{thm6.1} shows that the dual linear program
yields the ``block sizes'' needed for a communication-optimal sequential algorithm.
\Sectn{\ref{sec:attain-prod-par}} uses Theorem~\ref{thm:prod} to do the same for a parallel algorithm,
providing a generalization of the ``2.5D'' matrix algorithms in \cite{2.5D_EuroPar}.
Other examples again include tensor contractions, the direct N-body algorithm,
and database join.
\item[\Sectn{\ref{sec:conclusions}}] summarizes our results, and outlines the
contents of Part~2 of this paper.
Part~2 will discuss how to compute lower bounds more efficiently and will include more cases where optimal algorithms are possible,
including a discussion of loop dependencies.
\end{description}
This completes the outline of the paper. 
We note that it is possible to omit the detailed proofs in \sectns{\ref{sec:hbl}~and~\ref{sec:undecidability}}  on a first reading; the rest of the paper is self-contained.

To conclude this introduction, we apply
our theory to two examples (revisited later), and show how to derive communication-optimal sequential
algorithms.
\begin{exMatmul}[Part~2/5]
The code for computing $C=A\cdot B$ is
\begin{align*}
&\For{i_1}{1}{N} ,\quad \For{i_2}{1}{N} ,\quad \For{i_3}{1}{N}, \\
&\qquad C(i_1,i_2) \plusequals A(i_1,i_3) \cdot B(i_3,i_2)
\end{align*}
(We omit the details of replacing `$\plusequals$' by `$=$' when $i_3=1$; this will not affect our asymptotic analysis.)
We record everything we need to know about this algorithm in the following matrix $\Delta$, which has one column for each loop index $i_1,i_2,i_3$, one row for each array $A,B,C$, and ones and zeros to indicate which arrays have which loop indices as subscripts, i.e.,
\[
\Delta = 
\bordermatrix{
  & i_1 & i_2 & i_3  \cr
A & 1   & 0   & 1  \cr
B & 0   & 1   & 1  \cr
C & 1   & 1   & 0
};
\]
for example, the top-left one indicates that $i_1$ is a subscript of $A$.
Suppose we solve the linear program to maximize $x_1 + x_2 + x_3$
subject to $\Delta \cdot (x_1,x_2,x_3)^T \le (1,1,1)^T$, getting $x_1 = x_2 = x_3 = 1/2$.
Then Theorem~\ref{thm6.1} tells us that
$\shbl = x_1 + x_2 + x_3 = 3/2$, yielding the well-known communication lower bound of
$\Omega (N^3/M^{\shbl-1}) = \Omega (N^3/M^{1/2})$ for any sequential implementation with
a fast memory size of $M$. 
Furthermore, a communication-optimal sequential implementation, 
where the only optimization permitted is reordering the execution 
of the inner loop iterations to enable reduced communication
(e.g., using Strassen's algorithm instead is not permitted)
is given by blocking the $3$ loops using blocks of size $M^{x_1}$--by--$M^{x_2}$--by--$M^{x_3}$, i.e., $M^{1/2}$--by--$M^{1/2}$--by--$M^{1/2}$, yielding the code
\begin{align*}
&\ForStep{j_1}{1}{M^{x_1}}{N} ,\quad \ForStep{j_2}{1}{M^{x_2}}{N} ,\quad \ForStep{j_3}{1}{M^{x_3}}{N} ,\\
&\qquad \For{k_1}{0}{M^{x_1}-1} ,\quad \For{k_2}{0}{M^{x_2}-1} ,\quad \For{k_3}{0}{M^{x_3}-1} ,\\
&\qquad\qquad (i_1,i_2,i_3) = (j_1,j_2,j_3)+(k_1,k_2,k_3) \\
&\qquad\qquad C(i_1,i_2) \plusequals A(i_1,i_3) \cdot B(i_3,i_2)
\end{align*}
yielding the well-known blocked algorithm where the innermost three loops multiply
$M^{1/2}$--by--$M^{1/2}$ blocks of $A$ and $B$ and update a block of $C$.
We note that to have all three blocks fit in fast memory simultaneously, they would have to be slightly smaller than  $M^{1/2}$--by--$M^{1/2}$ by a constant factor. 
We will address this constant factor and others, which are important in practice, in Part~2 of this work (see \sectn{\ref{sec:conclusions}}).
\exend
\end{exMatmul}
\begin{exComplicated}[Part~1/4]
The code is
\begin{align*}
&\For{i_1}{1}{N} ,\quad \For{i_2}{1}{N} ,\quad \For{i_3}{1}{N} ,\quad \For{i_4}{1}{N} ,\quad \For{i_5}{1}{N} ,\quad \For{i_6}{1}{N} ,\\
&\qquad A_1(i_1,i_3,i_6) \plusequals \text{func}_1(A_2(i_1,i_2,i_4), A_3(i_2,i_3,i_5), A_4(i_3,i_4,i_6)) \\
&\qquad A_5(i_2,i_6) \plusequals \text{func}_2(A_6(i_1,i_4,i_5), A_3(i_3,i_4,i_6))
\end{align*}
This is not meant to be a practical algorithm but rather an example of the generality of our approach.
We record everything we need to know about this program into a 
6-by-6 matrix $\Delta$ with one column for every
loop index $i_1,\ldots,i_6$, one row for every distinct array expression $A_1,\ldots,A_6$, 
and ones and zeros indicating
which loop index appears as a subscript of which array.  
%We note that $A_3$ appears twice, once with the same subscripts as $A_4$, so we only need one row of $\Delta$ for the shared subscripts $i_3,i_4,i_6$.
We again solve the linear program to maximize $x_1 + \cdots + x_6$ subject to $\Delta \cdot (x_1,\ldots,x_6)^T \le (1,\ldots,1)^T$,
getting $(x_1,\ldots,x_6) = (\frac{2}{7},\frac{3}{7},\frac{1}{7},\frac{2}{7},\frac{3}{7},\frac{4}{7})$.
Then Theorem~\ref{thm6.1} tells us that
$\shbl = x_1 + \cdots + x_6 = \frac{15}{7}$, yielding the communication lower bound of
$\Omega (N^6/M^{\shbl-1}) = \Omega (N^6/M^{8/7})$ for any sequential implementation
with a fast memory size of $M$.
Furthermore, a communication-optimal sequential implementation is given by
blocking loop index $i_j$ by block size of $M^{x_j}$.
In other words, the six innermost loops operate on a block of size $M^{2/7}$--by--$\cdots$--by--$M^{4/7}$.
\exend
\end{exComplicated}
\noindent Other examples appearing later include matrix-vector multiplication, tensor contractions,
the direct $N$--body algorithm, database join, and computing matrix powers $A^k$.

\section{Geometric Model} \label{sec:model}

We begin by reviewing the geometric model of matrix multiplication introduced in \cite{ITT04},
describe how it was extended to more general linear 
algebra algorithms in \cite{BallardDemmelHoltzSchwartz11}, and finally show how to generalize it to the class of programs considered in this work.
% arbitrary array references with subscripts that are linear functions of the array indices.

We geometrically model matrix multiplication $C = A \cdot B$ as sketched in Parts~1/5 and~2/5 of the matrix multiplication example (above).
However the `classical' (3 nested loops) algorithm is organized, we can represent it by a set of integer lattice points indexed by $\indx=(i_1,i_2,i_3)\in\ZZ^3$.
If $A,B,C$ are $N$--by--$N$, these indices occupy an $N$--by--$N$--by--$N$ discrete cube $\iters = \set{1,\ldots,N}^3$.
Each point $\indx \in \iters$ represents an operation $C(i_1,i_2) \plusequals A(i_1,i_3)\cdot B(i_3,i_2)$ in the inner loop, and each projection of $\indx$ onto a face of the cube represents a required operand: e.g., $(i_1,i_2)$ represents $C(i_1,i_2)$ (see Figure~\ref{fig_LoomisWhitney}).
\begin{figure}[htbp!]
\begin{center}
\includegraphics[scale=0.4]{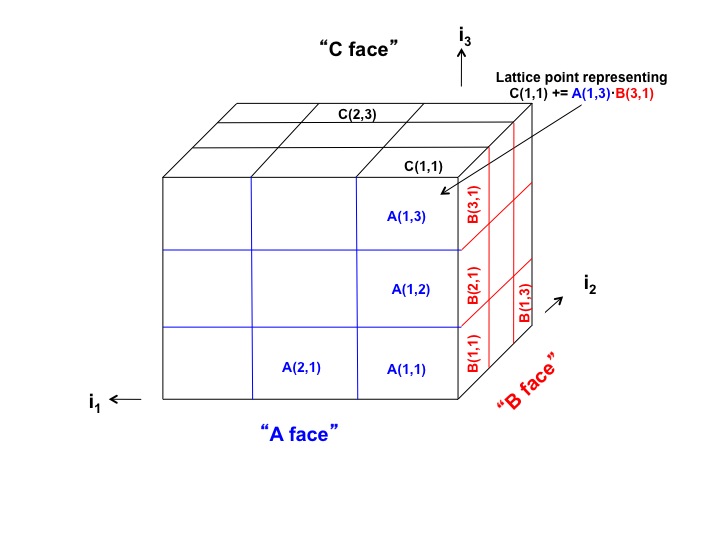}
\caption{Modeling 3x3 Matrix Multiplication as a Lattice 
(each lattice point represented by a cube)}
\label{fig_LoomisWhitney}
\end{center}
\end{figure}

We want a bound $F \ge |E|$, where $E\subseteq \iters$ is any set of lattice points $\indx$ representing operations that can be performed just using data available in fast memory of size $M$.
Let $E_A$, $E_B$ and $E_C$ be projections of $E$ onto faces of the cube representing all the required operands $A(i_1,i_3)$, $B(i_3,i_2)$ and $C(i_1,i_2)$, resp., needed to perform the operations represented by $E$. 
%
% Cannot interleave! 
% (Note that $C(i_1,i_2)$ itself may not be in fast memory, but rather an accumulator containing a partial sum that will eventually be added to $C(i_1,i_2)$.} 
%
% This is not a problem, since we still need a unique word of memory associated with $C(i_1,i_2)$ in fast memory to perform the operation.)
%
A special case of \cite[Theorem~2]{LW49} gives us the desired bound on $E$:
\begin{theorem}[Discrete Loomis-Whitney Inequality, 3D Case]\label{thm_LoomisWhitney}
With the above notation, for any finite subset $E$ of the 3D integer lattice $\ZZ^3$, $|E| \le |E_A|^{1/2} \cdot |E_B|^{1/2} \cdot |E_C|^{1/2}$.
\end{theorem}
\noindent Finally, since the entries represented by $E_A,\brk E_B,\brk E_C$ fit in fast memory by assumption (i.e., $|E_A|,\brk |E_B|,\brk |E_C| \le M$), this yields the desired bound $F \ge |E|$:
\begin{equation}\label{eqn2.1}
|E| \le |E_A|^{1/2} \cdot |E_B|^{1/2} \cdot |E_C|^{1/2} \leq M^{1/2} \cdot M^{1/2} \cdot M^{1/2} = M^{3/2} \qec F \; .
\end{equation}

Irony et al.\  \cite{ITT04} applied this approach to obtain a lower bound for matrix multiplication.
Ballard et al.\ \cite{BallardDemmelHoltzSchwartz11} extended this approach to programs of the form
\[
\begin{split}
\text{for all } &(i_1,i_2,i_3) \in \iters \subset \ZZ^3, \text{  in some order}, \\
 &C(i_1,i_2) = C(i_1,i_2) +_{i_1,i_2} A(i_1,i_3) \ast_{i_1,i_2,i_3} B(i_3,i_2)
\end{split}
\]
and picked the iteration space $\iters$, arrays $A,B,C$, and binary operations $+_{i_1,i_2}$ and $\ast_{i_1,i_2,i_3}$ to represent many (dense or sparse) linear algebra algorithms.
To make this generalization, Ballard et al.\ made several observations, which also apply to our case, below.

 First, explicitly nested loops are not necessary, as long as one can identify each execution of the statement(s) in the inner loop with a unique lattice point; for example, many recursive computations can also match this model.
Second, the upper bound $F$ when $\iters$ is an $N$--by--$N$--by--$N$ cube is valid for any subset $\iters \subset \ZZ^3$.  
Thus, we can use this bound for sparse linear algebra; sparsity may make the bound hard or impossible to attain, but it is still valid.  
Third, the memory locations represented by array expressions like $A(i_1,i_3)$ need not be contiguous in memory, as long as there is an injective\footnote{We assume injectivity since otherwise an algorithm could potentially access only the same few memory locations over and over again. Given our asymptotic analysis, however, we could weaken this assumption, allowing a constant number of array entries to occupy the same memory address.} mapping from array expressions to memory locations.  
In other words, one can shift the index expressions like $A(i_1+3,i_3+7)$, or use pointers, or more complex kinds of data structures, to determine where the array entry actually resides in memory.
Fourth, the arrays $A$, $B$ and $C$ do not have to be distinct; in the case of (in-place) $LU$ factorization they are in fact the same, since $L$ and $U$ overwrite $A$.  
Rather than bounding each array, say $|E_A| \leq M$, if we knew that $A$, $B$ and $C$ did not overlap, we could tighten the bound by a constant factor by using $|E_A| + |E_B| + |E_C| \leq M$.
(For simplicity, we will generally not worry about such constant factors in this paper, and defer their discussion to Part~2 of this work.)

Now we consider our general case, which we will refer to as the \emph{geometric model}.
This model generalizes the preceding one by allowing an arbitrary number of loops defining the index space and an arbitrary number of arrays each with arbitrary dimension.
The array subscripts themselves can be arbitrary affine (integer) combinations of the loop indices, e.g., $A(i_2,i_3-2i_4,i_1+i_2+3)$. 
%
%; \sectn{\ref{sec:lb}} will be more concrete about when we may apply this model to  specific algorithms and their implementations.
\begin{definition}
\label{def:geometricmodel}
An instance of the geometric model is an abstract representation of an algorithm, taking the form
 \begin{equation}
\label{eqn_AlgModel}
\begin{split}
\text{\rm for all } &\indx \in \iters \subseteq \ZZ^d, \text{ \rm in some order}, \\
 &\text{\rm inner\_loop}(\indx,(A_1,\ldots,A_m),(\phi_1,\ldots,\phi_m)) 
\end{split}
\end{equation}
For $j\in\set{1,\ldots,m}$, the functions $\phi_j \from \ZZ^d \to \ZZ^{d_j}$ are $\ZZ$--affine maps, and the functions $A_j \from \phi_j(\ZZ^d) \to \set{\text{\rm variables}}$ are injections into some set of variables.
The subroutine $\text{\rm inner\_loop}$ constitutes a fairly arbitrary section of code, with the constraint that it accesses each of the variables $A_1(\phi_1(\indx)),\ldots,A_m(\phi_m(\indx))$, either as an input, output, or both, in every iteration $\indx \in \iters$.
\end{definition}
\noindent (We will be more concrete about what it means to `access a variable' when we introduce our execution model in \sectn{\ref{sec:lb}}.)

Assuming that all variables are accessed in each iteration rules out certain optimizations, or requires us not to count certain inner loop iterations in our lower bound.
For example, consider Boolean matrix multiplication, where the inner loop body is $C(i_1,i_2) = C(i_1,i_2) \vee (A(i_1,i_3) \wedge B(i_3,i_2))$. 
As long as $A(i_1,i_3) \wedge B(i_3,i_2)$ is false, $C(i_1,i_2)$ does not need to be accessed. 
So we need either to exclude such optimizations, or not count such loop iterations.
And once $C(i_1,i_2)$ is true, its value cannot change, and no more values of $A(i_1,i_3)$ and $B(i_3,i_2)$ need to be accessed; if this optimization is performed, then these loop iterations would not occur and not be counted in the lower bound.
In general, if there are data-dependent branches in the inner loop (e.g., `flip a coin\ldots'), we may not know to which subset $\iters$ of all the loop iterations our model applies until the code has executed.
We refer to a later example (database join) and \cite[\sectn{3.2.1}]{BallardDemmelHoltzSchwartz11} for more examples and discussion of this assumption.

Following the approach for matrix multiplication, we need to bound
$|E|$ where $E$ is a nonempty finite subset of $\iters$. To execute $E$, we need array entries
$A_1(\phi_1(E)),\ldots,\brk A_m(\phi_m(E))$.
Analogous to the analysis of matrix multiplication, we need to 
bound $|E|$  in terms of the number of array entries needed:
$|\phi_j(E)|$ entries of $A_j$, for $j=\set{1,\ldots,m}$.

Suppose there were nonnegative constants $s_1,\ldots,s_m$ such that
for any such $E$,
\begin{equation}\label{eqn2.2}
|E| \leq \prod_{j=1}^m |\phi_j(E)|^{s_j}
\end{equation}
i.e., a generalization of Theorem~\ref{thm_LoomisWhitney}.
Then just as with matrix multiplication, if the number of entries
of each array were bounded by $|\phi_j(E)| \leq M$, we could bound
\begin{equation*} %\label{eqn2.3}
|E| \leq \prod_{j=1}^m |\phi_j(E)|^{s_j} \leq \prod_{j=1}^m M^{s_j}
= M^{\sum_{j=1}^m s_j} \qec M^{\shbl} \qec F
\end{equation*}
where $\shbl \ceq \sum_{j=1}^m s_j$.

The next section shows how to determine the constants $s_1,\ldots,s_m$
such that inequality~\eqref{eqn2.2} holds.
% when they exist, they will be rational numbers in $[0,1]$.
To give some rough intuition
for the result, consider again the special case of Theorem~\ref{thm_LoomisWhitney},
where $|E|$ can be interpreted as a volume, and each $|\phi_j(E)|$
as an area. Thus one can think of $|E|$ as having units of,
say, ${\rm meters}^3$, and $|\phi_j(E)|$ as having units of
${\rm meters}^2$. Thus, for $\prod_{j=1}^3 |\phi_j(E)|^{s_j}$
to have units of ${\rm meters}^3$, we would expect the $s_j$ to
satisfy $3 = 2s_1 + 2s_2 + 2s_3$. But if $E$ were a lower dimensional
subset, lying say in a plane, then $|E|$ would have units
${\rm meters}^2$ and we would get a different constraint on the $s_j$.
This intuition is formalized in Theorem~\ref{thm:1} below.

\section{Upper bounds from HBL theory} \label{sec:hbl}

As discussed above, we aim to establish an upper bound of the form \eqref{eqn2.2} for $|E|$, for any finite set $E$ of loop iterations.
In this section we formulate and prove Theorem~\ref{thm:1}, which states that such a bound is valid if and only if the exponents $s_1,\ldots,s_m$ satisfy a certain finite set of linear constraints, specified in \eqref{subcriticalhypothesis} below.

\subsection{Statement and Preliminaries}
%
% Recall from the previous section we are given some fixed finite subset $\iters\subset \ZZ^d$ of loop iterations $\indx = (i_1,\ldots,i_d)$.
%
% However, the upper bound developed below in Theorem~\ref{thm:1} is valid for any (finite) subset of $\ZZ^d$, so we will not mention $\iters$ explicitly again this section.

As in the previous section, $\indx = \tup{i_1,\ldots,i_d}$ denotes a point in $\ZZ^d$.
For each $j \in \set{1,2,\ldots,m}$ we are given a $\ZZ$--affine map $\phi_j \from \ZZ^d \to \ZZ^{d_j}$, i.e., a function that returns an affine integer combination of the coordinates $i_1,\ldots,i_d$ as a $d_j$--tuple, where $d_j$ is the dimension (number of arguments in the subscript) of the array $A_j$. 
So, the action of $\phi_j$ can be represented as multiplication by a $d_j$--by--$d$ integer matrix, followed by a translation by an integer vector of length $d_j$. 
As mentioned in \sectn{\ref{sec:model}}, there is no loss of generality in assuming that $\phi_j$ is linear, rather than affine. 
This is because $A_j$ can be any injection (into memory addresses), so we can always compose the translation (a bijection) with $A_j$.
In this section we assume that this has been done already.

\begin{notation}
We establish some group theoretic notation; for an algebra reference see, e.g., \cite{lang2002algebra}.
We regard the set $\ZZ^r$ (for any natural number $r$) as an additive Abelian group, and any $\ZZ$--linear map $\phi \from \ZZ^r \to \ZZ^{r'}$ as a group homomorphism.
An Abelian group $G$ has a property called its \emph{rank}, related to the dimension of a vector space: the rank of $G$ is the cardinality of any maximal subset of $G$ that is linearly independent.
All groups discussed in this paper are Abelian and finitely generated, i.e., generated over $\ZZ$ by finite subsets.
Such groups have finite ranks.
For example, the rank of $\ZZ^r$ is $r$.
The notation $H \subim G$ indicates that $H$ is a subgroup of $G$, and the notation $H \subpr G$ indicates that $H$ is a proper subgroup.
\end{notation}

The following theorem provides bounds which are fundamental to our application.
\begin{theorem} 
\label{thm:1}
Let $d$ and $d_j$ be nonnegative integers and $\phi_j \from \ZZ^d \to \ZZ^{d_j}$ be group homomorphisms for $j\in\set{1,2,\ldots,m}$, and consider some
% Consider an Abelian group HBL datum $\tup{\ZZ^d,\tup{\ZZ^{d_j}},\tup{\phi_j}}$ and 
$s \in [0,1]^m$.
Suppose that
\begin{equation} \label{subcriticalhypothesis} 
\Rank{H} \le \sum_{j=1}^m s_j\Rank{\phi_j(H)} \qquad \text{for all subgroups } H \subim \ZZ^d.
\end{equation}
Then
\begin{equation} \label{mainconclusion}
|E|\le\prod_{j=1}^m |\phi_j(E)|^{s_j} \qquad \text{for all nonempty finite sets } E \subseteq \ZZ^d.
\end{equation}
Conversely, if \eqref{mainconclusion} holds for $s \in [0,1]^m$, then $s$ satisfies \eqref{subcriticalhypothesis}.
\end{theorem}
While \eqref{subcriticalhypothesis} may suggest that there are an infinite number of inequalities constraining $s$ (one for each subgroup $H \subim \ZZ^d$), in fact each rank is an integer in the range $[0,d]$, and so there are at most $(d+1)^{m+1}$ different inequalities.

Note that if $s\in[0,\infty)^m$ satisfies \eqref{subcriticalhypothesis} or \eqref{mainconclusion}, then any $t\in[0,\infty)^m$ where each $t_j \ge s_j$ does too. 
This is obvious for \eqref{subcriticalhypothesis}; in the case of \eqref{mainconclusion}, since $E$ is nonempty and finite, then for each $j$, $1 \le |\phi_j(E)| < \infty$, so $|\phi_j(E)|^{s_j}\le |\phi_j(E)|^{t_j}$.

The following result demonstrates there is no loss of generality restricting $s\in[0,1]^m$, rather than in $[0,\infty)^m$, in Theorem~\ref{thm:1}.
\begin{proposition} \label{prop:maxexponent}
Whenever \eqref{subcriticalhypothesis} or \eqref{mainconclusion} holds for some $s\in[0,\infty)^m$, it holds for $t\in[0,1]^m$ where $t_j = \min(s_j,1)$ for $j\in\set{1,\ldots,m}$.
\end{proposition}

\subsection{Generalizations}
\label{sec:hbl-generalizations}
We formulate here two extensions of Theorem~\ref{thm:1}.
Theorem~\ref{thm:2} extends Theorem~\ref{thm:1} both by replacing the sets $E$ by functions and by replacing the codomains $\ZZ^{d_j}$ by finitely generated Abelian groups with torsion; this is a digression which is not needed for our other results.
Theorem~\ref{thm:field} instead replaces torsion-free finitely generated Abelian groups by finite-dimensional vector spaces over an arbitrary field.
In the case this field is the set of rational numbers $\QQ$, it represents only a modest extension of Theorem~\ref{thm:1}, but the vector space framework is more convenient for our method of proof. 
We will show in \Sectn{\ref{sec:hbl-generalizations}} how Theorem~\ref{thm:field} implies Theorem~\ref{thm:1} and Theorem~\ref{thm:2}.
The main part of \Sectn{\ref{sec:hbl}} will be devoted to the proof of Theorem~\ref{thm:field}.

\begin{notation}
Let $X$ be any set.
In the discussion below, $X$ will always be either a finitely generated Abelian group, or a finite-dimensional vector space over a field $\FF$.
For $p\in[1,\infty)$, $\ell^p(X)$ denotes the space of all $p$--power summable functions $f\from X\to\CC$, equipped with the norm 
\[ \norm{f}_p = \lt(\sum_{x\in X} |f(x)|^p \rt)^{1/p}.\]
$\ell^\infty(X)$ is the space of all bounded functions $f\from X\to\CC$, equipped with the norm $\norm{f}_\infty = \sup_{x\in X} |f(x)|$.
These definitions apply even when $X$ is uncountable: if $p\in[1,\infty)$, any function in $\ell^p(X)$ vanishes at all but countably many points in $X$, and the $\ell^\infty$--norm is still the supremum of $|f|$, not for instance an essential supremum defined using Lebesgue measure.
\end{notation}

We will use the terminology \emph{HBL datum} to refer to either of two types of structures; the first is defined as follows:
\begin{definition}
\label{def:groupdatum}
An \emph{Abelian group HBL datum} is a $3$--tuple 
\[ 
\dtup{G,\dtup{G_j},\dtup{\phi_j}} = \dtup{G,\dtup{G_1,\ldots,G_m},\dtup{\phi_1,\ldots,\phi_m}} 
\] 
where $G$ and each $G_j$ are finitely generated Abelian groups, $G$ is torsion-free, each $\phi_j\from G \to G_j$ is a group homomorphism, and the notation on the left-hand side always implies that $j\in\set{1,2,\ldots,m}$.
\end{definition}

\begin{theorem} 
\label{thm:2}
Consider an Abelian group HBL datum $\tup{G,\tup{G_j},\tup{\phi_j}}$ and $s\in[0,1]^m$.
Suppose that
\begin{equation} \label{subcriticalhypothesis2}
\Rank{H}\le \sum_{j=1}^m s_j\Rank{\phi_j(H)} \qquad \text{for all subgroups } H \subim G.
\end{equation}
Then
\begin{equation} \label{BL}
\sum_{x\in G} \prod_{j=1}^m f_j(\phi_j(x)) \le\prod_{j=1}^m \norm{f_j}_{1/s_j} \qquad  \text{for all functions } 0 \le f_j \in \ell^{1/s_j}(G_j).
% \text{for every } f\in\prod_{j=1}^m\dset{\ell^{1/s_j}(G_j) \ni f_j \from G_j \to [0,\infty)}.
\end{equation}
In particular,
\begin{equation} \label{BL;sets} 
|E| \le \prod_{j=1}^m |\phi_j(E)|^{1/s_j} \qquad\text{for all nonempty finite sets } E \subseteq G.
\end{equation}
Conversely, if \eqref{BL;sets} holds for $s\in[0,1]^m$, then $s$ satisfies \eqref{subcriticalhypothesisfield}.
%Conversely, if \eqref{BL} holds for $s\in[0,1]^m$, then $s$ satisfies \eqref{subcriticalhypothesis2}.
\end{theorem}
\begin{remark} \label{rmk:BCCT10}
\cite[Theorem~2.4]{BCCT10} treats the more general situation in which
$G$ is not required to be torsion-free, and establishes the conclusion \eqref{BL} in the weaker form
\begin{equation} \label{BL-C}
\sum_{x\in G} \prod_{j=1}^m f_j(\phi_j(x)) \le C\prod_{j=1}^m \norm{f_j}_{1/s_j} \qquad  \text{for all functions } 0 \le f_j \in \ell^{1/s_j}(G_j).
% \text{for every } f\in\prod_{j=1}^m\dset{\ell^{1/s_j}(G_j) \ni f_j \from G_j \to [0,\infty)}.
\end{equation}
where the constant $C < \infty$ depends only on $G$ and $\set{\phi_1,\ldots,\phi_m}$. 
The constant $C=1$ established here is optimal.
Indeed, consider any single $x \in G$, and for each $j$ define $f_j$ to be $\indicator_{\phi_j(x)}$, the indicator function for $\phi_j(x)$.
Then both sides of the inequalities in \eqref{BL} are equal to $1$.
In Remark~\ref{rmk:BCCT10-T}, we will show that the constant $C$ is bounded above by the number of torsion elements of $G$.
\end{remark}
\begin{proof}[Proof of Theorem~\ref{thm:1}]
Theorem~\ref{thm:1} follows from Theorem~\ref{thm:2} by setting $G=\ZZ^d$ and $G_j=\ZZ^{d_j}$ for each $j$.
\end{proof}
\begin{proof}[Proof of Theorem~\ref{thm:2} (necessity)]
Necessity of \eqref{subcriticalhypothesis2} follows from an argument given in \cite[Theorem~2.4]{BCCT10}.

Consider any subgroup $H \subim G$. 
Let $r=\Rank{H}$. 
By definition of rank, there exists a set $\set{e_i}_{i=1}^r$ of elements of $H$ such that for any coefficients $m_i,n_i\in\ZZ$, $\sum_{i=1}^r m_i e_i = \sum_{i=1}^r n_ie_i$ only if $m_i=n_i$ for all $i$.
For any positive integer $N$ define $E_N$ to be the set of all elements of the form $\sum_{i=1}^r n_i e_i$, where each $n_i$ ranges freely over $\set{1,2,\ldots,N}$. 
Then $|E_N| = N^r$.

On the other hand, for $j\in\set{1,\ldots,m}$,
\begin{equation}\label{necessitycount} 
|\phi_j(E_N)|\le A_j N^{\Rank{\phi_j(H)}}, 
\end{equation}
where $A_j$ is a finite constant which depends on $\phi_j$, on the structure of $H_j$, and on the choice of $\set{e_i}$, but not on $N$.
Indeed, it follows from the definition of rank that for each $j$ it is possible to permute the indices $i$ so that for each $i>\Rank{\phi_j(H)}$ there exist integers $k_i$ and $\kappa_{i,l}$ such that 
\[
k_i \phi_j(e_i)  = \sum_{l=1}^{\Rank{\phi_j(H)}} \kappa_{i,l} \phi_j(e_l).
\]
The upper bound \eqref{necessitycount} follows from these relations.

Inequality \eqref{mainconclusion} therefore asserts that 
\[ 
N^{\Rank{H}} \le \prod_{j=1}^m A_j^{s_j} N^{s_j\Rank{\phi_j(H)}} = AN^{\sum_{j=1}^m s_j\Rank{\phi_j(H)}} 
\]
where $A<\infty$ is independent of $N$.
By letting $N$ tend to infinity, we conclude that $\Rank{H} \le {\sum_{j=1}^m s_j\Rank{\phi_j(H)}}$, as was to be shown.
\end{proof}

%\begin{proof}[Proof of Theorem~\ref{thm:2}]
%According to the structure theory of finitely generated Abelian groups,
%each group $G_j$ is isomorphic to a direct sum
%$\ZZ^{d_j}\oplus T_j$ where $d_j=\Rank{G_j}$ and $T_j$ is a finitely
%generated torsion group, that is, $\Rank{T_j}=0$.
%There is a natural homomorphism $\pi_j:G_j\to\ZZ^{d_j}$ defined by $\pi_j(z\oplus t) = z$
%where $z\in\ZZ^{d_j}$ and $t\in T_j$.
%Define homomorphisms $\psi_j: G\to\ZZ^{d_j}$ by $\psi_j = \pi_j\circ\phi_j$. 
%$G$ is isomorphic to $\ZZ^d$ where $d=\Rank{G}$, so $\psi_j$ can be regarded as $\ZZ$--linear
%maps from $\ZZ^d$ to $\ZZ^{d_j}$.
%
%For any subgroup $H\subim G$, $\psi_j(H) = \pi_j(\phi_j(H))$. Since the nullspace $T_j$ of
%$\pi_j$ is a finite group, $\Rank{\psi_j(H)}=\Rank{\phi_j(H)}$. 
%Therefore the datum $(G,\{G_j\},\{\phi_j\})$ satisfies \eqref{subcriticalhypothesis}
%with respect to exponents $\{s_j\}$
%if and only if the datum $(G,\{\ZZ^{d_j}\},\{\psi_j\})$ does so for that same
%collection of exponents.
%
%Suppose that the datum $(G,\{G_j\},\{\phi_j\})$ satisfies \eqref{subcriticalhypothesis}
%with respect to $\{s_j\}$.  By Theorem~\ref{thm:1}, for any finite set $E\subseteq G$,
%\[ |E|\le \prod_{j=1}^m |\psi_j(E)|^{s_j}.  \] 
%Now $|\psi_j(E)| = |\pi_j(\phi_j(E))| \le |\phi_j(E)|$,
%so 
%$|E|\le \prod_{j=1}^m |\phi_j(E)|^{s_j}$, as was to be shown.
%\end{proof}

We show sufficiency in Theorem~\ref{thm:2} in its full generality is a consequence of the special case in which all of the groups $G_j$ are torsion-free. 
\begin{proof}[Reduction of Theorem~\ref{thm:2} (sufficiency) to the torsion-free case]
Let us suppose that the Abelian group HBL datum $\tup{G,\tup{G_j},\tup{\phi_j}}$ and $s$ satisfy \eqref{subcriticalhypothesis2}.
According to the structure theorem of finitely generated Abelian groups, each $G_j$ is isomorphic to $\tilde G_j \oplus T_j$ where $T_j$ is a finite group and $\tilde G_j$ is torsion-free.
Here $\oplus$ denotes the direct sum of Abelian groups; $G'\oplus G''$ is the Cartesian product $G'\times G''$ equipped with the natural componentwise group law.
Define $\pi_j \from \tilde G_j\oplus T_j\to\tilde G_j$ to be the natural projection; thus $\pi_j(x,t)=x$ for $\tup{x,t} \in \tilde G_j \times T_j$.
Define $\tilde \phi_j = \pi_j\circ\phi_j \from G \to \tilde G_j$.

If $K$ is a subgroup of $G_j$, then $\Rank{\pi_j(K)} = \Rank{K}$ since the kernel $T_j$ of $\pi_j$ is a finite group.
Therefore for any subgroup $H\subim G$, $\Rank{\tilde \phi_j(H)}=\Rank{\phi_j(H)}$.  
Therefore whenever $\tup{G,\tup{G_j},\tup{\phi_j}}$ and $s$ satisfy the hypotheses of Theorem~\ref{thm:2}, $\tup{G,\tup{\tilde G_j},\tup{\tilde \phi_j}}$ and $s$ also satisfy those same hypotheses.

Under these hypotheses, consider any $m$--tuple $f=\tup{f_j}$ from \eqref{BL}, and for each $j$, define $\tilde f_j$ by
\[ 
\tilde f_j(y) = \max_{t\in T_j} f_j(y,t), \text{ for } y\in \tilde G_j.  
\]
For any $x\in G$, $f_j(\phi_j(x))\le \tilde f_j(\tilde \phi(x))$.
Consequently $\prod_{j=1}^m f_j(\phi_j(x))\le\prod_{j=1}^m \tilde f_j(\tilde \phi_j(x))$.

We are assuming validity of Theorem~\ref{thm:2} in the torsion-free case. 
Its conclusion asserts that 
\[ \sum_{x\in G} \prod_{j=1}^m \tilde f_j(\tilde \phi_j(x)) \le \prod_{j=1}^m \norm{\tilde f_j}_{1/s_j}. \]
For each $j$, $\norm{\tilde f_j}_{1/s_j} \le \norm{f_j}_{1/s_j}$. 
This is obvious when $s_j=0$. 
If $s_j \ne 0$ then
\begin{align*}
\norm{\tilde f_j}_{1/s_j}^{1/s_j}
=\sum_{y\in \tilde G_j} \tilde f_j(y)^{1/s_j}
= \sum_{y\in \tilde G_j} \max_{t\in T_j}f_j(y,t)^{1/s_j}
\le \sum_{y\in \tilde G_j} \sum_{t\in T_j}f_j(y,t)^{1/s_j}
= \norm{f_j}_{1/s_j}^{1/s_j}.
\end{align*}
Combining these inequalities gives the conclusion of Theorem~\ref{thm:2}.
%Conversely, if $\tup{G,\tup{G_j},\tup{\phi_j}}$ and $s$ satisfy \eqref{BL}, then we can define $\tilde f_j (x,t) = f_j(x,t)$ if $t=0$, the identity element of $T_j$, and $\tilde f_j(x,t) = 0$ otherwise. 
%%
%This is equivalent to \eqref{BL} for $\tup{G,\tup{\tilde G_j},\tup{\tilde \phi_j}}$ and $s$, and we may then apply the converse of Theorem~\ref{thm:2} in the torsion-free case.
%%
%Recalling from above that for any $H \subim G$, $\Rank{\tilde \phi_j (H)} = \Rank{\phi_j(H)}$, we see that $\tup{G,\tup{G_j},\tup{\phi_j}}$ and $s$ satisfy \eqref{subcriticalhypothesis2}.
%
\end{proof}

Our second generalization of Theorem~\ref{thm:1} is as follows.
\begin{notation}
$\Dim{V}$ will denote the dimension of a vector space $V$ over a field $\FF$; all our vector spaces are finite-dimensional.
Similar to our notation for subgroups, $W \subim V$ indicates that $W$ is a subspace of $V$, and $W \subpr V$ indicates that $W$ is a proper subspace.
\end{notation}
\begin{definition}
\label{def:vectorspacedatum}
A \emph{vector space HBL datum} is a $3$--tuple 
\[ 
\dtup{V,\dtup{V_j},\dtup{\phi_j}} = \dtup{V,\dtup{V_1,\ldots,V_m},\dtup{\phi_1,\ldots,\phi_m}} 
\] 
where $V$ and $V_j$ are finite-dimensional vector spaces over a field $\FF$, $\phi_j\from V \to V_j$ is an $\FF$--linear map, and the notation on the left-hand side always implies that $j\in\set{1,2,\ldots,m}$.
\end{definition}
\begin{theorem} \label{thm:field}
Consider a vector space HBL datum $\tup{V,\tup{V_j},\tup{\phi_j}}$ and $s\in[0,1]^m$.
Suppose that
\begin{equation} \label{subcriticalhypothesisfield} 
\Dim{W} \le \sum_{j=1}^m s_j\Dim{\phi_j(W)} \qquad \text{for all subspaces } W \subim V.
\end{equation}
Then
\begin{equation} \label{BLfield}
\sum_{x\in V} \prod_{j=1}^m f_j(\phi_j(x)) \le\prod_{j=1}^m \norm{f_j}_{1/s_j} \qquad \text{for all functions } 0 \le f_j \in \ell^{1/s_j}(V_j).
%\text{for every } f\in\prod_{j=1}^m\dset{\ell^{1/s_j}(G_j) \ni f_j \from G_j \to [0,\infty)}.
\end{equation}
In particular,
\begin{equation} \label{BLfield;sets} 
|E| \le \prod_{j=1}^m |\phi_j(E)|^{1/s_j} \qquad\text{for all nonempty finite sets } E \subseteq V.
\end{equation}
Conversely, if \eqref{BLfield;sets} holds for $s\in[0,1]^m$, and if $\FF=\QQ$ or if $\FF$ is finite, then $s$ satisfies \eqref{subcriticalhypothesisfield}.
\end{theorem}
The conclusions \eqref{BL} and \eqref{BLfield} remain valid for 
%$m$--tuples $f$ 
functions $f_j$ without the requirement that the $\ell^{1/s_j}$ norms are finite, under the convention that the product $0\cdot\infty$ is interpreted as zero whenever it arises. 
For then if any $f_j$ has infinite norm, then either the right-hand side is $\infty$ while the left-hand side is finite, or both sides are zero; in either case, the inequality holds. 

\begin{proof}[Proof of Theorem~\ref{thm:2} (sufficiency) in the torsion-free case]
Let us suppose that the Abelian group HBL datum $\tup{G,\tup{G_j},\tup{\phi_j}}$ and $s$ satisfy the hypothesis \eqref{subcriticalhypothesis2} of Theorem~\ref{thm:2}; furthermore suppose that each group $G_j$ is torsion-free. 
$G_j$ is isomorphic to $\ZZ^{d_j}$ where $d_j=\Rank{G_j}$. 
Likewise $G$ is isomorphic to $\ZZ^d$ where $d=\Rank{G}$.
Thus we may identify $\tup{G,\tup{G_j},\tup{\phi_j}}$ with $\tup{\ZZ^d,\tup{\ZZ^{d_j}},\tup{\tilde \phi_j}}$, where each homomorphism $\tilde\phi_j\from\ZZ^d\to\ZZ^{d_j}$ is obtained from $\phi_j$ by composition with these isomorphisms.
Defining scalar multiplication in the natural way (i.e., treating $\ZZ^d$ and and $\ZZ^{d_j}$ as $\ZZ$--modules), we represent each $\ZZ$--linear map $\tilde\phi_j$ by a matrix with integer entries.

Let $\FF=\QQ$.
Regard $\ZZ^d$ and $\ZZ^{d_j}$ as subsets of $\QQ^d$ and of $\QQ^{d_j}$, respectively. 
Consider the vector space HBL datum $\tup{V,\tup{V_j},\tup{\psi_j}}$ defined as follows.
Let $V=\QQ^d$ and $V_j = \QQ^{d_j}$, and let $\psi_j\from\QQ^d \to\QQ^{d_j}$ be a $\QQ$--linear map represented by the same integer matrix as $\tilde \phi_j$.

%Consider any subgroup $H<\ZZ^d$. Then $\Rank{H}=\dim(W)$ where $W<V$ is the vector subspace spanned
%by $H$ over $\QQ$.
%Moreover $\psi_j(W)$ equals the span over $\QQ$ of  $\phi_j(H)$, and $\dim(\psi_j(W)) = \Rank{\phi_j(H)}$.
%Therefore any subspace $W$ of $V$ spanned over $\QQ$ by a subgroup of $\ZZ^d$ satisfies the hypothesis
%\eqref{subcriticalhypothesisfield} of Theorem~\ref{thm:field}.

Observe that $\tup{V,\tup{V_j},\tup{\psi_j}}$ and $s$ satisfy the hypothesis \eqref{subcriticalhypothesisfield} of Theorem~\ref{thm:field}.
Given any subspace $W \subim V$, there exists a basis $S$ for $W$ over $\QQ$ which consists of elements of $\ZZ^d$. 
Define $H\subim \ZZ^d$ to be the subgroup generated by $S$ (over $\ZZ$).
Then $\Rank{H}=\Dim{W}$.  
Moreover, $\psi_j(W)$ equals the span over $\QQ$ of $\tilde \phi_j(H)$, and $\Dim{\psi_j(W)} = \Rank{\tilde \phi_j(H)}$.
The hypothesis $\Rank{H}\le \sum_{j=1}^m s_j \Rank{\tilde \phi_j(H)}$ is therefore equivalently written as $\Dim{W} \le \sum_{j=1}^m s_j \Dim{\psi_j(W)}$, which is \eqref{subcriticalhypothesisfield} for $W$.

Consider the $m$--tuple $f=\tup{f_j}$ corresponding to any inequality in \eqref{BL} for $\tup{\ZZ^d,\tup{\ZZ^{d_j}},\tup{\tilde \phi_j}}$ and $s$, and for each $j$ define $F_j$ by $F_j(y)=f_j(y)$ if $y\in\ZZ^{d_j}\subseteq\QQ^{d_j}$, and $F_j(y)=0$ otherwise.
%
% finite subset $E\subset V$, 
%
By Theorem~\ref{thm:field}, 
\[
\sum_{x\in \ZZ^d} \prod_{j=1}^m f_j(\tilde \phi_j(x)) 
= \sum_{x\in \QQ^d} \prod_{j=1}^m F_j(\psi_j(x)) 
\le \prod_{j=1}^m \norm{F_j}_{1/s_j}
= \prod_{j=1}^m \norm{f_j}_{1/s_j}.
\]
Thus the conclusion \eqref{BL} of Theorem~\ref{thm:2} is satisfied.
%
%Necessity of the hypothesis \eqref{subcriticalhypothesis2} follows by considering the special case of Theorem~\ref{thm:1}.
%
\end{proof}

Conversely, it is possible to derive Theorem~\ref{thm:field} for the case $\FF=\QQ$ from the special case of Theorem~\ref{thm:2} in which $G=\ZZ^d$ and $G_j = \ZZ^{d_j}$ by similar reasoning involving multiplication by large integers to clear denominators.

\begin{remark} \label{rmk:BCCT10-T}
Suppose we relax our requirement of an Abelian group HBL datum (Definition~\ref{def:groupdatum}) that the finitely generated Abelian group $G$ is torsion-free; let us call this an HBL datum with torsion.
The torsion subgroup $T(G)$ of $G$ is the (finite) set of all elements $x\in G$ for which there exists $0\ne n\in\ZZ$ such that $nx=0$.
%
% For any finitely generated Abelian group $G$, $T(G)$ is a finite subgroup. 
%By an HBL datum with torsion we will mean a tuple $(G,\set{G_j},\set{\phi_j})$ such that $G,G_j$ are finitely many finitely generated Abelian groups, and $\phi_j:G\to G_j$ are group homomorphisms. 
%
As mentioned in Remark~\ref{rmk:BCCT10}, it was shown in \cite[Theorem~2.4]{BCCT10} that for an HBL datum with torsion, the rank conditions \eqref{subcriticalhypothesis2} are necessary and sufficient for the existence of some finite constant $C$ such that \eqref{BL-C} holds.
%
% $|E|\le C\prod_j |\phi_j(E)|^{s_j}$ for all finite subsets of $G$.
A consequence of Theorem~\ref{thm:2} is a concrete upper bound for the constant $C$ in these inequalities.

\begin{theorem}
Consider $\tup{G,\tup{G_j},\tup{\phi_j}}$, an HBL datum with torsion, and $s\in[0,1]^m$.
If \eqref{subcriticalhypothesis2} holds, then \eqref{BL-C} holds with $C=|T(G)|$.
In particular,  
\begin{equation} \label{inequalitywithtorsion}
|E|\le |T(G)|\cdot \prod_{j=1}^m |\phi_j(E)|^{s_j} \qquad\text{for all nonempty finite sets } E \subseteq G.
\end{equation}
Conversely, if \eqref{inequalitywithtorsion} holds for $s\in[0,1]^m$, then $s$ satisfies \eqref{subcriticalhypothesis2}.
%$\Rank(H)\le \sum_j s_j\Rank(\phi_j(H))$ for every subgroup $H\le G$
%if and only if for every finite subset $E$ of $G$,
\end{theorem}
\begin{proof}
To prove \eqref{BL-C}, express $G$ as $\tilde G\oplus T(G)$ where $\tilde G \le G$ is torsion-free.
Thus arbitrary elements $x\in G$ are expressed as $x=(\tilde x,t)$  with $\tilde x\in\tilde G$ and $t\in T(G)$.
Then $\tup{\tilde G,\tup{G_j},\tup{\restr{\phi_j}{\tilde G}}}$ is an Abelian group HBL datum (with $\tilde G$ torsion-free) to which Theorem~\ref{thm:2} can be applied. 
Consider any $t\in T(G)$. 
Define $g_j\from G_j\to[0,\infty)$ by $g_j(x_j) = f_j(x_j+\phi_j(0,t))$.
Then $f_j(\phi_j(y,t)) = f_j(\phi_j(y,0)+\phi_j(0,t)) = g_j(\phi_j(y,0))$, so
\begin{equation*}
\sum_{y\in\tilde G} \prod_{j=1}^m f_j(\phi_j(y,t))
= \sum_{y\in\tilde G} \prod_{j=1}^m g_j(\phi_j(y,0))
\le \prod_{j=1}^m \norm{\restr{g_j}{\phi_j(\tilde G)}}_{1/s_j}
\le \prod_{j=1}^m \norm{f_j}_{1/s_j}.
\end{equation*}
The first inequality is an application of Theorem~\ref{thm:2}.
Summation with respect to $t\in T(G)$ gives the required bound.

To show necessity, we consider just the inequalities \eqref{inequalitywithtorsion} corresponding to the subsets $E \subseteq \tilde G$ and follow the proof of the converse of Theorem~\ref{thm:2}, except substituting $A|T(G)|$ for $A$.
\end{proof}
The factor $|T(G)|$ cannot be improved if the groups $G_j$ are torsion free, or more generally if $T(G)$ is contained in the intersection of the kernels of all the homomorphisms $\phi_j$; this is seen by considering $E=T(G)$. 
\end{remark}

\subsection{The polytope $\PP$} \label{sec:P}
The constraints \eqref{subcriticalhypothesis2} and \eqref{subcriticalhypothesisfield} are encoded by convex polytopes in $[0,1]^m$.
\begin{definition} \label{def:P}
For any Abelian group HBL datum $\tup{G,\tup{G_j},\tup{\phi_j}}$, we denote the set of all $s\in[0,1]^m$ which satisfy \eqref{subcriticalhypothesis2} by $\PP\tup{G,\tup{G_j},\tup{\phi_j}}$.
For any vector space HBL datum $\tup{V,\tup{V_j},\tup{\phi_j}}$, we denote the set of all $s\in[0,1]^m$ which satisfy \eqref{subcriticalhypothesisfield} by $\PP\tup{V,\tup{V_j},\tup{\phi_j}}$.
\end{definition}
$\PP=\PP\tup{V,\tup{V_j},\tup{\phi_j}}$ is the subset of $\RR^m$ defined by the inequalities $0\le s_j\le 1$ for all $j$, and $r \le \sum_{j=1}^m s_j r_j$ where $\tup{r,r_1,\ldots,r_m}$ is any element of $\set{1,2,\ldots,\Dim{V}} \times \prod_{j=1}^m \set{0,1,\ldots,\Dim{\phi_j(V)}}$ for which there exists a subspace $W \subim V$ which satisfies $\Dim{W} = r$ and $\Dim{\phi_j(W)}=r_j$ for all $j$.
Although infinitely many candidate subspaces $W$ must potentially be considered in any calculation of $\PP$, there are fewer than $(\Dim{V}+1)^{m+1}$ tuples $\tup{r,r_1,\ldots,r_m}$ which generate the inequalities defining $\PP$.
Thus $\PP$ is a convex polytope with finitely many extreme points, and is equal to the convex hull of the set of all of these extreme points.
This discussion applies equally to $\PP\tup{G,\tup{G_j},\tup{\phi_j}}$.  

In the course of proving Theorem~\ref{thm:2} above, we established the following result. 
\begin{lemma} \label{lem:PZequalsPQ}
Let $\tup{G,\tup{G_j},\tup{\phi_j}}$ be Abelian group HBL datum, let $\tup{\ZZ^d,\tup{\ZZ^{d_j}},\tup{\tilde \phi_j}}$ be the associated datum with torsion-free codomains, and let $\tup{\QQ^d,\tup{\QQ^{d_j}},\tup{\psi_j}}$ be the associated vector space HBL datum.
Then 
\[
\PP\dtup{G,\dtup{G_j},\dtup{\phi_j}} = \PP\dtup{\ZZ^d,\dtup{\ZZ^{d_j}},\dtup{\tilde \phi_j}} = \PP\dtup{\QQ^d,\dtup{\QQ^{d_j}},\dtup{\psi_j}}.
\] 
\end{lemma}

Now we prove Proposition~\ref{prop:maxexponent}, i.e., that there was no loss of generality in assuming each $s_j \le 1$.
\begin{proof}[Proof of Proposition~\ref{prop:maxexponent}]
Proposition~\ref{prop:maxexponent} concerns the case of Theorem~\ref{thm:1}, but in the course of its proof we will show that this result also applies to Theorems~\ref{thm:2} and~\ref{thm:field}.

We first show the result concerning \eqref{subcriticalhypothesis}, by considering instead the set of inequalities \eqref{subcriticalhypothesisfield}.
Suppose that a vector space HBL datum $\tup{V,\tup{V_j},\tup{\phi_j}}$ and $s\in[0,\infty)^m$ satisfy \eqref{subcriticalhypothesisfield}, and suppose that $s_k \ge 1$ for some $k$.
Define $t \in [0,\infty)^m$ by $t_j=s_j$ for $j \ne k$, and $t_k=1$.
Pick any subspace $W \subim V$; let $W' \ceq W \cap \Kernel{\phi_k}$ and let $U$ be a supplement of $W'$ in $W$, i.e., $W=U+W'$ and $\Dim{W}=\Dim{U}+\Dim{W'}$.
Since $\Dim{\phi_k(U)} = \Dim{U}$, 
\[ 0 \le \sum_{j \ne k} t_j \Dim{\phi_j(U)} = (t_k-1)\Dim{U} + \sum_{j \ne k} t_j \Dim{\phi_j(U)} = -\Dim{U} + \sum_{j =1}^m t_j \Dim{\phi_j(U)}; \]
since $s$ satisfies \eqref{subcriticalhypothesisfield} and $\Dim{\phi_k(W')}=0$, by \eqref{subcriticalhypothesisfield} applied to $W'$,
\[ \Dim{W'} \le \sum_{j=1}^m s_j \Dim{\phi_j(W')} = \sum_{j=1}^m t_j \Dim{\phi_j(W')}. \]
Combining these inequalities and noting that $\phi_j(U)$ is a supplement of $\phi_j(W')$ in $\phi_j(W)$ for each $j$, we conclude that $t$ also satisfies  \eqref{subcriticalhypothesisfield}.
Given an $m$--tuple $s$ with multiple components $s_k > 1$, we can consider each separately and apply the same reasoning to obtain an $m$--tuple $t\in[0,1]^m$ with $t_j = \min(s_j,1)$ for all $j$.
Our desired conclusion concerning \eqref{subcriticalhypothesis} follows from Lemma~\ref{lem:PZequalsPQ}, by considering the associated vector space HBL datum $\tup{\QQ^d,\tup{\QQ^{d_j}},\tup{\psi_j}}$ and noting that the lemma was established without assuming $s_j \le 1$.
(A similar conclusion can be obtained for \eqref{BL;sets}.)

Next, we show the result concerning \eqref{mainconclusion}. 
Consider the following more general situation.
Let $X,X_1,\ldots,X_m$ be sets and $\phi_j \from X \to X_j$ be functions for $j\in\set{1,\ldots,m}$. 
Let $s\in[0,\infty)^m$ with some $s_k \ge 1$, and suppose that $|E| \le \prod_{j=1}^m |\phi_j(E)|^{s_j}$ for any finite nonempty subset $E \subseteq X$. 
Fix one such set $E$.
For each $y \in \phi_k(E)$, let $E_y = \phi_k^{-1}(y) \cap E$, the preimage of $y$ under $\restr{\phi_k}{E}$; thus $|\phi_k(E_y)| = 1$.
By assumption, $|E_y| \le \prod_{j=1}^m |\phi_j(E_y)|^{s_j}$, so it follows that 
\[
|E_y| \le \prod_{j\ne k}|\phi_j(E_y)|^{s_j} \le \prod_{j\ne k} |\phi_j(E)|^{s_k} =  \prod_{j\ne k} |\phi_j(E)|^{t_k}.
\]
Since $E$ can be written as the union of disjoint sets $\bigcup_{y\in \phi_k(E)} E_y$, we obtain
\[ 
|E| = \sum_{y\in \phi_k(E)} |E_y| \le \sum_{y \in  \phi_k(E)} \prod_{j \ne k} |\phi_j(E)|^{t_k} = |\phi_k(E)| \cdot \prod_{j \ne k} |\phi_j(E)|^{t_k} =  \prod_{j =1}^m |\phi_j(E)|^{t_k}.
\]
Given an $m$--tuple $s$ satisfying with multiple components $s_k > 1$, we can consider each separately and apply the same reasoning to obtain an $m$--tuple $t\in[0,1]^m$ with $t_j = \min(s_j,1)$ for all $j$.
Our desired conclusion concerning \eqref{mainconclusion} follows by picking $\tup{X,\tup{X_j},\tup{\phi_j}} \ceq \tup{\ZZ^d,\tup{\ZZ^{d_j}},\tup{\phi_j}}$, and a similar conclusion can be obtained for \eqref{BL;sets} and \eqref{BLfield;sets}.

We claim that this result can be generalized to $\eqref{BL}$ and $\eqref{BLfield}$, based on a comment in \cite[Section~8]{BCCT10} that it generalizes in the weaker case of \eqref{BL-C}. 
\end{proof}

%Let $K=\set{j: s_j\ge 1}$. 
%%
%Let $E$ be any nonempty finite subset of $\ZZ^d$.
%%
%For each $y=\tup{y_k: k\in K}\in\prod_{k\in K} \ZZ^{d_k}$, where $\prod$ denotes the Cartesian product, let 
%%
%\begin{equation*}
%E_y=\dset{x\in E: \phi_k(x)=y_k \text{ for all } k\in K}.
%\end{equation*}
%%
%Then by \eqref{mainconclusion},
%%
%\begin{equation*}
%|E_y|\le \prod_{j\notin K} |\phi_j(E_y)|^{s_j} \prod_{k\in K}|\phi_k(E_y)|^{s_k}.
%\end{equation*}
%%
%By definition of $E_y$, $|\phi_k(E_y)|\in\set{0,1}$ for every $k\in K$. 
%%
%Thus
%%
%\begin{equation*}
%|E_y|\le \prod_{j\notin K} |\phi_j(E_y)|^{s_j} 
%\le \prod_{j\notin K} |\phi_j(E)|^{s_j} 
%= \prod_{j\notin K} |\phi_j(E)|^{t_j}. 
%\end{equation*}
%%
%Now $|E|=\sum_y |E_y|$, where the sum is taken over all $y\in\prod_{k\in K} \ZZ^{d_k}$ such that $E_y\ne\emptyset$. 
%%
%Summing only over those $y$, we have
%%
%\begin{equation*}
%|E| = \sum_y |E_y| \le\sum_y \prod_{j\notin K} |\phi_j(E)|^{t_j}
%= \prod_{j\notin K} |\phi_j(E)|^{t_j} \cdot \lt|\dset{y\in\prod_{k\in K}\ZZ^{d_k}: E_y\ne\emptyset}\rt|.
%\end{equation*}
%%
%Now 
%%
%\begin{equation*}
%\dset{y\in \prod_{k\in K}\ZZ^{d_k}: E_y\ne\emptyset} \subseteq \prod_{k\in K}\phi_k(E),
%\end{equation*}
%%
%so
%%
%\begin{equation*} \lt|\dset{y\in\prod_{k\in K}\ZZ^{d_k}: E_y\ne\emptyset}\rt|
%\le \prod_{k\in K} |\phi_k(E)|	
%= \prod_{k\in K} |\phi_k(E)|^{t_k}.  
%\end{equation*}
%%
%Thus $|E|\le \prod_{j\notin K}|\phi_j(E)|^{t_j}\prod_{k\in K}|\phi_k(E)|^{t_k}$, as required. 

To prove Theorem~\ref{thm:field}, we show in \sectn{\ref{subsect:interpolation}} that if \eqref{BLfield} holds at each extreme point $s$ of $\PP$, then it holds for all $s\in\PP$.
Then in \sectn{\ref{subsect:criticalsubspaces}}, we show that when $s$ is an extreme point of $\PP$, the hypothesis \eqref{subcriticalhypothesisfield} can be restated in a special form.
Finally in \sectn{\ref{sec:thm:field-proof}} we prove Theorem~\ref{thm:field} (with restated hypothesis) when $s$ is any extreme point of $\PP$, thus proving the theorem for all $s\in\PP$.

\subsection{Interpolation between extreme points of $\PP$} \label{subsect:interpolation}
A reference for measure theory is \cite{halmos1974measure}.
Let $\tup{X_j,\mathcal{A}_j,\mu_j}$ be measure spaces for $j\in\set{1,2,\ldots,m}$, where each $\mu_j$ is a nonnegative measure on the $\sigma$--algebra $\mathcal{A}_j$.
Let $S_j$ be the set of all simple functions $f_j \from X_j\to\CC$.
Thus $S_j$ is the set of all $f_j \from X_j\to\CC$ which can be expressed in the form $\sum_i c_i \indicator_{E_i}$ where $c_i\in\CC$, $E_i\in\mathcal{A}_j$, $\mu_j(E_i)<\infty$, and the sum extends over finitely many indices $i$.

% Let $T\from \prod_{j=1}^m S_j\to\CC$ be a multilinear map.
Let $T \from \prod_{j=1}^m \CC^{X_j} \to \CC$ be a multilinear map; i.e., for any $m$--tuple $f\in \prod_{j=1}^m \CC^{X_j}$ where $f_k=c_0f_{k,0}+c_1f_{k,1}$ for $c_0,c_1\in\CC$ and $f_{k,0},f_{k,1}\in\CC^{X_k}$,
% and simple functions $f_{k,0}$, $f_{k,1}$,
%
%\begin{multline*}
\[
T(f)
% T(f_1,\ldots,f_{k-1},c_0f_{k,0}+c_1 f_{k,1},f_{k+1},\ldots,f_m) \\
= c_0T\dtup{f_1,\ldots,f_{k-1},f_{k,0},f_{k+1},\ldots,f_m} + c_1T\dtup{f_1,\ldots,f_{k-1},f_{k,1},f_{k+1},\ldots,f_m}.
\]
%\end{multline*}
%

One multilinear extension of the Riesz-Th\"orin theorem states the following (see, e.g., \cite{bennettsharpley}).
\begin{proposition}[Multilinear Riesz--Th\"orin theorem]
\label{prop:rieszthorin}
Suppose that $p_0=\tup{p_{j,0}},p_1=\tup{p_{j,1}} \in [1,\infty]^m$. 
Suppose that there exist $A_0,A_1\in[0,\infty)$ such that
\begin{equation*}
\lt|T(f)\rt| \le A_0\prod_{j=1}^m \norm{f_j}_{p_{j,0}} \qquad\text{and}\qquad \lt|T(f)\rt| \le A_1\prod_{j=1}^m \norm{f_j}_{p_{j,1}} \qquad\text{for all } f \in \prod_{j=1}^m S_j.
\end{equation*}
For each $\theta\in(0,1)$ define exponents $p_{j,\theta}$ by
\begin{equation*}
\frac1{p_{j,\theta}} \ceq \frac{\theta}{p_{j,0}} + \frac{1-\theta}{p_{j,1}}.
\end{equation*}
Then for each $\theta\in(0,1)$,
\begin{equation*}
 \lt|T(f)\rt| \le A_0^\theta A_1^{1-\theta}\prod_{j=1}^m \norm{f_j}_{p_{j,\theta}} \qquad\text{for all } f \in \prod_{j=1}^m S_j.
 \end{equation*}
Here $\norm{f_j}_{p}=\norm{f_j}_{L^{p}(X_j,\mathcal{A}_j,\mu_j)}$.
% \TODO{Change back to $L^p$? Or define as $\int_{X_j} f_j d\mu_j$?}
\end{proposition}

In the context of Theorem~\ref{thm:field} with  vector space HBL datum $\tup{V,\tup{V_j},\tup{\phi_j}}$, we consider the multilinear map 
\[ T(f) \ceq \sum_{x\in V} \prod_{j=1}^m f_j(\phi_j(x)) \]
representing the left-hand side in \eqref{BLfield}.
\begin{lemma}
\label{lem:atextremepoints}
If \eqref{BLfield} holds for every extreme point of $\PP$, then it holds for every $s\in\PP$.
\end{lemma}
\begin{proof}
For any $\tilde f\in \prod_{j=1}^m S_j$, we define another $m$--tuple $f$ where for each $j$,  $f_j = |\tilde f_j|$ is a nonnegative simple function.
By hypothesis, the inequality in \eqref{BLfield} corresponding to $f$ holds at every extreme point $s$ of $\PP$, giving
\[
\lt| T(\tilde f)\rt| 
%\ceq \lt| \sum_{x\in V} \prod_{j=1}^m \tilde f_j(\phi_j(x))\rt| 
\le \sum_{x\in V} \prod_{j=1}^m \lt|\tilde f_j(\phi_j(x))\rt| 
= \sum_{x\in V} \prod_{j=1}^m f_j(\phi_j(x)) 
\le \prod_{j=1}^m \norm{f_j}_{1/s_j} 
= \prod_{j=1}^m \norm{\tilde f_j}_{1/s_j}.
\]
%Because $|T\tup{\tup{\tilde f_j}}| \le T\tup{\tup{|\tilde f_j|}}$ for $\tup{\tilde f_j}$ where $\tilde f_j$ is simple but  
%\[ 
%\lt| \sum_x \prod_j f_j(\phi_j(x))\rt| 
%\le \sum_x \prod_j \lt|f_j(\phi_j(x))\rt|, 
%\]
% inequality \eqref{BLfield} for arbitrary nonnegative simple functions is equivalent to the same inequality for arbitrary simple functions.
As a consequence of Proposition~\ref{prop:rieszthorin} (with constants $A_i=1$), and the fact that any $s\in\PP$ is a finite convex combination of the extreme points, this expression holds for any $s\in\PP$.
For any nonnegative function $F_j$ (e.g., in $\ell^{1/s_j}(V_j)$), there is an increasing sequence of nonnegative simple functions $f_j$ whose (pointwise) limit is $F_j$. 
Consider the $m$--tuple $F=\tup{F_j}$ corresponding to any inequality in \eqref{BLfield}, and consider a sequence of $m$--tuples $f$ which converge to $F$; then $\prod_{j=1}^m f_j$ also converges to $\prod_{j=1}^m F_j$. 
So by the monotone convergence theorem, the summations on both sides of the inequality converge as well.
\end{proof}

\subsection{Critical subspaces and extreme points} 
\label{subsect:criticalsubspaces}
Assume a fixed vector space HBL datum $\tup{V,\tup{V_j},\tup{\phi_j}}$, and let $\PP$ continue to denote the set of all $s\in[0,1]^m$ which satisfy \eqref{subcriticalhypothesisfield}.
\begin{definition}
Consider any $s\in[0,1]^m$.
A subspace $W \subim V$ satisfying $\Dim{W}=\sum_{j=1}^m s_j\Dim{\phi_j(W)}$ is said to be a critical subspace; one satisfying $\Dim{W}\le \sum_{j=1}^m s_j\Dim{\phi_j(W)}$ is said to be subcritical; and a subspace satisfying $\Dim{W} > \sum_{j=1}^m s_j\Dim{\phi_j(W)}$ is said to be supercritical.
$W$ is said to be strictly subcritical if $\Dim{W}<\sum_{j=1}^m s_j\Dim{\phi_j(W)}$.
\end{definition}
In this language, the conditions \eqref{subcriticalhypothesisfield} assert that that every subspace $W$ of $V$, including $\set{0}$ and $V$ itself, is subcritical; equivalently, there are no supercritical subspaces.  
When more than one $m$--tuple $s$ is under discussion, we sometimes say that $W$ is critical, subcritical, supercritical or strictly subcritical with respect to $s$.

The goal of \sectn{\ref{subsect:criticalsubspaces}} is to establish the following:
\begin{proposition} 
\label{prop:0or1}
Let $s$ be an extreme point of $\PP$.
Then some subspace $\set{0} \subpr W \subpr V$ is critical with respect to $s$, or $s\in\set{0,1}^m$.
\end{proposition}
\noindent Note that these two possibilities are not mutually exclusive.
% \noindent Two remarks: firstly, 
% Secondly, the assumption that $s\in\PP$ implies in particular that $V$ itself is subcritical with respect to $s$. 
%
%
\begin{lemma} 
\label{lemma:0}
If $s$ is an extreme point of $\PP$, and if $i$ is an index for which $s_i\notin\set{0,1}$, then $\Dim{\phi_i(V)}\ne 0$.
\end{lemma}
\begin{proof}
Suppose $\Dim{\phi_i(V)}=0$.
If $t\in[0,1]^m$ satisfies $t_j=s_j$ for all $j\ne i$, then $\sum_{j=1}^m t_j\Dim{\phi_j(W)} = \sum_{j=1}^m s_j \Dim{\phi_j(W)}$ for all subspaces $W \subim V$, so $t\in\PP$ as well.
If $s_i\notin\set{0,1}$, then this contradicts the assumption that $s$ is an extreme point of $\PP$. 
\end{proof}
\begin{lemma} 
\label{lemma:2}
Let $s$ be an extreme point of $\PP$.
Suppose that no subspace $\set{0} \subpr W \subim V$ is critical with respect to $s$. 
Then $s\in\set{0,1}^m$.
\end{lemma}
\begin{proof}
Suppose to the contrary that for some index $i$, $s_i\notin\set{0,1}$. 
If $t\in[0,1]^m$ satisfies $t_j=s_j$ for all $j\ne i$ and if $t_i$ is sufficiently close to $s_i$, then $t\in\PP$. 
This again contradicts the assumption that $s$ is an extreme point.
\end{proof}
\begin{lemma} 
\label{lemma:1}
Let $s$ be an extreme point of $\PP$.
Suppose that there exists no subspace $\set{0} \subpr W \subpr V$ which is critical with respect to $s$.
Then there exists at most one index $i$ for which $s_i\notin\set{0,1}$.
\end{lemma}
\begin{proof}
Suppose to the contrary that there were to exist distinct indices $k,l$ such that neither of $s_k,s_l$ belongs to $\set{0,1}$.
By Lemma~\ref{lemma:0}, both $\phi_k(V)$ and $\phi_l(V)$ have positive dimensions.
For $\eps\in\RR$ define $t$ by $t_j=s_j$ for all $j\notin\set{k,l}$,
\begin{equation*} 
t_k = s_k+\eps \Dim{\phi_l(V)} \text{ and } t_l = s_l-\eps \Dim{\phi_k(V)}.  
\end{equation*}
Whenever $|\eps|$ is sufficiently small, $t\in[0,1]^m$.
Moreover, $V$ remains subcritical with respect to $t$. 
If $|\eps|$ is sufficiently small, then every subspace $\set{0} \subpr W \subpr V$ remains strictly subcritical with respect to $t$, because the set of all parameters $\tup{\Dim{W},\Dim{\phi_1(W)},\ldots,\Dim{\phi_m(W)}}$ which arise, is finite. 
Thus $t\in\PP$ for all sufficiently small $|\eps|$. Therefore $s$ is not an extreme point of $\PP$.
\end{proof}
\begin{lemma} 
\label{lemma:3}
Let $s\in[0,1]^m$. Suppose that $V$ is critical with respect to $s$.
Suppose that there exists exactly one index $i\in \set{1,2,\ldots,m}$ for which $s_i\notin\set{0,1}$. 
Then $V$ has a subspace which is supercritical with respect to $s$.
\end{lemma}
\begin{proof}
By Lemma~\ref{lemma:0}, $\Dim{\phi_i(V)}>0$.
Let $K$ be the set of all indices $k$ for which $s_k=1$.
The hypothesis that $V$ is critical means that
\begin{equation*}
\Dim{V} = s_i\Dim{\phi_i(V)} + \sum_{k\in K} s_k\Dim{\phi_k(V)}. 
\end{equation*}
Since $s_i>0$ and $\Dim{\phi_i(V)}>0$,
\begin{equation*}
\sum_{k\in K}\Dim{\phi_k(V)} = \sum_{k\in K} s_k \Dim{\phi_k(V)} = \Dim{V}-s_i\Dim{\phi_i(V)} < \Dim{V}.
\end{equation*}
Consider the subspace $W \subim V$ defined by
\begin{equation*} 
W=\bigcap_{k\in K}\Kernel{\phi_k};
\end{equation*}
this intersection is interpreted to be $W=V$ if the index set $K$ is empty.
$W$ necessarily has positive dimension. 
Indeed, $W$ is the kernel of the map $\psi\from V\to\bigoplus_{k\in K}\phi_k(V)$, defined by $\psi(x)=\tup{\phi_k(x): k\in K}$, where $\bigoplus$ denotes the direct sum of vector spaces.
The image of $\psi$ is isomorphic to some subspace of $\bigoplus_{k\in K}\phi_k(V)$, a vector space whose dimension $\sum_{k\in K} \Dim{\phi_k(V)}$ is strictly less than $\Dim{V}$.
Therefore $\Kernel{\psi}=W$ has dimension greater than or equal to $\Dim{V}-\sum_{k\in K}\Dim{\phi_k(V)}>0$.
Since $\phi_k(W)=\set{0}$ for all $k\in K$,
\begin{align*} 
\sum_{j=1}^m s_j \Dim{\phi_j(W)} &= s_i\Dim{\phi_i(W)} + \sum_{k\in K} \Dim{\phi_k(W)} \\
&= s_i\Dim{\phi_i(W)} \\
&\le s_i\Dim{W}.
\end{align*}
Since $s_i<1$  and $\Dim{W}>0$, $s_i\Dim{\phi_i(W)}$ is strictly less than $\Dim{W}$, whence $W$ is supercritical. 
\end{proof}
\begin{proof}[Proof of Proposition~\ref{prop:0or1}]
Suppose that there exists no critical subspace $\set{0} \subpr W \subpr V$. 
By Lemma~\ref{lemma:2}, either $s_j\in\set{0,1}^m$ --- in which case the proof is complete --- or $V$ is critical.
By Lemma~\ref{lemma:1}, there can be at most one index $i$ for which $s_i\notin\set{0,1}$.
By Lemma~\ref{lemma:3}, for critical $V$, the existence of one single such index $i$ implies the presence of some supercritical subspace, contradicting the main hypothesis of Proposition~\ref{prop:0or1}.  
Thus again, $s\in\set{0,1}^m$. 
\end{proof}

\subsection{Factorization of HBL data}
\begin{notation}
\label{not:splittingspaces}
Suppose $V,V'$ are finite-dimensional vector spaces over a field $\FF$, and $\phi\from V\to V'$ is an $\FF$--linear map.
When considering a fixed subspace $W \subim V$, then $\restr{\phi}{W}\from W \to \phi(W)$ denotes the restriction of $\phi$ to $W$, also a $\FF$--linear map.
$V/W$ denotes the quotient of $V$ by $W$; elements of $V/W$ are cosets $x+W=\set{x+w: w\in W}\subseteq V$ where $x\in V$. 
Thus $x+W=x'+W$ if and only if $x-x'\in W$. 
$V/W$ is also a (finite-dimensional) vector space, under the definition $(x+W)+(x'+W)=(x+x')+W$; every subspace of $V/W$ can be written as $U/W$ for some $W \subim U \subim V$.

Similarly we can define $V'/\phi(W)$, the quotient of $V'$ by $\phi(W)$, and the quotient map $[\phi]\from V/W \ni x+W \mapsto \phi(x) + \phi(W) \in V'/\phi(W)$,
also a $\FF$--linear map.
%
%Then $[\phi](x+W)$ is by definition the coset $\phi(x) + \phi(W)$, and $[\phi]$ is also a $\FF$--linear map.
%
For any $U/W \subim V/W$, $[\phi](U/W)=\phi(U)/\phi(W)$.
\end{notation}
Let $\tup{V,\tup{V_j},\tup{\phi_j}}$ be a vector space HBL datum.
To any subspace $W\subim V$ can be associated two HBL data:
\[
\tup{W,\tup{\phi_j(W)},\tup{\restr{\phi_j}{W}}} \text{ and } \tup{V/W,\tup{V_j/\phi_j(W)},\tup{[\phi_j]}}
\]
\begin{lemma} \label{lemma:factorization1}
Given the vector space HBL datum $\tup{V,\tup{V_j},\tup{\phi_j}}$, for any subspace $W\subim V$, 
\begin{equation}
\PP\dtup{W,\dtup{V_j},\dtup{\restr{\phi_j}{W}}} 
\cap \PP\dtup{V/W,\dtup{V_j/\phi_j(W)},\dtup{[\phi_j]}}
\subseteq \PP\dtup{V,\dtup{V_j},\dtup{\phi_j}} 
\end{equation}
\end{lemma}
\begin{proof}
Consider any subspace $U\subim V$ and some $s\in[0,1]^m$ such that both $s\in\PP\tup{W,\tup{V_j},\tup{\restr{\phi_j}{W}}}$ and $s\in\PP\tup{V/W,\tup{V_j/\phi_j(W)},\tup{[\phi_j]}}$.
Then
\begin{align*}
\Dim{U} &= \Dim{(U+W)/W} + \Dim{U\cap W} \\
&\le \sum_{j=1}^m s_j \Dim{[\phi_j]((U+W)/W))}  + \sum_{j=1}^m s_j \Dim{\phi_j(U\cap W)} \\
&= \sum_{j=1}^m s_j \Dim{\phi_j(U+W)/\phi_j(W)}  + \sum_{j=1}^m s_j \Dim{\phi_j(U\cap W)} \\
&= \sum_{j=1}^m s_j \lt(\Dim{\phi_j(U+W)} - \Dim{\phi_j(W)} \rt) + \sum_{j=1}^m s_j \Dim{\phi_j(U\cap W)} \\
&= \sum_{j=1}^m s_j \lt(\Dim{\phi_j(U)+\phi_j(W)} + \Dim{\phi_j(U\cap W)} - \Dim{\phi_j(W)}\rt) \\
&\le \sum_{j=1}^m s_j \lt(\Dim{\phi_j(U)+\phi_j(W)} + \Dim{\phi_j(U)\cap \phi_j(W)} - \Dim{\phi_j(W)}\rt) \\
&=  \sum_{j=1}^m s_j \Dim{\phi_j(U)}.
\end{align*}
The last inequality is a consequence of the inclusions
%
%\\ & =\sum_j s_j \dim(\phi_j(U))
%+\sum_j s_j \Big(  \dim(\phi_j(U\cap W))
%- \dim(\phi_j(U)\cap \phi_j(W))\Big).
%\\ &\le
%\sum_j s_j \dim(\phi_j(U))
%\end{align*}
%
$\phi_j(U\cap W) \subseteq \phi_j(U)\cap \phi_j(W)$. 
The last equality is the relation $\Dim{A}+\Dim{B} = \Dim{A+B}+\Dim{A\cap B}$, which holds for any subspaces $A,B$ of a vector space.
Thus $U$ is subcritical.
\end{proof}
\begin{lemma} \label{lemma:factorization}
Given the vector space HBL datum $\tup{V,\tup{V_j},\tup{\phi_j}}$, let $s\in\PP\dtup{V,\dtup{V_j},\dtup{\phi_j}}$. 
Let $W\subim V$ be a subspace which 
% satisfies $0<\Dim{W}<\Dim{V}$, and which 
is critical with respect to $s$.
Then 
\begin{equation}
\PP\dtup{V,\dtup{V_j},\dtup{\phi_j}} 
= \PP\dtup{W,\dtup{V_j},\dtup{\restr{\phi_j}{W}}} 
\cap \PP\dtup{V/W,\dtup{V_j/\phi_j(W)},\dtup{[\phi_j]}}.
\end{equation}
\end{lemma}
\begin{proof}
With Lemma~\ref{lemma:factorization1} in hand, it remains to show that $\PP\tup{V,\tup{V_j},\tup{\phi_j}}$ is contained in the intersection of the other two polytopes.

Any subspace $U \subim W$ is also a subspace of $V$.
$U$ is subcritical with respect to $s$ when regarded as a subspace of $W$, if and only if $U$ is subcritical when regarded as a subspace of $V$.
So $s\in\PP\tup{W,\tup{V_j},\tup{\restr{\phi_j}{W}}}$.

Now consider any subspace of $U/W \subim W/V$; we have $W \subim U \subim V$ and $\Dim{U/W} = \Dim{U}-\Dim{W}$.
%Now let $U \subim V/W$.
%
%There exists a subspace $U^*\subim V$ which contains $W$, such that $U =\set{u+W: u\in U^*}$.
% 
Moreover, 
\[
\Dim{[\phi_j](U/W)} = \Dim{\phi_j(U)/\phi_j(W)} = \Dim{\phi_j(U)}-\Dim{\phi_j(W)}.
\]
Therefore since $\Dim{W} = \sum_{j=1}^m s_j\Dim{\phi_j(W)}$,
\begin{multline*}
\Dim{U/W} = \Dim{U}-\Dim{W}
\le \sum_{j=1}^m s_j \Dim{\phi_j(U^*)} - \sum_{j=1}^m s_j\Dim{\phi_j(W)}\\ 
= \sum_{j=1}^m s_j \lt(\Dim{\phi_j(U)} - \Dim{\phi_j(W)}\rt)
= \sum_{j=1}^m s_j \Dim{[\phi_j](U/W)}
\end{multline*}
by the subcriticality of $U$, which holds because $s\in\PP\tup{V,\tup{V_j},\tup{\phi_j}}$. 
Thus any $U/W \subim V/W$ is subcritical with respect to $s$, so $s\in \PP\tup{V/W,\tup{V_j/\phi_j(W)},\tup{[\phi_j]}}$ as well.
\end{proof}

\subsection{Proof of Theorem~\ref{thm:field}} \label{sec:thm:field-proof}
Recall we are given the vector space HBL datum $\tup{V,\tup{V_j},\tup{\phi_j}}$; we prove Theorem~\ref{thm:field} by induction on the dimension of the ambient vector space $V$.
If $\Dim{V}=0$ then $V$ has a single element, and the result \eqref{BLfield} is trivial.

To establish the inductive step, consider any extreme point $s$ of $\PP$.
According to Proposition~\ref{prop:0or1}, there are two cases which must be analyzed.
We begin with the case in which there exists a critical subspace $\set{0} \subpr W \subpr V$, which we prove in the following lemma. 
We assume that Theorem~\ref{thm:field} holds for all HBL data for which the ambient vector space has strictly smaller dimension than is the case for the given datum.
\begin{lemma} \label{lemma:criticalW}
Let $\tup{V,\tup{V_j},\tup{\phi_j}}$ be a vector space HBL datum, and let $s\in\PP(V,\tup{V_j},\tup{\phi_j})$.
Suppose that subspace $\set{0} \subpr W \subpr V$ is critical with respect to $s$.
Then \eqref{BLfield} holds for this $s$.
%Then for all nonnegative functions $\tup{f_j}$ in $\ell^{1/s_j}(V_j)$,
%\[ 
%\sum_{x\in V} \prod_{j=1}^m f_j(\phi_j(x)) \le \prod_{j=1}^m \norm{f_j}_{1/s_j}.  
%\]
\end{lemma}
\begin{proof}
Consider any inequality in \eqref{BLfield}.
We may assume that none of the exponents equal zero. 
For if $s_k=0$, then $f_k(\phi_k(x)) \le \norm{f_k}_{1/s_k}$ for all $x$, and therefore 
\[\sum_{x\in V} \prod_{j=1}^m f_j(\phi_j(x)) \le \norm{f_k}_{1/s_k} \cdot \sum_{x\in V} \prod_{j\ne k} f_j(\phi_j(x)).\] 
If $\norm{f_k}_{1/s_k}=0$, then \eqref{BLfield} holds with both sides $0$. 
Otherwise we divide by $\norm{f_k}_{1/s_k}$ to conclude that $s\in\PP\tup{V,\tup{V_j},\tup{\phi_j}}$ if and only if $\tup{s_j}_{j \ne k}$ belongs to the polytope associated to the HBL datum $\tup{V,\tup{V_j}_{j\ne k},\tup{\phi_j}_{j\ne k}}$.
Thus the index $k$ can be eliminated. 
This reduction can be repeated to remove all indices which equal zero.

%As mentioned above, every subspace $U \subim W$ is also a subspace of $V$; $U$ is subcritical as a subspace of $V$ with respect to $\tup{\phi_j}$ if and only if $U$ is subcritical as a subspace of $W$ with respect to $\tup{\restr{\phi_j}{W}}$.
%
Let $W_j \ceq \phi_j(W)$.
By Lemma~\ref{lemma:factorization}, $s\in\PP\tup{W,\tup{W_j},\tup{\restr{\phi_j}{W}}}$.
Therefore by the inductive hypothesis, one of the inequalities in \eqref{BLfield} is
\begin{equation} \label{Winequality}
\sum_{x\in W} \prod_{j=1}^m f_j(\phi_j(x))\le \prod_{j=1}^m \norm{\restr{f_j}{W_j}}_{1/s_j}.
\end{equation}

%\TODO{unify coset notation}
% Let $[x]$ be the unique coset containing $x$; we use this notation for elements $[x]=x+W\in V/W$ and $[x]=x+W_j\in V_j/W_j$.
%
%Define $\psi_j\from V/W\to V_j/W_j$ by $\psi_j([x])=[\phi_j(x)]$, that is, $\psi_j = [\phi_j]$ in the notation used above. 

Define $F_j \from V_j/W_j\to[0,\infty)$ to be the function 
\begin{equation*}
F_j(x+W_j) = \Big(\sum_{y\in W_j} f_j(y+x)^{1/s_j}\Big)^{s_j}.
\end{equation*}
This quantity is a function of the coset $x+W_j$ alone, rather than of $x$ itself, because for any $z\in W_j$,
\begin{equation*}
\sum_{y\in W_j} f_j(y+(x+z))^{1/s_j} =\sum_{y\in W_j} f_j(y+x)^{1/s_j}
\end{equation*}
by virtue of the substitution $y+z\mapsto y$.
Moreover,
\begin{equation} \label{Fjnorm} 
\norm{F_j}_{1/s_j} = \norm{f_j}_{1/s_j}.
\end{equation}
To prove this, choose one element $x\in V_j$ for each coset $x+W_j\in V_j/W_j$.
Denoting by $X$ the set of all these representatives, 
\begin{equation*}
\norm{F_j}_{1/s_j}^{1/s_j}  
=\sum_{x\in X} \sum_{y\in W_j} f_j(y+x)^{1/s_j} 
=\sum_{z\in V_j}  f_j(z)^{1/s_j}
\end{equation*}
because the map $X\times W_j\ni (x,y)\mapsto x+y\in V_j$ is a bijection.

The inductive bound \eqref{Winequality} can be equivalently written in the more general form
\begin{equation} \label{Winequality2}
\sum_{x\in W} \prod_{j=1}^m f_j(\phi_j(x+y))\le \prod_{j=1}^m F_j([\phi_j](y+W))
\end{equation}
for any $y\in V$, by applying \eqref{Winequality} to $\tup{\hat{f}_j}$ where $\hat{f}_j(z) = f_j(z + \phi_j(y))$.

Denote by $Y\subseteq V$ a set of representatives of the cosets $y+W\in V/W$, and identify $V/W$ with $Y$. 
Then
\begin{equation*} 
\sum_{x\in V} \prod_{j=1}^m f_j(\phi_j(x))
= \sum_{y\in Y} \sum_{x\in W} \prod_{j=1}^m f_j(\phi_j(y+x))
\le \sum_{y+W\in V/W} \prod_{j=1}^m F_j([\phi_j](y+W)) 
\end{equation*}
by \eqref{Winequality2}.

By \eqref{Fjnorm}, it suffices to show that
\begin{equation} \label{quotientBL}
\sum_{y+W\in V/W} \prod_{j} F_j([\phi_j](y+W))\le \prod_j\norm{F_j}_{1/s_j}  \qquad \text{for all functions } 0 \le F_j \in \ell^{1/s_j}(V_j/W_j).
\end{equation}
%
%for all nonnegative functions $\tup{F_j}$.
%
This is 
% an inequality 
a set of inequalities of exactly the form \eqref{BLfield}, with $\tup{V,\tup{V_j},\tup{\phi_j}}$ replaced by $\tup{V/W,\tup{V_j/W_j},\tup{[\phi_j]}}$.
By Lemma~\ref{lemma:factorization}, $s\in\PP\tup{V/W,\tup{V_j/W_j},\tup{[\phi_j]}}$, and since $\Dim{V/W}<\Dim{V}$, we conclude directly from the inductive hypothesis that \eqref{quotientBL} holds, concluding the proof of Lemma~\ref{lemma:criticalW}.
\end{proof}
%
%%
%Since $\dim(V/W) = \dim(V)-\dim(W)<\dim(V)$, \eqref{quotientBL} follows from the inductive hypothesis, provided that every subspace of $V/W$ is subcritical with respect to the  linear transformation $\psi_j$. 
%To verify this condition, consider any subspace $\mathcal{V} \subim V/W$. Then 
%$\mathcal{V}^*=\set{x\in V: [x]\in \mathcal{V}}$ is a subspace of $V$, and 
%%
%\begin{equation*}
%\dim(\mathcal{V}^*) = \dim(W)+\dim(\mathcal{V}).
%\end{equation*}
%%
%Likewise, 
%%
%\begin{equation*}
%\dim(\phi_j(\mathcal{V}^*)) = \dim(W_j) + \dim(\psi_j(\mathcal{V})).
%\end{equation*}
%%
%Therefore
%%
%\begin{align*}
%\dim(\mathcal{V}) &-\sum_j s_j\dim(\psi_j(\mathcal{V})) \\
%&= \dim(\mathcal{V}^*)-\dim(W) - \sum_j s_j \dim(\phi_j(\mathcal{V}^*)) + \sum_j s_j \dim(W_j) \\
%&= \Big( \dim(\mathcal{V}^*) - \sum_j s_j \dim(\phi_j(\mathcal{V}^*)) \Big) + \Big(\sum_j s_j \dim(W_j)-\dim(W)\Big) \\
%&=  \dim(\mathcal{V}^*) - \sum_j s_j \dim(\phi_j(\mathcal{V}^*)) \\
%&\le 0
%\end{align*}
%%
%because $W$ is critical and $\mathcal{V}^*$ is subcritical.
%%
%Therefore $\mathcal{V}$ is indeed subcritical, as required. 

%\TODO{Could also split off the $s_i = 0$ following above...}
According to Proposition~\ref{prop:0or1}, in order to complete the proof of Theorem~\ref{thm:field}, it remains only to analyze the case where the extreme point $s\in\set{0,1}^m$. 
Let $K=\set{k: s_k=1}$. 
Consider $W=\bigcap_{k\in K} \Kernel{\phi_k}$. 
Since $W$ is subcritical by hypothesis,
\begin{equation*}
\Dim{W}\le \sum_{j=1}^m s_j\Dim{\phi_j(W)} = \sum_{k\in K}\Dim{\phi_k(W)}=0 
\end{equation*}
so $\Dim{W}=0$, that is, $W=\set{0}$.
Therefore the map $x\mapsto\tup{\phi_k(x)}_{k\in K}$ from $V$ to the Cartesian product $\prod_{k\in K} V_k$ is injective. 

For any $x\in V$, 
\begin{equation*}
\prod_{j=1}^m f_j(\phi_j(x))
\le \prod_{k\in K} f_k(\phi_k(x))\prod_{i\notin K} \norm{f_i}_\infty
= \prod_{k\in K} f_k(\phi_k(x))\prod_{i\notin K} \norm{f_i}_{1/s_i}
\end{equation*}
since $s_i=0$ for all $i\notin K$.
Thus it suffices to prove that
\begin{equation*}
\sum_{x\in V}\prod_{k\in K} f_k(\phi_k(x)) \le \prod_{k\in K} \norm{f_k}_1.
\end{equation*}
This is a special case of the following result.
\begin{lemma} \label{lemma:exponentsallone}
Let $V$ be any finite-dimensional vector space over $\FF$.
Let $K$ be a finite index set, and for each $k\in K$, let $\phi_k$ be an $\FF$--linear map from $V$ to a finite-dimensional vector space $V_k$.
If $\bigcap_{k\in K}\Kernel{\phi_k}=\set{0}$ then for all functions $f_k\from V_k\to[0,\infty)$, %\TODO{in $\ell^{1}(V_k)$?},
\begin{equation*}
\sum_{x\in V} \prod_{k\in K} f_k(\phi_k(x))\le \prod_{k\in K}\norm{f_k}_1.
\end{equation*}
\end{lemma}
\begin{proof}
Define $\Phi\from V\to\prod_{k\in K} V_k$ by $\Phi(x) = \tup{\phi_k(x)}_{k\in K}$.
The hypothesis $\bigcap_{k\in K}\Kernel{\phi_k}=\set{0}$ is equivalent to $\Phi$ being injective.
The product $\prod_{k\in K} \norm{f_k}_1$ can be expanded as the sum of products
\begin{equation*}
\sum_{y} \prod_{k\in K} f_k(y_k)
\end{equation*}
where the sum is taken over all $y=\tup{y_k}_{k \in K}$ belonging to the Cartesian product $\prod_{k\in K}V_k$.
The quantity of interest, 
\begin{equation*}
\sum_{x\in V} \prod_{k\in K} f_k(\phi_k(x)),
\end{equation*}
is likewise a sum of such products.
Each term of the latter sum appears as a term of the former sum, and by virtue of the injectivity of $\Phi$, appears only once.
Since all summands are nonnegative, the former sum is greater than or equal to the latter. 
Therefore
\begin{equation*}
\prod_{k\in K}\norm{f_k}_1 
= \sum_{y} \prod_{k\in K} f_k(y_k) 
\ge \sum_{x\in V} \prod_{k\in K} f_k(\phi_k(x)).
\end{equation*}
\end{proof}
\noindent Having shown sufficiency for extreme points of $\PP$, we apply Lemma~\ref{lem:atextremepoints} to conclude sufficiency for all $s\in\PP$.

As mentioned above, necessity in the case $\FF=\QQ$ can be deduced from necessity in Theorem~\ref{thm:1}, by clearing denominators.
First, we identify $V$ and $V_j$ with $\QQ^d$ and $\QQ^{d_j}$ and let $E$ be any nonempty finite subset of $\QQ^d$.
Let $\hat \phi_j \from \QQ^d \to \QQ^{d_j}$ be the linear map represented by the matrix of $\phi_j$ multiplied by the lowest common denominator of its entries, i.e., an integer matrix.
Likewise, let $\hat E$ be the set obtained from $E$ by multiplying each point by the lowest common denominator of the coordinates of all points in $E$.
Then by linearity, \[ |\hat E| = |E| \le \prod_{j=1}^m |\phi_j(E)|^{s_j} = \prod_{j=1}^m |\hat \phi_j(\hat E)|^{s_j}. \]
Recognizing $\tup{\ZZ^d,\tup{\ZZ^{d_j}},\tup{\restr{\hat \phi_j}{\ZZ^d}}}$ as an Abelian group HBL datum, we conclude \eqref{subcriticalhypothesis} for this datum from the converse of Theorem~\ref{thm:1}.
According to Lemma~\ref{lem:PZequalsPQ}, \eqref{subcriticalhypothesisfield} holds for the vector space HBL datum $\tup{\QQ^d,\tup{\QQ^{d_j}},\tup{\hat \phi_j}}$; our conclusion follows since $\Dim{\hat \phi_j (W)} = \Dim{\phi_j(W)}$ for any $W \subim \QQ^d$.

It remains to treat the case of a finite field $\FF$. 
Whereas the above reasoning required only the validity of \eqref{BLfield;sets} in the weakened form $|E|\le C\prod_{j=1}^m|\phi_j(E)|^{s_j}$ for some constant $C<\infty$ independent of $E$ (see proof of necessity for Theorem~\ref{thm:2}), now the assumption that this holds with $C=1$ becomes essential.
Let $W$ be any subspace of $V$.
Since $|\FF|<\infty$ and $W$ has finite dimension over $\FF$, $W$ is a finite set and the hypothesis \eqref{BLfield;sets} can be applied with $E=W$.
Therefore $|W|\le\prod_{j=1}^m|\phi_j(W)|^{s_j}$.
This is equivalent to
\[ 
|\FF|^{\Dim{W}} \le \prod_{j=1}^m |\FF|^{s_j\Dim{\phi_j(W)}}, 
\]
so since $|\FF|\ge 2$, taking base--$|\FF|$ logarithms of both sides, we obtain $\Dim{W} \le \sum_{j=1}^m s_j\Dim{\phi_j(W)}$, as was to be shown.
\qed
\section{Communication lower bounds from Theorem~\ref{thm:1}} \label{sec:lb}

In this section we introduce a concrete execution model for programs running on a sequential machine with a two-level memory hierarchy.
With slight modification, our model also applies to parallel executions; we give details below in order to extend our sequential lower bounds to the parallel case in \sectn{\ref{sec:gen-machine}}.
We assume the memory hierarchy is program-managed, so all data movement between slow and fast memory is seen as explicit instructions, and computation can only be performed on values in fast memory.  
We combine this with the geometric model from \sectn{\ref{sec:model}} and the upper bound from Theorem~\ref{thm:1} to get a communication lower bound of the form $\Omega(\niters/M^{\shbl-1})$, where $\niters$ is the number of inner loop iterations, represented by the finite set $\iters\subseteq\ZZ^d$. 
Then we discuss how the lower bounds extend to programs on more complicated machines, like heterogeneous parallel systems.

%% Three levels of abstraction: 
In addition to the concrete execution model and the geometric model, we use pseudocode in our examples.   
At a high level, these three different algorithm representations are related as follows: 
\begin{itemize}

\item A concrete execution is a sequence of instructions executed by the machine, according to the model detailed in \sectn{\ref{sec:concrete}}.
The concrete execution, unlike either the geometric model or pseudocode, contains explicit data movement operations between slow and fast memory.
\item The geometric model (Definition~\ref{def:geometricmodel}), is the most abstract, and is the foundation of the bounds in \sectn{\ref{sec:hbl}}. 
Each instance  \eqref{eqn_AlgModel} of the geometric model corresponds to a set of concrete executions, as detailed in \sectn{\ref{sec:geometric-to-concrete}}.
\item A pseudocode representation, like our examples with (nested) for-loops, identifies an instance \eqref{eqn_AlgModel} of the geometric model.
Since the bounds in \sectn{\ref{sec:hbl}} only depend on $\tup{\ZZ^d,\tup{\ZZ^{d_j}},\tup{\phi_j}}$, one may vary the iteration space $\iters \subseteq \ZZ^d$, the order it is traversed, and the statements in the inner loop body (provided all $m$ arrays are still accessed each iteration), to obtain a different instance \eqref{eqn_AlgModel} of the geometric model with the same bound. 
So, when we prove a bound for a program given as pseudocode, we are in fact proving a bound for a larger class of programs.
\end{itemize}

 The rest of this section is organized as follows. 
 \Sectn{\ref{sec:concrete}} describes the concrete execution model mentioned above, and \sectn{\ref{sec:geometric-to-concrete}} relates the concrete execution model to the geometric model. 
 \Sectn{\ref{sec:lb_deriv}} states and proves the main communication lower bound result of this paper, Theorem~\ref{Thm4.1}. 
 \Sectn{\ref{sec:lb_examples}} presents a number of examples showing why the assumptions of Theorem~\ref{Thm4.1} are in fact necessary to obtain a lower bound. 
 \Sectn{\ref{sec_LoopSplitting}} looks at one of these assumptions in more detail (``no loop splitting''), and shows that loop splitting can only improve (reduce) the lower bound.
Finally, \sectn{\ref{sec:gen-machine}} discusses generalizations of the lower bound result to other machine models.

\subsection{Concrete execution model}
\label{sec:concrete}
The hypothetical machine in our execution model has a two-level memory hierarchy: a slow memory of unbounded capacity and a fast memory that can store $M$ words (all data in our model have one-word width).
Data movement between slow and fast memory is explicitly managed by software instructions (unlike a hardware cache), and data is copied\footnote{We will use the word `copied' but our analysis does not require that a copy remains, e.g., exclusive caches.} at a one-word granularity (we will discuss spatial locality in Part~2 of this work).
Every storage location in slow and fast memory (including the array elements $A_j(\phi_j(\indx))$) has a unique memory address, called a \emph{variable}; since the slow and fast memory address spaces (variable sets) are disjoint, we will distinguish between slow memory variables and fast memory variables.
When a fast memory variable represents a copy of a slow memory variable $v$, we refer to $v$ as a \emph{cached slow memory variable}; in this case, we assume we can always identify the corresponding fast memory variable given $v$, even if the copy is relocated to another fast memory location. 
%
%Note that we will sometimes abuse notation and use a variable to represent the value stored at that memory address; this is always clear from context.

%Assumptions about $A$, index expressions $\phi$, and indices $\indx$ are the same as in the geometric model.  Given a program with a $d$-deep loop nest, we require $\indx \in \iters \subseteq \ZZ^d$ and $\phi: \ZZ^d -> \ZZ^{d'}$ is a linear (integer) function, and that $A \from \ZZ^{d'} \to \set{\text{variables}}$ is an injection for $\phi_j(\ZZ^d)$.  The requirements on $\phi$ are fundamental to invoking Theorem~\ref{thm:1}.   
%Given these building blocks, an {\em sequential execution} is a sequence of {\em statements} of the following form:

We define a \emph{sequential execution} $E=\tup{e_1,e_2,\ldots,e_n}$ as a sequence of \emph{statements} $e_i$ of the following types:
\begin{description}
\item[Read] $read(v)$: allocates a location in fast memory and copies variable $v$ from slow to fast memory. 
\item[Write] $write(v)$: copies variable $v$ from fast to slow memory and deallocates the location in fast memory. 
\item[Compute] $compute(\set{v_1,\ldots,v_n})$ is a statement accessing variables $v_1,\ldots,v_n$.
\item[Allocate] $allocate(v)$ introduces variable $v$ in fast memory.
\item[Free] $free(v)$ removes variable $v$ from fast memory.
\end{description}
A sequential execution defines a total order on the statements, and thereby a natural notion of when one statement succeeds or precedes another. 
We say that a Read or Allocate statement and a subsequent Write or Free statement are \emph{paired} if the same variable $v$ appears as an operand in both and there are no intervening Reads, Allocates, Writes, or Frees of $v$.
A sequential execution is considered to be \emph{well formed} if and only if 
\begin{itemize}
\item operands to Read (resp., Write) statements are uncached (resp., cached) slow memory variables,
\item operands to Allocate statements are uncached slow memory variables or fast memory variables\footnote{If a fast memory variable $v$ in an $allocate(v)$ statement already stores a cached slow memory variable, then we assume the system will first relocate the cached variable to an available location within fast memory)},
\item operands to Free statements are cached slow memory variables or fast memory variables,
\item every Read, Allocate, Write, and Free statement is paired, and
\item every Compute statement involving variable $v$ interposes between paired Read/Allocate and Write/{\allowbreak}Free statements of $v$, i.e., each operand $v$ in a Compute statements \emph{resides in fast memory} before and after the statement.
\end{itemize}
Essentially, fast memory variables must be allocated and deallocated, either implicitly (Read/Write) or explicitly (Allocate/Free), while slow memory variables cannot be allocated/deallocated.
Given the finite capacity of fast memory, we need an additional assumption to ensure the memory operations are well-defined.
Given a well-formed sequential execution $E = \tup{e_1,\ldots,e_n}$, we define $footprint_i(E)$ to be the fast memory usage after executing statement $e_i$ in the program, i.e.,
\[
footprint_i(E) = \begin{cases} 
0 & i = 0 \text{ or } i = n
\\ footprint_{i-1}(E) + 1 & \text{if $e_i$ is a Read or Allocate statement,}
\\ footprint_{i-1}(E) & \text{if $e_i$ is a Compute statement, and}
\\ footprint_{i-1}(E) - 1 & \text{if $e_i$ is a Write or Free statement.}
\end{cases}
\]
Then $E$ is said to be \emph{$M$--fit} (for fast memory size $M$) if $\max_i footprint_i(E) \le M$.

 It is of practical interest to permit variables to reside in fast memory before and after the execution, e.g., to handle parallel code as described in the next paragraph; however, this violates our notion of well-formedness.
Rather than redefine well-formedness to account for this possibility, we take a simpler approach: given an execution that is well formed except for the $I$ (`input') and $O$ (`output') variables which reside in fast memory before and after the execution, we insert up to $M$ Reads and $M$ Writes at the beginning and end of the execution so that all memory statements are paired, and then later reduce the lower bound (on Reads/Writes) by $I+O \le 2M$.

%
% ?? Why can't we include them inside Compute statements?
%We will not include branches, other control statements, or computations that involve only scalars in our executions, since do not play a role in the analysis.
%
% ?? This assumption may be violated by our parallel model...
% We assume that all input variables start in slow memory, all output variables end in slow memory,
%
% ?? Must Allocations only be used for "temporaries"?

% Leave arrays in fast memory after store?  Allocate and free by blocks?
% \comment{do we need compute and memory subsequences?}
% Given a sequential execution $E$, we will consider the following subsequences: $E|compute$ is the subsequence of all computation statements; and $E|memory$ is the subsequence of all memory (read/write/alloc/free) statements.  

Although we will establish our bounds first for a sequential execution, they also apply to parallel executions as explained in \sectn{\ref{sec:gen-machine}}.
We define a \emph{parallel execution} as a set of $P$ sequential executions, $\set{E_1, E_2,\ldots,E_P}$.
In our parallel model, the global (`slow') memory for each processor is a subset of the union of the local (`fast') memories of the other processors.
That is, for a given processor, each of its slow memory variables is really a fast memory variable for some other processor, and each of its fast memory variables is a slow memory variable for every other processor, unless it corresponds to a cached slow memory variable, in which case it is invisible to the other processors.
(We could remove this last assumption by extending our model to distinguish between copies of a slow memory variable.)
A parallel execution is well formed or (additionally) $\set{M_1,\ldots,M_P}$--fit if each of its serial executions $E_i$ is well formed or (additionally) $M_i$--fit. 
Since well-formedness assumes no variables begin and end execution in fast/local memory, it seems impossible for there to be any nontrivial well-formed parallel execution. 
As mentioned above, we can always allow for a sequential execution with this property by inserting up to $I+O \le 2M$ Reads/Writes, and later reducing the lower bound by this amount; we insert Reads/Writes in this manner to each sequential execution $E_i$ in a parallel execution.

\subsection{Relation to the geometric model}
\label{sec:geometric-to-concrete}

%% Model algorithm
Recall from \sectn{\ref{sec:model}} our geometric model \eqref{eqn_AlgModel}: 
\begin{equation*}
\begin{split}
&\text{for all } \indx \in \iters \subseteq \ZZ^d, \text{ in some order}, \\
&\qquad\text{inner\_loop}(\indx, (A_1,\ldots,A_m),(\phi_1,\ldots,\phi_m))
\end{split}
\end{equation*}
The subroutine $\text{inner\_loop}$ represents a given `computation' involving arrays $A_1,\ldots,A_m$ referenced by corresponding subscripts $\phi_1(\indx),\ldots,\phi_m(\indx)$; each $A_j \from \phi_j(\ZZ^d) \to \set{\text{variables}}$ is an injection and each $\phi_j \from \ZZ^d \to \ZZ^{d_j}$ is a $\ZZ$--affine map.  
Each \emph{array variable} $A_1(\phi_1(\indx)),\ldots,A_m(\phi_m(\indx))$ is accessed in each $\text{inner\_loop}(\indx)$, as an input, output, or both, perhaps depending on the iteration $\indx$. 

Given an execution, we assume we can discern the expression $A_j(\phi_j(\indx))$ from any variable which represents an array variable in the geometric model.
The execution may contain additional variables that act as \emph{surrogates} (or copies) of the variables specified in the program text.  
As an extreme example, an execution could use an array $A'_j$ as a surrogate for the array $A_j$ in the computations, and then later set $A_j$ to $A'_j$. 
In such examples, one can always associate each surrogate variable with the `master' copy, and there is no loss of generality in our analysis to assume all variables are in fact the master copies. 

We say a \emph{legal sequential execution} of an instance of the geometric model is a sequential execution whose subsequence of Compute statements can be partitioned into contiguous chunks in one-to-one correspondence with $\iters$, and furthermore all $m$ array variables $A_j(\phi_j(\indx))$ appear as operands in the chunk corresponding to $\indx$.
Given a possibly overlapping partition $\iters=\bigcup_{i=1}^P \iters_i$, a \emph{legal parallel execution} is a parallel execution $\set{E_1,\ldots,E_p}$ where each sequential execution $E_i$ is legal with respect to loop iterations $\indx \in \iters_i$. 
Legality restricts the set of possible concrete executions we consider (for a given instance of the geometric model), and in general is a necessary requirement for the lower bound to hold for all concrete executions.
For example, transforming the classical algorithm for matrix multiplication into Strassen's algorithm (which can move asymptotically less data) is illegal, since it exploits the distributive property to interleave computations, and any resulting execution cannot be partitioned contiguously according to the original iteration space $\iters$.
%
% It seems legality prevents surrogates
% More generally, legality prevents a compiler from introducing additional variables to store/reuse intermediate computations in subsequent iterations.
%
% (It is, however, legal to introduce such variables within an iteration, provided they are not accessed in subsequent iterations.)
%
As another example, legality prevents \emph{loop splitting}, an optimization which can invalidate the lower bound as discussed in \sectn{\ref{sec_LoopSplitting}}.

%We are concerned with the variables $v$ appearing in the concrete execution corresponding to the subscripted arrays $A_1,\ldots,A_m$ in the geometric model, i.e., \emph{array variables}.
%%
%We recall from \sectn{\ref{sec:model}} that $A_j$ need not be an array, but can be any ``container'' type that holds other variables, e.g., a sparse matrix, tree, or graph, as long as it maps each distinct subscript $\phi_j(\indx)$ to a unique variable. 
%%
%In order to count the Reads/Writes of array variables between slow and fast memory, we need to be able to recognize them in a concrete execution.
%

%
% Kathy's comment about allowing `variable renaming': ``\emph{It is needed to allow  for the .5D algorithms and others.  We need to allow transformations on iters too, but need to think about the best way to do that.}''}
%
% Each {\em inner loop chunk} $E'$ has the form $\langle A_1(\phi_1(\indx)),...,A_n(\phi_n(\indx)) \rangle$ and $accesses(E') superseteq inner\_loop$.   However, we will also allow variable renaming, so if a geometric model has a legal execution $E$, we will also allow as legal an execution that consistently renames variables between a given load and store pair.
% 
% Discussion on branches?

We note that there are no assumptions about preserving dependencies in the original program or necessarily computing the correct answer. 
Restricting the set of executions further to ones that are correct may make the lower bound unattainable, but does not invalidate the bound.

%There are infinitely many ways of representing a variable by a set of other variables; for example, in matrix multiplication where the inner loop body is $C(i_1,i_2) = C(i_1,i_2) + A(i_1,i_3)\cdot B(i_3,i_2)$, it is common to use a separate variable to accumulate (part of) the sum of the $A(i_1,i_3)\cdot B(i_3,i_2)$, and only later to add this quantity to $C(i_1,i_2)$, which may itself be accessed just once. 
%%
%One can imagine accumulating several partial sums over several variables, or storing the real and imaginary parts separately, etc.
%%
%From a communication perspective, there is no benefit to representing a variable by a (nonempty) set of variables unless the latter has a smaller memory footprint, which is impossible given our fixed data width.
%%
%From a computation perspective, however, it may be beneficial to, e.g., reuse intermediate quantities.
%%
%Unfortunately, it may then be difficult or impossible to identify an intermediate variable $v$ uniquely with an array variable $A_j(\phi_j(\indx))$.

% The set of array access expressions of the form $A(\phi(\indx))$ in an execution $E$ is denoted $accesses(E)$.  

\subsection{Derivation of Communication Lower Bound}
\label{sec:lb_deriv}

Now we present the more formal derivation of the communication lower bound, which was sketched 
in \sectn{\ref{sec:model}}.
The approach used here was introduced in \cite{ITT04} and generalized in \cite{BallardDemmelHoltzSchwartz11}.
% second issue raised in the first paragraph of this section.
Here we generalize it again to deal with the more complicated algorithms considered in this paper.

%% Asymptotic parameters
In this work, we are interested in the \emph{asymptotic} communication complexity, in terms of parameters like the fast memory size $M$ and the number of inner loop body iterations $\niters$.
We treat the algorithm's other parameters, like the number of array references $m$, the dimension of the iteration space $d$, the array dimensions $d_j$, and the coefficients in subscripts $\phi_j$, as constants.
When discussing an upper bound $F$ on the number of loop iterations doable with operands in a fast memory of $M$ words, we assume $F = \Omega(M)$, since $m$ is a constant. 
%\TODO{Kathy suggests `since $m$ is a constant' is unclear. 
%%
%The old explanation of the $F = \Omega(M)$ assumption was:  
%``As shown in [Section delegated to Part II], if for all $\indx$ each of the $m$ array references $A_j(\phi_j(\indx))$ are disjoint memory locations, then our upper bound $F=M^{\shbl}$ overestimates the tightest upper bound by a factor of $m^{\shbl}$.
%Thus our resulting lower bound may be too small (thus unattainable) by the same factor. 
%In this work we are primarily interested in a given algorithm's communication cost as a function of $M$; we will assume throughout that $M/m = \Theta(M)$ and hide the factors of $m$ as constants in Big-Oh notation.
%Equivalently, we assume that we are always able to perform $\Omega(M)$ code segments with $M$ operands; in practice if the operands of more than a few code segments cannot fit in cache simultaneously, then other optimizations, like loop-splitting, are more promising than the approach in this paper.''}
% 
% We also assume the total number of loop iterations $\niters = \omega(F)$; otherwise, when $\niters = O(F)$, more care must be taken to treat the possibility that all operands may begin and end in fast memory, resulting in zero communication (we will revisit this assumption in \sectn{\ref{sec:gen-machine}} 
%
So, our asymptotic lower bound may hide a multiplicative factor $c(m,d,\{\phi_1,\ldots,\phi_m\})$, and this constant will also depend on whether any of the arrays $A_1,\ldots,A_m$ overlap.
(Constants are important in practice, and will be addressed in Part~2 of this work.) 
Lastly, we assume $d > 0$, otherwise $\niters \le 1$ and the given algorithm is already trivially communication optimal, and we assume $m > 0$, otherwise there are no arrays and thus no data movement.

% We model the execution of the algorithm on a sequential processor (which may for example be one processor in a parallel machine) as a sequence of individual instructions of three types: ``reads'' that move a word of data (some $A_j(\phi_j(\indx))$ or unique surrogate) from slow memory of unbounded capacity to fast memory (e.g.\ cache) of size $M$, ``writes'' that move a word of data from fast memory to slow memory, and ``arithmetic/logic'' that is only allowed to perform operations on data in fast memory, and store the result in fast memory. We only look at the subset of instructions whose operands include an array element $A_j(\phi_j(\indx))$ or a unique surrogate.
% There may be many other instructions, e.g.\ that update the value of $\indx$, but these generally cause no reads or writes, and ignoring them does not invalidate our lower bound in any event. 

Given an $M$--fit legal sequential execution $E=\tup{e_1,\ldots,e_n}$, we proceed as follows:
\begin{enumerate}
  \item Break $E$ into \emph{$M$--Read/Write segments} of consecutive statements, where each segment (except possibly the last one) contains exactly $M$ Reads and/or Writes.
  Each segment (except the last) ends with the $M^\text{th}$ Read/Write and the next segment starts with whatever statement follows.
  The last segment may have statements other than Reads/Writes at the end to complete the execution.
  (We will simply refer to these as Read/Write segments when $M$ is clear from context.)

\item Independently, break $E$ into \emph{Compute segments} of consecutive statements so that the Compute statements within a segment correspond to the same iteration $\indx$. 
(It will not matter that this does not uniquely define the first and last statements of a Compute segment.)
Our assumption of a legal execution guarantees that there is one Compute segment per iteration $\indx$.
We associate each Compute segment with the (unique) Read/Write segment that contains the Compute segment's first Compute statement.

\item Using the limited availability of data in any one Read/Write segment, we will use Theorem~\ref{thm:1} to establish an upper bound $F$ on the number of complete Compute segments that can be executed during one Read/Write segment (see below).
This is an upper bound on the number of complete loop iterations that can be executed.

\item Now, we can bound below the number of complete Read/Write segments by $\lfloor \niters/(F+1) \rfloor$.
We add $1$ to $F$ to account for Compute segments that overlap two (or more) Read/Write segments.
%(if we did not permit an arithmetic/logic segment to contain read or write instructions, as typical in matrix multiplication, adding 1 would not be necessary). 
We need the floor function because the last Read/Write segment may not contain $M$ Reads/Writes.
(Since we are doing asymptotic analysis, $\niters$ can often be replaced by the total number of Compute statements.)

\item Finally we bound below the total number of Reads/Writes by the lower bound on the number of 
complete Read/Write segments times the number of Reads/Writes per such segment, minus the number of Reads/Writes we inserted to account for variables residing in fast memory before/after the execution: 
\begin{equation} \label{eqn:comm_lb}
M\left\lfloor \frac{\niters}{F+1} \right\rfloor- (I+O) = \Omega\left(\frac{M\cdot \niters}{F}\right)-O(M),
\end{equation}
where we have applied our asymptotic assumption $F = \Omega (M) =\omega(1)$. % and $\niters = \omega(F)$.
\end{enumerate}

%\subsection{Communication Lower Bound}
% \subsection{Determining an Upper Bound $F$}

To determine an upper bound $F$, we will use Theorem~\ref{thm:1} to bound the amount of (useful) computation that can be done given only $O(M)$ array variables.
First, we discuss how to ensure that only a fixed number of array variables is available during a single Read/Write segment. 

Given an $M$--fit legal sequential execution, consider any $M$--Read/Write segment.
There are at most $M$ array variables in fast memory when the segment starts, at most $M$ array variables are read/written during the segment, and at most $M$ array variables remain in fast memory when the segment ends.
If there are no Allocates of array variables, or if there are no Frees of array variables, then at most $2M$ distinct array variables appear in the segment (at most $M$ may already reside in fast memory, and at most $M$ more can be read or allocated).
More generally, if there are no paired Allocate/Free statements of array variables, then at most $3M$ array variables appear in the segment (at most $M$ already reside in fast memory, at most $M$ can be read, and at most $M$ can be allocated).  
However, if we allow array variables to be allocated and subsequently freed, then it is possible to have an unbounded number of array variables contributing to computation in the same segment; this can occur in practice and we give concrete examples in the following section.
Thus, we need an additional assumption in order to obtain a lower bound that is valid for all executions.

We will assume that the execution contains no paired Allocate/Free statements of array variables.
However, we note that of these paired statements, we only need to avoid the ones where both statements occur within a given Read/Write segment; e.g., one could remove the preceding assumption by proving that at least $M$ Read/Write statements (of variables besides $v$) interpose every paired Allocate/Free of an array variable $v$.  
(This weaker assumption is equivalent to an assumption in \cite[\sectn{2}]{BallardDemmelHoltzSchwartz11} that there are no `R2/D2 operands.')

%Recall that we associate each inner loop chunk with a single read-write segment, so each read-write is associated with $t\ge 1$ innter loop chunks, corresponding to loop iterations $\indx_1,\indx_2,\ldots,\indx_t$.
%Given an $M$-fit execution, let $E$ be an $M$-read-write segment of that execution.  There are at most $M$ useful variables in fast memory when $E$ starts and at most $M$ more than can be read during the segment;  if $E$ is temporarly-free, this means are at most $2M$ variables that may be accessed in the segment.   
% 
% We assume that all non-temporary variables are read from memory before being written, i.e., the temporary variables are distinct from those in slow memory; this assumption simplifies the identification of true temporaries, as opposed to ``blind writes'' to a variable that will be written back to slow memory, but allowing blind writes does not change the $2M$ bound. 
%
% Note that is some of the array expressions, $A(\phi(\indx))$ map to the same variable due to a non-injective $\phi$ or aliased arrays, the total number of variables (memory location) may be lower, but cannot exceed $2M$.  

The communication lower bound is now a straightforward application of Theorem~\ref{thm:1}.

\begin{theorem}\label{Thm4.1}
Consider an algorithm in the geometric model of \eqref{eqn_AlgModel} in \sectn{\ref{sec:model}}, and consider any $M$--fit, legal sequential execution which contains no paired Allocate/Free statements of array variables.
If the linear constraints \eqref{subcriticalhypothesis} of Theorem~\ref{thm:1} are feasible, then for sufficiently large $\niters$, the number of Reads/Writes in the execution is $\Omega (\niters/M^{\shbl-1})$, where $\shbl$ is the minimum value of $\sum_{j=1}^m s_j$ subject to \eqref{subcriticalhypothesis}.
\end{theorem}
\begin{proof}
 % {I don't get the step 4 reference.}
% The bound $F$ needed in step~4 of the analysis above may be derived from
% Theorem~\ref{thm:1} as follows: If $E \subset \ZZ^d$ is the (finite) set of complete inner loop
The bound $F$ may be derived from Theorem~\ref{thm:1} as follows.
If $E \subseteq \iters$ is the (finite) set of complete inner loop body iterations executed during a Read/Write segment, then we may bound $|E| \le \prod_{j=1}^m |\phi_j(E)|^{s_j}$ for any $s=\tup{s_1,\ldots,s_m}$ satisfying \eqref{subcriticalhypothesis}.
By our assumptions, each $M$--Read/Write segment has at most $3M$ distinct array variables whose values reside in fast memory.
This implies that $\max_j |\phi_j(E)| \le 3M$, since we allow arrays to alias%
\footnote{If the arrays do not alias, then the tighter constraint $\sum_{j=1}^m |\phi_j(E)| \le 3M$ holds instead, although our bound here is still valid. We will discuss tightening our bounds for nonaliasing arrays in Part~2 of this work.}.
So, $|E| \le \prod_{j=1}^m (3M)^{s_j} = (3M)^{\sum_{j=1}^m s_j}$.
Since this bound applies for any $s$ satisfying \eqref{subcriticalhypothesis}, we choose an $s$ minimizing $\sum_{j=1}^m s_j$, obtaining the tightest bound $|E| \le (3M)^{\shbl} \qec F$.
The communication lower bound is $\Omega(M\cdot\niters/F)-O(M)$; if we assume $\niters=\omega(F)$, i.e., the iteration space is `sufficiently large,' then we obtain $\Omega(\niters/M^{\shbl-1})$ as desired.
\end{proof}
When $\niters = \Theta(F)$, i.e., the problem (iteration space) is not sufficiently large, the lower bound becomes $\Omega(M)-O(M)$, so the subtractive $O(M)$ term may dominate and lead to zero communication; this is increasingly likely as the ratio $\niters/F$ goes to zero.
The parallel case also demonstrates this behavior in the `strong scaling' limit, when the problem is decomposed to the point that each processor's working set fits in their local memory (see \sectn{\ref{sec:gen-machine}}).
In the regime $\niters = O(F)$, a \emph{memory-independent} lower bound \cite{BDHLS12} provides more insight than the bound above.
    Let $W$ be the (unknown) number of Reads/Writes performed, and let $I$ and $O$ be the numbers of input/output variables residing in fast memory before/after execution.
    Then
  \[ \niters \le \prod_j |\phi_j(\iters)|^{s_j} \le (\max_j |\phi_j(\iters)|)^{\shbl} \le (W+I+O)^{\shbl}, \]
 and we have $W \ge \niters^{1/\shbl}-(I+O)$; see \sectn{7.3} for further discussion.
While this bound applies for any $\niters\ge1$, the memory-\emph{dependent} bound in Theorem~\ref{Thm4.1} is tighter when $\niters$ is sufficiently large; when $\niters$ is not so large, the memory-independent bound may be tighter. 
In Part~1 of this work, we are interested in lower bounds for large problems, and so only consider the memory-dependent bound.
In Part~2, we will revisit the case of smaller problems when we discuss the constants hidden in our asymptotic bounds.

\subsection{Examples}
\label{sec:lb_examples}
We give four examples to demonstrate why our assumptions in Theorem~\ref{Thm4.1} are necessary.
Then, we discuss how one can sometimes deal with the presence of imperfectly nested loops, paired Allocate/Free statements, and an infeasible linear program to compute a useful lower bound; we give an example of this approach.

\begin{description}
\item[Example: Modified Matrix Multiplication with paired Allocates/Frees (I).]
The following simple modification of matrix multiplication demonstrates how paired Allocate/Frees can invalidate our lower bound.
\begin{align*}
&\For{i_1}{1}{N} ,\quad \For{i_2}{1}{N} ,\quad \For{i_3}{1}{N} ,\\
&\qquad C(i_1,i_2) = A(i_1,i_3)\cdot B(i_3,i_2) \\
&\qquad t = t + C(i_1,i_2)^7
\end{align*}
Suppose we know the arrays do not alias each other.
Clearly $C$ only depends on the data when $i_3=N$, but we need to do all $N^3$ multiplications to compute $t$ correctly.
But the same analysis from \sectn{\ref{sec:hbl}} applies to these loops as to matrix multiplication, suggesting a sequential communication lower bound of  $\Omega (N^3/M^{1/2})$.
However, it is clearly possible to execute all $N^3$ iterations moving only $O(N^3/M + N^2)$ words, by hoisting the (unblocked) $i_3$ loop outside and blocking the $i_1$ and $i_2$ loops by $M/2$, doing $(M/2)^2$ multiplications in a Read/Write segment using $M/2$ entries each of $A(\cdot,i_3)$ and $B(i_3,\cdot)$, and (over)writing the values $C(i_1,i_2)$ to a single location in fast memory, which is repeatedly Allocated and Freed. 
Only when $i_3=N$ would $C(i_1,i_2)$ actually be written to slow memory. 
So in this case there are a total of $N^3 - N^2$ paired Allocates/Frees, corresponding to the overwritten $C(i_1,i_2)$ operands.
\exend
\item[Example: Modified Matrix Multiplication with Paired Allocate/Frees (II).]
This example also demonstrates how paired Allocate/Frees can invalidate our lower bound.
Consider the following code:
\begin{equation*}
\begin{split}
&\For{i_1}{1}{N} ,\; \For{i_2}{1}{N} ,\quad A(i_1,i_2) = e^{2\pi i (i_1-1)(i_2-1)/N} \\
&\For{i_1}{1}{N} ,\; \For{i_2}{1}{N} ,\; \For{i_3}{1}{N} ,\quad C(i_1,i_2) \plusequals A(i_1,i_3)\cdot B(i_3,i_2)
%&\text{{\bf return } $C$ \qquad // $A$ is not accessed again} 
%&\For{i_1}{1}{N},\quad\For{i_2}{1}{N} \\
%&\qquad A(i_1,i_2) = e^{2\pi i (i_1-1)(i_2-1)/N}/N^{1/2} \\
%&\For{i_1}{1}{N},\quad\For{i_2}{1}{N},\quad\For{i_3}{1}{N} \\
%&\qquad C(i_1,i_2) \plusequals A(i_1,i_3)\cdot B(i_3,i_2) \\
%&\text{{\bf return } $C$ \qquad // $A$ is not accessed again} 
\end{split}
\end{equation*}
Again, suppose we know that the arrays do not alias each other.
Toward a lower bound, we ignore the initialization of $A$ (first loop nest) and only look at the second loop nest, a matrix multiplication.
However, by computing entries of $A$ on-the-fly from $\indx$ and discarding them (i.e., Allocating/Freeing them), one can beat the lower bound of $\Omega(N^3/M^{1/2})$ words for matrix multiplication.
That is, by hoisting the (unblocked) $i_2$ loop outside and blocking the $i_1$ and $i_3$ loops by $M/2$, and finally writing each $A(i_1,i_3)$ to slow memory when $i_2=N$, we can instead move $O(N^3/M+N^2)$ words.
So, there are $N^3-N^2$ possible paired Allocate/Frees.
%
% This example also demonstrates that, when computing a lower bound for a subset of a program, one must be careful when omitting code that enables R2/D2 operands.
\exend
\item[Example: Infeasibility.]
This example demonstrates how infeasibility of the linear constraints \eqref{subcriticalhypothesis} of Theorem~\ref{thm:1} can invalidate our lower bound.
Consider the following code:
\begin{align*}
&\For{i_1}{1}{N_1} ,\quad \For{i_2}{1}{N_2} ,\\
&\qquad A(i_1) =\text{func}(A(i_1))
\end{align*}
It turns out the linear constraints  \eqref{subcriticalhypothesis} are infeasible, so we cannot apply Theorem~\ref{thm:1} to find a bound on data reuse of the form $M^{\shbl-1}$. 
It is easy to see that only $2N_1$ Reads and Writes of $A$ are needed to execute the inner loop $N_1 \cdot N_2$ times, i.e., unbounded ($N_2$--fold) data reuse is possible.
In general, infeasibility suggests that the number of available array variables in fast memory during a Read/Write segment is constrained only by the iteration space $\iters$.
(We will explore infeasibility further in \sectns{\ref{sec:bound-nontrivialkernel}~and~\ref{sec:attain-nontrivialkernel}}.)

% Not true (?) :
% Also note that variables of $A$ in iterations $\indx \in \set{(i_1,i_2) : 1 < i_2 < N_2}$ can be allocated and subsequently freed.
%
While infeasibility may be sufficient for there to be an unbounded number of array variables in a Read/Write segment, the previous two examples show that it is not necessary, since their linear programs are feasible.
We will be more concrete about this relationship between infeasibility and unbounded data reuse in Part~2.
\exend
\item[Example: Loop interleaving.]
This example demonstrates how an execution which interleaves the inner loop bodies (an \emph{illegal} execution) can invalidate our lower bound; see also \sectn{\ref{sec_LoopSplitting}}.
We will see in \sectn{\ref{sec_LoopSplitting}} that the lower bound for each split loop is no larger than the lower bound for the original loop.
Consider splitting the two lines of the inner loop body in the Complicated Code example (see \sectn{\ref{sec:intro}}) into two disjoint loop nests (each over $\indx \in \set{1,\ldots,N}^6$).
We assume $\text{func}_1$ and $\text{func}_2$ do not modify their arguments, and that the arrays do not alias --- the two lines share only read accesses to one array, $A_3$, so correctness is preserved.
As will be seen later by using Theorem~\ref{thm6.1}, the resulting two loop nests have lower bounds $\Omega(N^6/M^{3/2})$ and $\Omega(N^6/M^2)$, resp., both better than the $\Omega(N^6/M^{8/7})$ of the original, and both these lower bounds are attainable.
\exend
\end{description}

% Example (expanded below)
%Another example, that violates the assumption of perfectly nested loops and
%also leads to a large number of R2/D2 operations, is matrix powering $B = A^k$
%(shown for simplicity for odd $k$):
%\begin{equation}\label{eqn_R2D2c}
%B = A, \; {\rm for} \; i = 1: \lfloor k/2 \rfloor, \; C = A \cdot B, \; B  = A \cdot C
%\end{equation}
%These loops are not perfectly nested, because the 3 nested loops computing the product
%$C = A \cdot B$ need to (partially) complete before the product $B = A \cdot C$ can begin. 
%When the matrix dimension $N$ is small enough that all three matrices $A,B,C$ fit
%in fast memory, it is easy to see that we can do an unbounded  (with $k$) number of
%arithmetic operations with only $2N^2$ reads and writes. Another way to explain this
%is that all but the input $A$ and the final $B$ matrix are R2/D2 operands. We show how to
%derive a valid lower bound for this example below.

Theorem~\ref{Thm4.1} is enough for many direct linear algebra computations such as (dense or sparse) $LU$ decomposition, which do not have paired Allocate/Frees, but not all algorithms for the $QR$ decomposition or eigenvalue problems, which can potentially have large numbers of paired Allocates/Frees (see \cite{BallardDemmelHoltzSchwartz11}).
We can often deal with interleaving iterations, paired Allocates/Frees, and infeasibility of \eqref{subcriticalhypothesis} by {\em imposing Reads and Writes} \cite[\sectn{3.4}]{BallardDemmelHoltzSchwartz11}: we modify the algorithm to add (``impose'') Reads and Writes of array variables which are allocated/freed or repeatedly overwritten, apply the lower bound from Theorem~\ref{Thm4.1}, and then subtract the number of imposed Reads and Writes to get the final lower bound.
(Note that we have already used a similar technique, above, to allow an execution to begin/end with a nonzero fast memory footprint.)
We give an example of this approach (see also \cite[Corollary~5.1]{BallardDemmelHoltzSchwartz11}).

\begin{exAk}[Part~1/4]
Consider computing $B = A^k$ using the following code, shown (for simplicity) for odd $k$ and initially $B=A$:
\begin{align*}
&\text{// Original Code} \\
&\For{i_1}{1}{\lfloor k/2 \rfloor} \\
&\qquad C = A\cdot B \\
&\qquad B = A\cdot C
\end{align*}
In order to get the correct answer, we assume the arrays do not alias each other.
Under our asymptotic assumption that the number $m$ of arrays accessed in the inner loop is constant with respect to $\niters$, one cannot simply model the code as a $1$--deep loop nest (over $i_1$), since each entry of $A,B,C$ would be considered an array.
Considering instead the scalar multiply/adds as the `innermost loop,' we violate an assumption of the geometric model \eqref{eqn_AlgModel}: since $C = A \cdot B$ needs to be (partially) completed before $B = A \cdot C$ can be computed, the innermost loops necessarily interleave.
Toward obtaining a lower bound, we could try to apply our theory to a perfectly nested subset of the code, omitting $B = A \cdot C$; as will be shown in part~2/4 of this example (see \sectn{\ref{sec:bound}}), the corresponding linear constraints in \eqref{subcriticalhypothesis} are then infeasible, violating another assumption of Theorem~\ref{Thm4.1}.
The same issue arises if we omit $C=A\cdot B$; furthermore, neither modification prevents the possibility of paired Allocates/Frees of array variables of $B,C$.
%
%Finally, one can see that if the $N$--by--$N$ matrices $A$, $B$, and $C$ are small enough to all fit in fast memory simultaneously, then one can perform the original code with just $2N^2$ Reads and Writes (independent of $k$), because for all iterations $i_1$ except for the last, array variables of $C$ can exist within paired Allocates/Frees, violating another assumption of Theorem~\ref{Thm4.1}.
%
We deal with all three violations by rewriting the algorithm in the model \eqref{eqn_AlgModel} and imposing Reads and Writes of all intermediate powers $A^i$ with $1 < i_1 < k$, so at most $2(k-2)N^2$ Reads and Writes altogether. 
This can be expressed by using one array $\hat{A}$ with three subscripts, where $\hat{A}(1,i_2,i_3) \ceq A(i_2,i_3)$ and all other entries are zero-initialized.
\begin{align*}
&\text{// Modified Code (Imposed Reads and Writes)} \\
&\For{i_1}{2}{k} ,\quad \For{i_2}{1}{N} ,\quad \For{i_3}{1}{N} ,\quad \For{i_4}{1}{N} ,\\
&\qquad\hat{A}(i_1,i_2,i_3) \plusequals \hat{A}(1,i_2,i_4) \cdot \hat{A}(i_1-1,i_4,i_3)
\end{align*}
This code clearly matches the geometric model \eqref{eqn_AlgModel}, and admits legal executions (which would be interleaving executions of the original code).
%
% By insisting that intermediate powers of $A$ be written to memory and not discarded, this eliminates all Frees (and therefore paired Allocates/Frees).
%
% There still may be Allocates, since entries of $A^{i_1}$ may be created ``on the fly'' during a Read/Write segment, but we may now apply Theorem~\ref{Thm4.1}.
%
As will be shown in part~3/4 of this example (see \sectn{\ref{sec:bound}}), the linear constraints \eqref{subcriticalhypothesis} are now feasible, and the resulting exponent from Theorem~\ref{Thm4.1} will be $\shbl = 3/2$, so if the matrices are all $N$--by--$N$, the lower bound of the original program will be $\Omega(kN^3/M^{\shbl-1}) - 2(k-2)N^2$.
For simplicity, suppose that this can be rewritten as $\Omega(kN^3/M^{1/2} - kN^2)$.
So we see that when the matrices are small enough to fit in fast memory $M$, i.e., $N \le M^{1/2}$, the lower bound degenerates to $0$, which is the best possible lower bound which is also proportional to the total number of loop iterations $\niters = (k-1)N^3$.
But for larger $N$, the lower bound simplifies to $\Omega(kN^3/M^{1/2})$ which is in fact attained by the natural algorithm that does $k-1$ consecutive (optimal) $N$--by--$N$ matrix multiplications.
 
This example also illustrates that our results will only be of interest for sufficiently large problems, certainly where the floor function in the lower bound \eqref{eqn:comm_lb} is at least $1$.
\exend
\end{exAk}

The above approach covers many but not all algorithms of interest.
We refer to reader to \cite[\sectns{3.4~and~5}]{BallardDemmelHoltzSchwartz11} for more examples of imposing Reads and Writes, and \cite[\sectn{4}]{BallardDemmelHoltzSchwartz11} on orthogonal matrix factorizations for an important class of algorithms  where a subtler analysis is required to deal with paired Allocates/Frees.

Imposing Reads and Writes may fundamentally alter the program, so the lower bound obtained for the modified code need not apply to the original code.
In the example above, one could reorder\footnote{For simplicity, and without reducing data movement, this code performs $(k-1)N^2$ additional $+$ operations.} the original code to
\begin{align*}
&\For{i_2}{1}{N} ,\quad \For{i_3}{1}{N} ,\quad \For{i_4}{1}{N} ,\quad \For{i_1}{1}{\lfloor k/2 \rfloor} ,\\
&\qquad C(i_2,i_3) \plusequals A(i_2,i_4)\cdot B(i_4,i_3) \\
&\qquad B(i_2,i_3) \plusequals A(i_2,i_4)\cdot C(i_4,i_3).
\end{align*}
Recall that our lower bounds are valid for any reordering of the iteration space, correct or otherwise.
Since $i_1$ does not appear in the inner loop body, data movement is independent of $k$.
In fact, this code moves $O(N^3/M^{1/2})$ words, asymptotically beating the lower bound for the modified code (with imposed Reads/Writes).
For another example, see \cite[\sectn{5.1.3}]{BallardDemmelHoltzSchwartz11}.
In Part~2, we will present an alternative to imposing Reads/Writes which can yield a (nontrivial) lower bound on the original code: roughly speaking, one ignores the loops whose indices do not appear in any subscripts (e.g., $i_1$, above).
While this approach will eliminate infeasibility, paired Allocates/Frees may still arise and imposing Reads/Writes may still be necessary.

\subsection{Loop Splitting Can Only Help}
\label{sec_LoopSplitting}
Here we show that loop splitting can only reduce (improve) the communication lower
bound expressed in Theorem~\ref{Thm4.1}. More formally, we state this as the following.

\begin{theorem}\label{Thm4.2}
Suppose we have an algorithm satisfying the hypotheses of Theorem~\ref{Thm4.1}, with lower bound determined by the value $\shbl$.
Also suppose that this algorithm can be rewritten as $t>1$ consecutive (disjoint) loop nests, where each has the same iteration space as the original algorithm but accesses a subset of the original array entries, and each satisfies the hypotheses of Theorem~\ref{Thm4.1}, leading to exponents $\shbla{i}$ for $i\in\set{1,\ldots,t}$.
Then each $\shbla{i} \ge \shbl$, i.e., the communication lower bound
for each new loop nest is asymptotically at least as small as the
communication lower bound for the original algorithm.
\end{theorem}

\begin{proof}
By our assumptions, $\shbl$ is the minimum value of $\sum_{j=1}^m s_j$
where $s\in[0,1]^m$ satisfies the 
constraints \eqref{subcriticalhypothesis}.
We need to prove that if we replace these constraints by 
\begin{equation}\label{subcriticalsubset}
\Rank{H}\le \sum_{j \in L_i} s_j\Rank{\phi_j(H)}
\qquad\text{for all subgroups } H \subim \ZZ^d,
\end{equation}
where $L_i$ is any given nonempty subset of $\set{1,\ldots,m}$
corresponding to the $i^\text{th}$ (of $t$) new loop nest,
then the minimum value of
$\shbla{i} = \sum_{j \in L_i} s_j$ for any $s$ satisfying \eqref{subcriticalsubset}
must satisfy $\shbla{i} \ge \shbl$.
We proceed by replacing both sets of constraints by a common finite set,
yielding linear programs, and then use duality.

As mentioned immediately after the statement of Theorem~\ref{thm:1}, even though
there are infinitely many subgroups $H \subim \ZZ^d$, there are only finitely
many possible values of each rank, so we may choose a finite set ${\cal S}$
of subgroups $H \subim \ZZ^d$ that yield all possible constraints 
\eqref{subcriticalhypothesis}.
Similarly, there is a finite set of subgroups ${\cal S}_i$ that yields all
possible constraints \eqref{subcriticalsubset}. 
Now let ${\cal S}^* = {\cal S} \cup {\cal S}_i$; we may therefore
replace both \eqref{subcriticalhypothesis} and \eqref{subcriticalsubset}
by a finite set of constraints, for the subgroups 
$H \in {\cal S}^* = \{H_1,\ldots,H_g\}$.

This lets us rewrite \eqref{subcriticalhypothesis} as the linear program
of minimizing $\shbl = \sum_{j=1}^m s_j = 1_m^T \cdot s$
subject to $s^T \cdot \Delta \geq r^T$, where 
$s = (s_1,\ldots,s_m)^T$, $1_m$ is a column vector of ones,
$r = (\Rank{H_1},\ldots,\Rank{H_g})^T$,
and $\Delta$ is $m$--by--$g$ with $\Delta_{ij} = \Rank{\phi_i(H_j)}$.
Assuming without loss of generality that the constraints $L_i$ are the
first constraints $L_i = \set{1,2,\ldots,|L_i|}$,
we may similarly rewrite \eqref{subcriticalsubset} as the linear program
of minimizing $\shbla{i} = \sum_{j=1}^{|L_i|} s_j = 1_{|L_i|}^T \cdot s_{L_i}$
subject to $s_{L_i}^T \cdot \Delta_{L_i} \ge r^T$, where $\Delta_{L_i}$, $s_{L_i}$ and $1_{|L_i|}$ 
consist of the first $|L_i|$ rows of $\Delta$, $s$ and $1_m$, resp.
The duals of these two linear programs, which have the same optimal
values as the original linear programs, are
\begin{equation} \label{eqn_Dual}
{\rm maximize} \; \; r^T \cdot x \; \;  {\rm subject\ to} \; \;  \Delta \cdot x \le 1_m
\end{equation}
and
\begin{equation} \label{eqn_Dualsubset}
{\rm maximize} \; \; r^T \cdot x_{L_i} \; \; {\rm subject\ to} \; \;  \Delta_{L_i} \cdot x_{L_i} \le 1_{|L_i|}
\end{equation}
resp., where both $x$ and $x_{L_i}$ have dimension $g$.
It is easy to see that \eqref{eqn_Dualsubset} is maximizing the same quantity 
as \eqref{eqn_Dual}, but subject to a subset of the constraints of \eqref{eqn_Dual}, 
so its optimum value, $\shbla{i}$, must be at last as large as the other optimum, 
$\shbl$.
\end{proof}

While loop splitting appears to always be worth attempting, in practice data dependencies limit our ability to perform this optimization; we will discuss the practical aspects of loop splitting further in Part~2 of this work.

\subsection{Generalizing the machine model}
\label{sec:gen-machine}
Earlier we said the reader could think of a sequential algorithm where the fast memory consists of a cache of $M$ words, and slow memory is the main memory. 
In fact, the result can be extended to the following situations:
\begin{description}
\item[Multiple levels of memory.] 
If a sequential machine has a memory hierarchy, i.e., multiple levels of cache (most do), where data may only move between adjacent levels, and arithmetic done only on the ``top'' level, then it is of interest to bound the data transferred between every pair of adjacent levels, say $i$ and $i+1$, where $i$ is higher (faster and closer to the arithmetic unit) than $i+1$.
In this case we apply our model with $M$ representing the total memory available in levels 1 through $i$, typically an increasing function of $i$.
\item[Homogeneous parallel processors.]
 We call a parallel machine {\em homogeneous} if it consists of $P$ identical processors, each with its own memory, connected over some kind of network.
  For any processor, fast memory is the memory it owns, and Reads and Writes refer to moving data over the network from or to the other processors' memories. 
  Recalling notation from above, consider an $\{M,\ldots,M\}$--fit, legal parallel execution $\tup{E_1,\ldots,E_P}$, i.e., a set of $M$--fit, legal sequential executions $E_i$ (with corresponding $\iters_i \subseteq \iters$); assuming there are no paired Allocate/Frees, we can apply Theorem~\ref{Thm4.1} to each  and obtain the lower bound of $\Omega(|\iters_i|/M^{\shbl - 1})$ words moved.
  As argued in \cite{BDHLS12}, to minimize computational costs, necessarily
  %a constant fraction of the processors perform $\Theta(\niters/P)$ (distinct) iterations; since a lower bound for words moved by one of these processors is a lower bound for the words moved along the critical path, we assume $|\iters_i| = \Theta(\niters/P)$.
  at least one of the processors performs $\niters/P$ (distinct) iterations; since a lower bound for this processor is a lower bound for the critical path, we take $|\iters_i| = \niters/P$, and obtain the lower bound of $\Omega(\niters/(PM^{\shbl - 1}))$ words moved (along the critical path).
  
  We recall that Theorem~\ref{Thm4.1} makes an asymptotic assumption on $\niters$; having replaced $\niters$ by $\niters/P$ in the parallel case, this assumption becomes $\niters/P = \omega(M^\sigma)$.
   When this assumption fails (e.g., $\niters$ is small or $P$ is large), it may be possible for each processor to store its entire working set locally, with no communication needed.
  As in the sequential case (see above), one may obtain a memory-independent lower bound  $W = (\niters/P)^{1/\shbl}-(I+O)$, which may be tighter in this regime; this result generalizes \cite[Lemma~2.3]{BDHLS12}.
  Interestingly, one can show that under certain assumptions on $I$ and $O$ (e.g., only one copy of the inputs/outputs is permitted at the beginning/end), while the communication cost continues to decrease with $P$ (up to the natural limit $P=\niters$), it fails to strong-scale perfectly when the assumption on $\niters/P$ breaks; e.g., see \cite[Corollary~3.2]{BDHLS12}.
  We will discuss strong scaling limits further in Part~2 of this work.

 One may also ask what value of $M$ to use for each processor.
  Suppose that each processor has $M_\text{max}$ words of fast memory, and that the total problem size of all the array entries accessed is $M_\text{array}$. 
  So if each processor gets an equal share of the data we use $M = M_\text{array}/P \leq M_\text{max}$.
But the lower bound may still apply, and be smaller, if $M$ is larger than $M_\text{array}/P$ (but at most $M_\text{max}$).
In some cases algorithms are known that attain these smaller lower bounds (e.g., matrix multiplication in \cite{2.5D_EuroPar}), i.e., replicating data can reduce communication.

In \sectn{7.3}, we discuss attainability of these parallel lower bounds, and reducing communication by replicating data.
\item[Hierarchical parallel processors.] The simplest possible hierarchical machine is the sequential one with multiple levels of memory discussed above.
But real parallel machines are similar: each processor has its own memory organized in a hierarchy. 
So just as we applied our lower bound to measure memory traffic between levels $i$ and $i+1$ of cache on a sequential machine, we can similarly analyze the memory hierarchy on each processor in a parallel machine.
\item[Heterogeneous machines.] Finally, people are building heterogeneous parallel machines, where the various processors, memories, and interconnects can have different speeds or sizes. 
  Since minimizing the total running time means minimizing the time when the last processor finishes, it may no longer make sense to assign an equal fraction $\niters/P$ of the work and equal subset of memory $M$ to each processor. 
  Since our lower bounds apply to each processor independently, they can be used to formulate an optimization problem that will give the optimal amount of work to assign to each processor \cite{Hetero_SPAA11}.
\end{description}

\section{(Un)decidability of the communication lower bound} 
\label{sec:undecidability}

In \sectn{\ref{sec:hbl}}, we proved Theorem~\ref{thm:1}, which tells us that the exponents $s\in[0,1]^m$ satisfy the inequalities \eqref{subcriticalhypothesis}, i.e.,
\begin{equation*}
\Rank{H} \leq \sum_{j=1}^m s_j \Rank{\phi_j(H)} \qquad \text{for all subgroups } H \subim \ZZ^d,
\end{equation*}
precisely when the desired bound \eqref{mainconclusion} holds.
If so, then following Theorem~\ref{Thm4.1}, the sum of these exponents $\shbl \ceq \sum_{j=1}^m s_j$ leads to a communication lower bound of $\Omega(\niters/M^{\shbl-1})$. 
Since our goal is to get the tightest bound, we want to minimize $\sum_{j=1}^m s_j$ subject to \eqref{subcriticalhypothesis} and $s\in[0,1]^m$.
In this section, we will discuss computing the set of inequalities \eqref{subcriticalhypothesis}, in order to write down and solve this minimization problem.

We recall from \sectn{\ref{sec:P}} that the feasible region for $s$ (defined by these inequalities) is a convex polytope $\PP \subseteq [0,1]^m$ with finitely many extreme points.
While $\PP$ is uniquely determined by its extreme points, there may be many sets of inequalities which specify $\PP$; thus, it suffices to compute any such set of inequalities, rather than the specific set \eqref{subcriticalhypothesis}.
This distinction is important in the following discussion.
In \sectn{\ref{sec:computeP}}, we show that there is an effective algorithm to determine $\PP$.
However, it is not known whether it is decidable to compute the set of inequalities \eqref{subcriticalhypothesis} which define $\PP$ (see \sectn{\ref{sec5.1}}).
In \sectn{\ref{sec6.4}}, we discuss two approaches for approximating $\PP$, providing upper and lower bounds on the desired $\shbl$.

% In Part~2 of this work, we will show how to improve the efficiency of the second approach (see \sectn{\ref{sec:conclusions}}). 

\subsection{An Algorithm Which Computes $\PP$}
\label{sec:computeP}
We have already shown in Lemma~\ref{lem:PZequalsPQ} that the polytope $\PP$ is unchanged when we embed the groups $\ZZ^d$ and $\ZZ^{d_j}$ into the vector spaces $\QQ^d$ and $\QQ^{d_j}$ and consider the homomorphisms $\phi_j$ as $\QQ$--linear maps.
Thus it suffices to compute the polytope $\PP$ corresponding to the inequalities
\begin{equation*}
\Dim{V} \leq \sum_{j=1}^m s_j \Dim{\phi_j(V)} \qquad \text{for all subgroups } V \subim \QQ^d.
\end{equation*}
Indeed, combined with the constraints $s\in[0,1]^m$, this is the hypothesis \eqref{subcriticalhypothesisfield} of Theorem~\ref{thm:field} in the case $\FF=\QQ$.

We will show how to compute $\PP$ in the case $\FF = \QQ$; for the remainder of this section, $V$ and $V_j$ denote finite-dimensional vector spaces over $\QQ$, and $\phi_j$ denotes a $\QQ$--linear map.
We note that the same reasoning applies to any countable field $\FF$, provided that elements of $\FF$ and the field operations are computable.

% when $\FF$ is any field with countably many elements; this includes the case $\FF=\QQ$ but not, e.g., $\FF=\RR$.
%
%Let $\FF$ be a field with countably many elements.
%By an HBL (vector space) datum \[(V,\set{V_j: 1\le j\le m},\set{\phi_j: 1\le j\le m})\] 
%is meant a finite-dimensional vector space $V$ over $\FF$,
%a positive integer $m$,
%finite-dimensional vector spaces $V_j$ over $\FF$,
%and $\FF$--linear transformations $\phi_j:V\to V_j$.
%
% Throughout this discussion, $\FF$ denotes a finite or countably infinite field, and it is assumed that there is given some algorithm which produces a list of the elements of $\FF$.  
% It is assumed that there is given some algorithm which produces a list of the elements of $\FF$. 
%
% We say that such a field is computable.
%
% This holds in particular for $\FF=\QQ$.

\begin{theorem} \label{thm:decision}
There exists an algorithm which takes as input any vector space HBL datum over 
% a computable countable field
the rationals, i.e., $\FF=\QQ$, and returns as output both a list of finitely many linear inequalities over $\ZZ$ which jointly specify the associated polytope $\PP\tup{V,\tup{V_j},\tup{\phi_j}}$, and a list of all extreme points of $\PP\tup{V,\tup{V_j},\tup{\phi_j}}$.
\end{theorem}
\begin{remark} \label{rmk:V07}
The algorithm we describe below to prove Theorem~\ref{thm:decision} is highly inefficient, because of its reliance on a search of arbitrary subspaces of $V$.  
A similar algorithm was sketched in \cite{Valdimarsson07} for computing the polytope in \cite[Theorem~2.1]{BCCT10}, a result related to Theorem~\ref{thm:field} in the case $\FF=\RR$ but where the vector spaces $V,V_j$ are endowed with additional inner-product structure.
That algorithm searches a smaller collection of subspaces, namely the lattice generated by the nullspaces of $\phi_1,\ldots,\phi_m$. 
In Part~2 of this work (see \sectn{\ref{sec:conclusions}}) we will show that it suffices to search this same lattice in our case; this may significantly improve the efficiency of our approach.
Moreover, this modification will also allow us to relax the requirement that $\FF$ is countable, although this is not necessary for our application.
\end{remark}

The proof of Theorem~\ref{thm:decision} is built upon several smaller results.

\begin{lemma} \label{lemma:enumeratesubspaces}
There exists an algorithm which takes as input a finite-dimensional vector space $V$ over 
% $\FF$
$\QQ$, and returns a list of its subspaces.
More precisely, this algorithm takes as input a finite-dimensional vector space $V$ and a positive integer $N$, and returns as output the first $N$ elements $W_i$ of a list $\tup{W_1,W_2,\ldots}$ of all subspaces of $V$.
This list is independent of $N$.
Each subspace $W$ is expressed as a finite sequence $\tup{d;w_1,\ldots,w_d}$ where $d=\Dim{W}$ and $\set{w_i}$ is a basis for $W$.
\end{lemma}
\begin{proof}
Generate a list of all nonempty subsets of $V$ having at most $\Dim{V}$ elements.
Test each subset for linear independence, and discard all which fail to be independent.
Output a list of those which remain.
\end{proof}
We do not require this list to be free of redundancies.

\begin{lemma} \label{lemma:enumeratevertices}
For any positive integer $m$, there exists an algorithm which takes as input a finite set of linear inequalities over $\ZZ$ for $s\in[0,1]^m$, and returns as output a list of all the extreme points of the convex subset $\PP\subseteq [0,1]^m$ specified by these inequalities.
\end{lemma}
\begin{proof}
To the given family of inequalities, adjoin the $2m$ inequalities $s_j\ge 0$ and $-s_j\ge -1$. 
$\PP$ is the convex polytope defined by all inequalities in the resulting enlarged family. 
Express these inequalities as $\langle s,v_\alpha\rangle\ge c_\alpha$ for all $\alpha\in A$, where $A$ is a finite nonempty index set.

An arbitrary point $\tau\in\RR^m$ is an extreme point of $\PP$ if and only if 
(i) there exists a set $B$ of indices $\alpha$ having cardinality $m$, such that $\set{v_\beta: \beta\in B}$ is linearly independent and $\langle \tau,v_\beta\rangle = c_\beta$ for all $\beta\in B$, and 
(ii) $\tau$ satisfies $\langle \tau,v_\alpha\rangle\ge c_\alpha$ for all $\alpha\in A$.

Create a list of all subsets $B \subset A$ 
% $B \subseteq A$ 
with cardinality equal to $m$.
There are finitely many such sets, since $A$ itself is finite.
Delete each one for which $\set{v_\beta: \beta\in B}$ is not linearly independent. 
For each subset $B$ not deleted, compute the unique solution $\tau$ of the system of equations $\langle \tau,v_\beta\rangle = c_\beta$ for all $\beta\in B$. 
Include $\tau$ in the list of all extreme points, if and only if $\tau$ satisfies $\langle \tau,v_\alpha\rangle \ge c_\alpha$ for all $\alpha\in A\setminus B$. 
\end{proof}

\begin{proposition} \label{prop:subalg}
There exists an algorithm which takes as input a vector space HBL datum $\tup{V,\tup{V_j},\tup{\phi_j}}$, an element $t\in[0,1]^m$, and a subspace $\set{0} \subpr W \subpr V$ which is critical with respect to $t$, and determines whether $t\in\PP\tup{V,\tup{V_j},\tup{\phi_j}}$.
\end{proposition}

Theorem~\ref{thm:decision} and Proposition~\ref{prop:subalg} will be proved inductively in tandem, according to the following induction scheme. 
The proof of Theorem~\ref{thm:decision} for HBL data in which $V$ has dimension $n$, will rely on Proposition~\ref{prop:subalg} for HBL data in which $V$ has dimension $n$.
The proof of Proposition~\ref{prop:subalg} for HBL data in which $V$ has dimension $n$ and there are $m$ subspaces $V_j$, will rely on Proposition~\ref{prop:subalg} for HBL data in which $V$ has dimension strictly less than $n$, on Theorem~\ref{thm:decision} for HBL data in which $V$ has dimension strictly less than $n$, and also on Theorem~\ref{thm:decision} for HBL data in which $V$ has dimension $n$ and the number of subspaces $V_j$ is strictly less than $m$.
Thus there is no circularity in the reasoning.

\begin{proof}[Proof of Proposition~\ref{prop:subalg}]
Let $\tup{V,\tup{V_j},\tup{\phi_j}}$ and $t,W$ be given. 
Following Notation~\ref{not:splittingspaces}, consider the two HBL data $\tup{W,\tup{\phi_j(W)},\tup{\restr{\phi_j}{W}}}$ and $\tup{V/W,\tup{V_j/\phi_j(W)},\tup{[\phi_j]}}$, where $[\phi_j] \from V/W\to V_j/\phi_j(W)$ are the quotient maps.
From a basis for $V$, bases for $V_j$, a basis for $W$, and corresponding matrix representations of $\phi_j$, it is possible to compute the dimensions of, and bases for, $V/W$ and  $V_j/\phi_j(W)$, via row operations on matrices.
According to Lemma~\ref{lemma:factorization}, $t\in\PP\tup{V,\tup{V_j},\tup{\phi_j}}$ if and only if $t\in\PP\tup{W,\tup{\phi_j(W)},\tup{\restr{\phi_j}{W}}}\cap\PP\tup{V/W,\tup{V_j/\phi_j(W)},\tup{[\phi_j]}}$.

Because $0 \subpr W \subpr V$, both $W,V/W$ have dimensions strictly less than the dimension of $V$. 
Therefore by Theorem~\ref{thm:decision} and the induction scheme, there exists an algorithm which computes both a finite list of inequalities characterizing $\PP\tup{W,\tup{\phi_j(W)},\tup{\restr{\phi_j}{W}}}$, and a finite list of inequalities characterizing $\PP\tup{V/W,\tup{V_j/\phi_j(W)},\tup{[\phi_j]}}$. 
Testing each of these inequalities on $t$ determines whether $t$ belongs to these two polytopes, hence whether $t$ belongs to $\PP\tup{V,\tup{V_j},\tup{\phi_j}}$. 
\end{proof}

\begin{lemma} \label{lemma:cleanup} 
Let HBL datum $\tup{V,\tup{V_j},\tup{\phi_j}}$ be given. 
Let $i\in\set{1,2,\ldots,m}$.
Let $s\in[0,1]^m$ and suppose that $s_i=1$.
Let $V'=\set{x\in V: \phi_i(x)=0}$ be the nullspace of $\phi_i$.
Define $\widehat{s}\in[0,1]^{m-1}$ to be $(s_1,\ldots,s_m)$ with the $i^\text{th}$ coordinate deleted.
Then $s\in\PP\tup{V,\tup{V_j},\tup{\phi_j}}$ if and only if $\widehat{s}\in\PP\tup{V',\tup{V_j}_{j\ne i},\tup{\restr{\phi_j}{V'}}_{j\ne i}}$.
\end{lemma}
\begin{proof}
For any subspace $W \subim V'$, since $\Dim{\phi_i(W)}=0$,
\[ \sum_j s_j \Dim{\phi_j(W)} = \sum_{j\ne i} s_j \Dim{\phi_j(W)}. \]
So if $s\in\PP\tup{V,\tup{V_j},\tup{\phi_j}}$ then $\widehat{s}\in\PP\tup{V',\tup{V_j}_{j\ne i},\tup{\restr{\phi_j}{V'}}_{j\ne i}}$.

Conversely, suppose that $\widehat{s}\in\PP\tup{V',\tup{V_j}_{j\ne i},\tup{\restr{\phi_j}{V'}}_{j\ne i}}$.
Let $W$ be any subspace of $V$. 
Write $W=W''+(W\cap V')$ where the subspace $W'' \subim V$ is a supplement to $W \cap V'$ in $W$, so that $\Dim{W} = \Dim{W''} + \Dim{W\cap V'}$.
Then
\begin{align*}
\sum_j s_j \Dim{\phi_j(W)} &= \Dim{\phi_i(W)} + \sum_{j\ne i}s_j \Dim{\phi_j(W)}
\\ &\ge  \Dim{\phi_i(W'')} + \sum_{j\ne i}s_j \Dim{\phi_j(W\cap V')}
\\ &\ge  \Dim{W''} + \Dim{W\cap V'};
\end{align*}
$\Dim{\phi_i(W'')} = \Dim{W''}$ because $\phi_i$ is injective on $W''$.
So $s\in\PP\tup{V,\tup{V_j},\tup{\phi_j}}$. 
\end{proof}

To prepare for the proof of Theorem~\ref{thm:decision}, let $\PP\tup{V,\tup{V_j},\tup{\phi_j}}$ be given.
Let $\tup{W_1,W_2,W_3,\ldots}$ be the list of subspaces of $V$ produced by the algorithm of Lemma~\ref{lemma:enumeratesubspaces}. 
Let $N\ge 1$.
To each index $\alpha \in\set{1,2,\ldots,N}$ is associated a linear inequality $\sum_{j=1}^m s_j\Dim{\phi_j(W_\alpha)} \ge \Dim{W_\alpha}$ for elements $s\in[0,1]^m$, which we encode by an $(m+1)$--tuple $\tup{v(W_\alpha),c(W_\alpha)}$; the inequality is $\langle s,v(W_\alpha)\rangle\ge c(W_\alpha)$.
Define $\PP_N\subseteq[0,1]^m$ to be the polytope defined by this set of inequalities.

\begin{lemma} \label{lemma:musthalt}
\begin{equation} \label{tautologicalinclusion}
\PP_N \supseteq \PP\dtup{V,\dtup{V_j},\dtup{\phi_j}} \qquad \text{for all $N$}. 
\end{equation}
Moreover, there exists a positive integer $N$ such that $\PP_M=\PP\tup{V,\tup{V_j},\tup{\phi_j}}$ for all $M\ge N$.
\end{lemma}
\begin{proof}
The inclusion holds for every $N$, because the set of inequalities defining $\PP_N$ is a subset of the set defining $\PP\tup{V,\tup{V_j},\tup{\phi_j}}$. 

$\PP\tup{V,\tup{V_j},\tup{\phi_j}}$ is specified by some finite set of inequalities, each specified by some subspace of $V$.
Choose one such subspace for each of these inequalities.
Since $\tup{W_\alpha}$ is a list of all subspaces of $V$, there exists $M$ such that each of these chosen subspaces belongs to $\tup{W_\alpha: \alpha\le M}$. 
\end{proof}

\begin{lemma}
Let $m\ge 2$.
If $s$ is an extreme point of $\PP_N$, then either $s_j\in\set{0,1}$ for some $j\in\set{1,2,\ldots,m}$, or there exists $\alpha\in\set{1,2,\ldots,N}$ for which $W_\alpha$ is critical with respect to $s$ and $0<\Dim{W_\alpha}<\Dim{V}$.
\end{lemma}
In the following argument, we say that two inequalities $\langle s,v(W_\alpha)\rangle\ge c(W_\alpha)$, $\langle s,v(W_\beta)\rangle\ge c(W_\beta)$ are distinct if they specify different subsets of $\RR^m$.
\begin{proof}
For any extreme point $s$, equality must hold in at least $m$ genuinely distinct inequalities among those defining $\PP_N$.
These inequalities are of three kinds: $\langle s,v(W_\alpha)\rangle\ge c(W_\alpha)$ for $\alpha\in\set{1,2,\ldots,N}$, $s_j\ge 0$, and $-s_j\ge -1$, with $j\in\set{1,2,\ldots,m}$.
If $W_\beta=\set{0}$ then $W_\beta$ specifies the tautologous inequality $\sum_j s_j\cdot 0=0$, so that index $\beta$ can be disregarded.

If none of the coordinates $s_j$ are equal to $0$ or $1$, there must exist $\beta$ such that equality holds in at least two distinct inequalities $\langle s,v(W_\beta)\rangle \ge  c(W_\beta)$ associated to subspaces $W_\alpha$ among those which are used to define $\PP_N$.
We have already discarded the subspace $\set{0}$, so there must exist $\beta$ such that $W_\beta$ and $V$ specify distinct inequalities.
Thus $0<\Dim{W_\beta}<\Dim{V}$.
\end{proof}

\begin{proof}[Proof of Theorem~\ref{thm:decision}]
Set $\PP=\PP\tup{V,\tup{V_j},\tup{\phi_j}}$.
Consider first the base case $m=1$. 
The datum is a pair of finite-dimensional vector spaces $V,V_1$ with 
% an $\FF$--linear 
a $\QQ$--linear map $\phi\from V\to V_1$.
The polytope $\PP$ is the set of all $s\in[0,1]$ for which $s\Dim{\phi(W)}\ge \Dim{W}$ for every subspace $W\subim V$. 
If $\Dim{V}=0$ then $\PP\tup{V,\tup{V_j},\tup{\phi_j}} = [0,1]$.
If $\Dim{V}>0$ then since $\Dim{\phi(W)} \le \Dim{W}$ for every subspace, the inequality can only hold if the nullspace of $\phi$ has dimension $0$, and then only for $s=1$.
The nullspace of $\phi$ can be computed.
So  $\PP$ can be computed when $m=1$.

Suppose that $m\ge 2$.
Let $\tup{V,\tup{V_j},\tup{\phi_j}}$ be given.
Let $N=0$. 
Recursively apply the following procedure.

Replace $N$ by $N+1$. 
Consider $\PP_N$. 
Apply Lemma~\ref{lemma:enumeratevertices} to obtain a list of all extreme points $\tau$ of $\PP_N$, and for each such $\tau$ which belongs to $(0,1)^m$, a nonzero proper subspace $W(\tau)\subim V$ which is critical with respect to $\tau$. 

Examine each of these extreme points $\tau$, to determine whether $\tau\in\PP\tup{V,\tup{V_j},\tup{\phi_j}}$.
There are three cases. 
Firstly, if $\tau\in (0,1)^m$, then Proposition~\ref{prop:subalg} may be invoked, using the critical subspace $W(\tau)$, to determine whether $\tau\in \PP$.

Secondly, if some component $\tau_i$ of $\tau$ equals $1$, let $V'$ be the nullspace of $\phi_i$.
Set 
\[
\PP'=\PP\dtup{V',\dtup{V_j}_{j\ne i},\dtup{\restr{\phi_j}{V'}}_{j\ne i}}.
\]
According to Lemma~\ref{lemma:cleanup}, $\tau\in\PP$ if and only if $\widehat{\tau} = \tup{\tau_j}_{j\ne i}\in\PP'$. 
This polytope $\PP'$ can be computed by the induction hypothesis, since the number of indices $j$ has been reduced by one.

Finally, if some component $\tau_i$ of $\tau$ equals $0$, then because the term $s_i\Dim{\phi_i(W)}=0$ contributes nothing to sums $\sum_{j=1}^m s_j \Dim{\phi_j(W)}$, $\tau\in \PP$ if and only if $\widehat{\tau}$ belongs to $\PP\tup{V,\tup{V_j}_{j\ne i},\tup{\phi_j}_{j\ne i}}$.
To determine whether $\widehat{\tau}$ belongs to this polytope requires again only an application of the induction hypothesis.

If every extreme point $\tau$ of $\PP_N$ belongs to $\PP$, then because $\PP_N$ is the convex hull of its extreme points, $\PP_N\subseteq \PP$.
The converse inclusion holds for every $N$, so in this case $\PP_N= \PP$.
The algorithm halts, and returns the conclusion that $\PP =\PP_N$, along with information already computed: a list of the inequalities specified by all the subspaces  $W_1,\ldots,W_N$, and a list of extreme points of $\PP_N = \PP$.

On the other hand, if at least one extreme point of $\PP_N$ fails to belong to $\PP$, then $\PP_N\ne\PP$.
Then increment $N$ by one, and repeat the above steps. 

Lemma~\ref{lemma:musthalt} guarantees that this procedure will halt after finitely many steps. 
\end{proof}

%%
%The procedure to check the constraints calls the same algorithm recursively, with a maximum depth of $d$; based on this, it is stated that the recursion completes in ``a finite number of steps.'' 
%%
%We interpret this to mean a finite number of recursive calls; it is unclear to us that the algorithm terminates in general.
%%
%In the continuous case (see \sectn{\ref{sec:Tarskidecides}}), we suggest an alternative approach to \cite{Valdimarsson07} which always terminates with the (correct) polytope $\PP$ in a finite number of steps. 
%%
%However, we note that if the lattice is finite, or the enumeration of its elements include a sufficient subset of \eqref{subcriticalhypothesis} after a small number of steps, then the procedure in \cite{Valdimarsson07} may indeed terminate faster than our procedure in \sectn{\ref{sec:Tarskidecides}}.
%%
%We believe this result (namely \cite[Theorem~1.8]{Valdimarsson07}) generalizes to the discrete case; Part~2 of this paper will explore the connections between critical subgroups and the lattice generated by the kernels.

\subsection{On Computation of the Constraints Defining $\PP$}
\label{sec5.1}

In order to compute the set of inequalities \eqref{subcriticalhypothesis}, we would like to answer the following question:
Given any group homomorphisms $\phi_1, \ldots, \phi_m$ and integers $0 \leq r,r_1,\ldots, r_m \leq d$, does there exist a subgroup $H \subim \ZZ^d$ such that
\begin{equation*}
\Rank{H} = r, \quad \Rank{\phi_1(H)} = r_1,\quad \ldots, \quad \Rank{\phi_m(H)} = r_m
\end{equation*}
or, in other words, is
\begin{equation*}
r \le \sum_{j=1}^m s_j \cdot r_j
\end{equation*}
one of the inequalities?
Based on the following result, it is not known whether this problem is decidable for general $\tup{\ZZ^d,\tup{\ZZ^{d_j}},\tup{\phi_j}}$. 
\begin{theorem} \label{thm4.1}
There exists an effective algorithm for computing the set of constraints \eqref{subcriticalhypothesis} defining $\PP$ if and only if there exists an effective algorithm to decide whether a system of polynomial equations with rational coefficients has a rational solution. 
\end{theorem}
\noindent The conclusion of Theorem~\ref{thm4.1} is equivalent to a positive answer to Hilbert's Tenth Problem for the rational numbers $\QQ$ (see Definition~\ref{def:HTPQ}), a longstanding open problem.

Let us fix some notation.
\begin{notation}[see, e.g., \cite{lang2002algebra}]
For a natural number $d$ and ring $R$, we write $\Mat{d}{R}$ to denote the ring of $d$--by--$d$ matrices with entries from $R$.
(Note that elsewhere in this work we also use the notation $R^{m\times n}$ to denote the \emph{set} of $m$--by--$n$ matrices with entries from $R$.)
 We identify $\Mat{d}{R}$ with the endomorphism ring of the $R$--module $R^d$ and thus may write elements of $\Mat{d}{R}$ as $R$--linear maps rather than as matrices. 
 Via the usual coordinates, we may identify $\Mat{d}{R}$ with $R^{d^2}$. 
 We write $R[x_1,\ldots,x_q]$ to denote the ring of polynomials over $R$ in variables $x_1,\ldots,x_q$.

Recall we are given $d,d_j\in\NN$ and $\ZZ$--linear maps $\phi_j\from\ZZ^d\to\ZZ^{d_j}$, for $j\in\set{1,2,\ldots,m}$ for some positive integer $m$.
Without loss, we may assume each $d_j=d$, so each $\phi_j$ is an endomorphism of $\ZZ^d$. 
%
%Each subgroup $H \subim \ZZ^d$ corresponds to a unique subspace of $\QQ^d$ by extension of scalars; we denote this subspace by $H_\QQ$.
%
%We have $\Rank{H} = \Dim{H_\QQ}$, where $\Dim{\cdot}$ denotes the dimension of a rational vector space everywhere in this section.
%
Each $\phi_j$ can also be interpreted as a $\QQ$--linear map (from $\QQ^d$ to $\QQ^{d_j}$), represented by the same integer matrix.
\end{notation}

\begin{definition} 
\label{defZtoQ}
Given $m,d \in \NN$, and a finite sequence $r,r_1,\ldots,r_m$ of natural numbers each bounded by $d$, we define the sets
\begin{align*}
E_{d;r,r_1,\ldots,r_m} &\ceq \dset{(\phi_1,\ldots,\phi_m) \in (\Mat{d}{\ZZ})^m : (\exists H \subim \ZZ^d)\; \Rank{H} = r \text{ and } \Rank {\phi_j(H)} = r_j, 1 \le j \le m }, \\
E_{d;r,r_1,\ldots,r_m}^\QQ &\ceq \dset{(\phi_1,\ldots,\phi_m) \in (\Mat{d}{\QQ})^m : (\exists V \subim \QQ^d)\; \Dim{V} = r \text{ and } \Dim{\phi_j(V)} = r_j, 1 \le j \le m }.
\end{align*}
%
%\begin{eqnarray*}
%E_{d;r,r_1,\ldots,r_m} &:=& \{
%\end{eqnarray*}
%%
%Working instead with rational vector spaces we define
%\begin{eqnarray*}
%E_{d;r,r_1,\ldots,r_m}^\QQ &:=& \{ (\phi_1,\ldots,\phi_m) \in (\Mat{d}{\QQ})^m ~:~ (\exists V < \QQ^d \text{ a vector subspace}) 
% \dim(V) = r \text{ and } \\ 
%&& \dim (\phi_j(V)) = r_j \text{ for } 1 \leq j \leq m \} \text{ .}
%\end{eqnarray*}
%%
%We will also consider real vector spaces (Theorem~\ref{thm:decidableoverR}):
%\begin{eqnarray*}
%E_{d;r,r_1,\ldots,r_m}^\RR &:=& \{ (\phi_1,\ldots,\phi_m) \in (\Mat{d}{\RR})^m ~:~ (\exists V < \RR^d \text{ a vector subspace}) 
% \dim_\RR(V) = r \text{ and } \\ 
%&& \dim_\RR (\phi_j(V)) = r_j \text{ for } 1 \leq j \leq m \} \text{ .}
%\end{eqnarray*}
%
\end{definition}

\begin{remark} \label{rmk:minorpolynomials}
The question of whether a given $m$--tuple $(\phi_1,\ldots,\phi_m) \in (\Mat{d}{R})^m$ is a member of $E_{d;r,r_1,\ldots,r_m}$ (when $R=\ZZ$) or $E_{d;r,r_1,\ldots,r_m}^\QQ$ (when $R=\QQ$) is an instance of the problem of whether some system of polynomial equations has a solution over the ring $R$.
We let $B$ be a $d$--by--$r$ matrix of variables, and construct a system of polynomial equations in the $dr$ unknown entries of $B$ and $md^2$ known entries of $\phi_1,\ldots,\phi_m$ that has a solution if and only if the aforementioned rank (or dimension) conditions are met.
The condition $\Rank{M} = s$ for a matrix $M$ is equivalent to all $(s+1)$--by--$(s+1)$ minors of $M$ equaling zero (i.e., the sum of their squares equaling zero), and at least one $s$--by--$s$ minor being nonzero (i.e., the sum of their squares not equaling zero --- see Remark~\ref{rmk:5.1}). 
We construct two polynomial equations in this manner for $M=B$ (with $s=r$) and for each matrix $M=\phi_j B$ (with $s=r_j$).
\end{remark}

\begin{lemma} 
\label{rationalgroup}
With the notation as in Definition~\ref{defZtoQ}, $E_{d;r,r_1,\ldots,r_m} = E_{d;r,r_1,\ldots,r_m}^\QQ \cap (\Mat{d}{\ZZ})^m$.
\end{lemma}
\begin{proof}
This result was already established in Lemma~\ref{lem:PZequalsPQ}; we restate it here using the present notation.
For the left-to-right inclusion, observe that if $H \subim \ZZ^d$ witnesses that $(\phi_1,\ldots,\phi_m) \in E_{d;s,r_1,\ldots,r_m}$, then $H_\QQ$ witnesses that $(\phi_1,\ldots,\phi_m) \in E_{d;r,r_1,\ldots,r_m}^\QQ$.
For the other inclusion, if  $(\phi_1,\ldots,\phi_m) \in E_{d;r,r_1,\ldots,r_m}^\QQ \cap (\Mat{d}{\ZZ})^m$ witnessed by $V \subim \QQ^d$, then we may find a subgroup $H = V\cap\ZZ^d$ of $\ZZ^d$, with $\Rank{H}=\Dim{V}$.
Then $\Dim{\phi_j(V)}  = \Rank{\phi_j(H)} = r_j$ showing that $(\phi_1,\ldots,\phi_m) \in E_{d;r,r_1,\ldots,r_m}$. 
\end{proof}

\begin{definition} \label{def:HTPQ}
\emph{Hilbert's Tenth Problem for $\QQ$} is the question of whether there is an algorithm which given a finite set of polynomials $f_1(x_1,\ldots,x_q),\ldots, f_p(x_1,\ldots,x_q) \in \QQ[x_1,\ldots,x_q]$ (correctly) determines whether or not there is some $a \in \QQ^q$ for which $f_1(a) = \cdots = f_p(a) = 0$.
\end{definition}

\begin{remark} \label{rmk:5.1}
One may modify the presentation of Hilbert's Tenth Problem for $\QQ$ in various ways without affecting its truth value.  
For example, one may allow a condition of the form $g(a) \neq 0$ as this is equivalent to $(\exists b) (g(a)b - 1 = 0)$.  
On the other hand, using the fact that $x^2 + y^2 = 0 \Longleftrightarrow x = 0 = y$, one may replace the finite sequence of polynomial equations with a single equality (see also Remark~\ref{rmk:4.3}).
\end{remark}

\begin{remark}
Hilbert's Tenth Problem, proper, asks for an algorithm to determine solvability in integers of finite systems of equations over $\ZZ$.
 From such an algorithm one could positively resolve Hilbert's Tenth Problem for $\QQ$.  
 However, by the celebrated theorem of Matiyasevich-Davis-Putnam-Robinson~\cite{Matiyasevich93}, no such algorithm exists.   
 The problem for the rationals remains open.  
 The most natural approach would be to reduce from the problem over $\QQ$ to the problem over $\ZZ$, say, by showing that $\ZZ$ may be defined by an expression of the
form
$$
a \in \ZZ \Longleftrightarrow (\exists y_1) \cdots (\exists y_q) P(a;y_1,\ldots,y_q) = 0
$$
 for some fixed polynomial $P$.
 K\"{o}nigsmann~\cite{Kon} has shown that there is in fact a \emph{universal} definition of $\ZZ$ in $\QQ$, that is, a formula of the form
$$
a \in \ZZ \Longleftrightarrow (\forall y_1) \cdots (\forall y_q) \theta(a;y_1,\ldots,y_q) = 0
$$
 where $\theta$ is a finite Boolean combination of polynomial equations, but he also demonstrated that the existence of an existential
definition of $\ZZ$ would violate the Bombieri-Lang conjecture. 
K\"{o}nigsmann's result shows that it is unlikely that Hilbert's Tenth Problem for $\QQ$ can be resolved by reducing to the problem over $\ZZ$ using an existential definition of $\ZZ$ in $\QQ$. 
However, it is conceivable that this problem could be resolved without such a definition.
\end{remark}

\begin{proof}[Proof of Theorem~\ref{thm4.1} (necessity)]
%\begin{proposition} \label{prop4.2}
Evidently, if Hilbert's Tenth Problem for $\QQ$ has a positive solution, then there is an algorithm to (correctly) determine for $d \in \NN$, $r, r_1, \ldots, r_m \leq d$ also in $\NN$, and $(\phi_1,\ldots,\phi_m) \in (\Mat{d}{\ZZ})^m$ whether $(\phi_1, \ldots, \phi_m) \in  E_{d;r,r_1,\ldots,r_m}$.
%\end{proposition}
%
%\begin{proof}
By Lemma~\ref{rationalgroup}, $(\phi_1,\ldots,\phi_m) \in E_{d;r,r_1,\ldots,r_m}$ just in case $(\phi_1,\ldots,\phi_m) \in E_{d;r,r_1,\ldots,r_m}^\QQ$.  
This last condition leads to an instance of Hilbert's Tenth Problem (for $\QQ$) for the set of rational polynomial equations given in Remark~\ref{rmk:minorpolynomials}.
\end{proof}

\begin{notation}
Given a set $S \subseteq \QQ[x_1,\ldots,x_q]$ of polynomials, we denote the set of rational solutions to the equations $f = 0$ as $f$ ranges through $S$ by
$$
V(S)(\QQ) \ceq \dset{ a \in \QQ^q : (\forall f \in S) f(a) = 0 }.
$$
\end{notation}
\begin{remark} \label{rmk:4.3}
Since $V(S)(\QQ)$ is an algebraic set over a field, Hilbert's basis theorem shows that we may always take $S$ to be finite. 
Then by replacing $S$ with $S' \ceq \set{ \sum_{f \in S} f^2 }$, one sees that $S$ may be assumed to consist of a single polynomial.  
While the reduction to finite $S$ is relevant to our argument, the reduction to a single equation is not.
\end{remark}

\begin{definition}
For any natural number $t$, we say that the set $D \subseteq \QQ^t$ is \emph{Diophantine} if there is some $q-t \in \NN$ and a set 
$S \subseteq \QQ[x_1,\ldots,x_t;y_1,\ldots,y_{q-t}]$ for which 
$$
D = \dset{ a \in \QQ^t : (\exists b \in \QQ^{q-t})  (a;b) \in V(S)(\QQ) }.
$$
\end{definition}
\noindent We will show sufficiency in Theorem~\ref{thm4.1} by establishing a stronger result: namely, that an algorithm to decide membership in sets of the form $E_{d;r,r_1,\ldots,r_m}$ could also be used to decide membership in any Diophantine set. 
(Hilbert's Tenth Problem for $\QQ$ concerns membership in the specific Diophantine set $V(S)(\QQ)$.)

With the next lemma, we use a standard trick of replacing composite terms with single applications of the basic operations to put a general Diophantine set in a standard form (see, e.g., \cite{Vak}).
\begin{lemma}\label{lem:basicset}
Given any finite set of polynomials $S \subset \QQ[x_1,\ldots,x_q]$, let $d \ceq \max_{f\in S} \max_{i=1}^q \deg_{x_i}(f)$ and $\mathcal{D} \ceq \set{0,1,\ldots,d}^q$.
There is another set of variables $\set{u_\alpha}_{\alpha \in \mathcal{D}}$ and another finite set $S' \subset \QQ [ \set{u_\alpha}_{\alpha \in \mathcal{D}} ]$ consisting entirely of affine polynomials (polynomials of the form $c + \sum c_\alpha u_\alpha$ where not all $c$ and $c_\alpha$ are zero) and polynomials of the form $u_\alpha u_\beta - u_\gamma$ with $\alpha$, $\beta$, and $\gamma$ distinct, so that $V(S)(\QQ) = \pi (V(S')(\QQ))$ where $\pi\from\QQ^{\mathcal{D}} \to \QQ^q$ is given by 
$$
\dtup{u_\alpha}_{\alpha \in \mathcal{D}} \mapsto \dtup{u_{(1,0,\ldots,0)},u_{(0,1,0,\ldots,0)}, \ldots, u_{(0,\ldots,0,1)}}.
$$ 
\end{lemma} 
\begin{proof}
Let $T \subset \QQ [ \set{ u_\alpha}_{\alpha \in \mathcal{D}} ]$ consist of
\begin{itemize}
\item $u_{(0,\ldots,0)} - 1$ and 
\item $u_{\alpha + \beta} - u_\alpha u_\beta$ for $(\alpha + \beta) \in \mathcal{D}$, $\alpha \neq {\mathbf 0}$ and $\beta \neq {\mathbf 0}$. 
\end{itemize}
Define $\chi\from\QQ^q \to \QQ^{\mathcal{D}}$ by 
$$
\dtup{x_1,\ldots,x_q} \mapsto \dtup{x^\alpha}_{\alpha \in \mathcal{D}}
$$
where $x^\alpha \ceq \prod_{j=1}^q x_j^{\alpha_j}$.
One sees immediately that $\chi$ induces a bijection $\chi\from\QQ^q \to V(T)(\QQ)$ with inverse $\restr{\pi}{V(T)(\QQ)}$.

Let $S'$ be the set containing $T$ and the polynomials $\sum_{\alpha \in \mathcal{D}} c_\alpha u_\alpha$ for which $\sum_{\alpha \in \mathcal{D}} c_\alpha x^\alpha \in S$.   

One checks that if $a \in \QQ^q$, then $\chi(a) \in V(S')(\QQ)$ if and only if $a \in V(S)(\QQ)$.  
Applying $\pi$, and noting that $\pi$ is the inverse to $\chi$ on $V(T)(\QQ)$, the result follows.
\end{proof}
\begin{notation}
For the remainder of this argument, we call a set enjoying the properties identified for $S'$ (namely that each polynomial is either affine or of the form $u_\alpha u_\beta - u_\gamma$) a \emph{basic set}.
\end{notation}

%Our goal is to show that for any Diophantine set $D \subseteq \QQ^t$ there is is a sequence of natural numbers $d;r,r_1,\ldots,r_m$ and a computable function $f\from\QQ^t \to (\Mat{d}{\ZZ})^m$ so that for $a \in \QQ^t$ one has 
%$$ a \in D \Longleftrightarrow f(a) \in E_{d;r,r_1,\ldots,r_m}. $$
% Once this result is established, we conclude that if there were an algorithm to decide membership in sets of the form $E_{d;r,r_1,\ldots,r_m}$, then Hilbert's Tenth Problem for $\QQ$ would have a positive solution.

%\begin{theorem} \label{thm4.1}
\begin{proof}[Proof of Theorem~\ref{thm4.1} (sufficiency)]
It follows from Remark~\ref{rmk:4.3} and Lemma~\ref{lem:basicset} that the membership problem for a general Diophantine set may be reduced to the membership problem for a Diophantine set defined by a finite basic set of equations.
Let  $S \subseteq \QQ[x_1,\ldots,x_q]$ be a finite basic set of equations and let $t \leq q$ be some natural number, 
We now show that there are natural numbers $\mu,\nu,\rho,\rho_1, \ldots, \rho_\mu$ and a computable function $f\from\QQ^t \to (\Mat{\nu}{\ZZ})^\mu$ so that for $a \in \QQ^t$ one has that there is some $b \in \QQ^{q-t}$ with $(a,b) \in V(S)(\QQ)$ if and only if $f(a) \in E_{\nu;\rho,\rho_1,\ldots,\rho_\mu}$.

List the $\ell$ affine polynomials in $S$ as
$$
\lambda_{0,1} + \sum_{i=1}^q \lambda_{i,1} x_i, \qquad \ldots,\qquad \lambda_{0,\ell} + \sum_{i=1}^q \lambda_{i,\ell} x_i
$$
and the $k$ polynomials expressing multiplicative relations in $S$ as
$$
x_{i_{1,1}} x_{i_{2,1}} - x_{i_{3,1}},\qquad \ldots,\qquad x_{i_{1,k}} x_{i_{2,k}} - x_{i_{3,k}}.
$$
Note that by scaling, we may assume that all of the coefficients $\lambda_{i,j}$ are integers. 
  
We shall take $\mu \ceq 4 + q + t + |S|$, $\nu \ceq 2q+2$, $\rho \ceq 2$ and the sequence $\rho_1, \ldots, \rho_\mu$ to consist of $4 + q + k$ ones followed by $t + \ell$ zeros. 
 Let us describe the map $f\from\QQ^t \to (\Mat{\nu}{\ZZ})^\mu$ by expressing each coordinate.  
 For the sake of notation, our coordinates on $\QQ^\nu$ are $(u;v) \ceq (u_0,u_1\ldots, u_q; v_0, v_1, \ldots, v_q)$.
\begin{itemize}
\item[A.] The map $f_1$ is constant taking the value $(u;v) \mapsto (u_0,0,\ldots,0)$.
\item[B.] The map $f_2$ is constant taking the value $(u;v) \mapsto (v_0,0,\ldots,0)$.
\item[C.] The map $f_3$ is constant taking the value $(u;v) \mapsto (u_0,\ldots,u_q;0,\ldots,0)$.
\item[D.] The map $f_4$ is constant taking the value $(u;v) \mapsto (v_0,\ldots,v_q;0,\ldots,0)$.
\item[E.] The map $f_{4+j}$ (for $0 < j \leq q$) is constant taking the value $(u;v) \mapsto (u_0 - v_0,u_j - v_j,0,\ldots,0)$.
\item[F.] The map $f_{4+q+j}$ (for $0 < j \leq k$) is constant taking the value $(u;v) \mapsto (u_{i_{1,j}} + v_0,u_{i_{3,j}} + v_{i_{2,j}},0,\ldots,0)$.
\item[G.] The map $f_{4 + q + k + j}$ (for $0 < j \leq t$) takes $a = (\frac{p_1}{q_1}, \ldots, \frac{p_t}{q_t})$ (written
in lowest terms) to the linear map $(u;v) \mapsto (p_j u_0 -  q_j u_j,0,\ldots,0)$.
\item[H.] The map $f_{4 + q + k + t + j}$ (for $0 < j \leq \ell$) takes the value $(u;v) \mapsto (\sum_{i=0}^q \lambda_{i,j} u_i,0,\ldots,0)$.
\end{itemize}
Note that only the components $f_{4 + q + k + j}$ for $0 < j \leq t$ actually depend on $(a_1,\ldots,a_t) \in \QQ^t$.

Let us check that this construction works.  
First, suppose that $a \in \QQ^t$ and that there is some $b \in \QQ^{q-t}$  for which $(a,b) \in V(S)(\QQ)$. 
For the sake of notation, we write $c = (c_1,\ldots,c_q) \ceq (a_1,\ldots,a_t,b_1,\ldots,b_{q-t})$.  

Let $V \ceq \QQ (1,c_1,\ldots,c_q,0,\ldots,0) + \QQ (0,\ldots,0,1,c_1,\ldots,c_q)$.  
We will check now that $V$ witnesses that $f(a) \in E_{\nu,\rho,\rho_1,\ldots,\rho_\mu}$.   
Note that a general element of $V$ takes the form $(\alpha, \alpha c_1, \ldots, \alpha c_q,\beta,\beta c_1,\ldots,\beta c_q)$ for $(\alpha,\beta) \in \QQ^2$.  
Throughout the rest of this proof, when we speak of a general element of some image of $V$, we shall write $\alpha$ and $\beta$ as variables over $\QQ$.

Visibly, $\Dim{V}=2=\rho$.

Clearly, $f_1(a)(V) = \QQ (1,0,\ldots,0) = f_2(a)(V)$, so that $\rho_1 = \Dim{f_1(a)(V)} = 1 = \Dim{f_2(a)(V)} = \rho_2$ as required.

Likewise, $f_3(a)(V) = \QQ(1,c_1,\ldots,c_q,0,\ldots,0) = f_4(a)(V)$, so that $\rho_3 = \Dim{f_3(a)(V)} = 1 = \Dim{f_4(a)(V)} = \rho_4$.

For $0 < j \leq q$ the general element of $f_{4+j}(a)(V)$ has the form $(\alpha - \beta, \alpha c_j - \beta c_j,0,\ldots,0) = (\alpha - \beta) (1,c_j,0,\ldots,0)$. Thus, $f_{4+j}(a)(V) = \QQ (1,c_j,0,\ldots,0)$ has dimension $\rho_{4+j}=1$.
 
For $0 < j \leq k$ we have  $c_{i_{3,j}} = c_{i_{1,j}} c_{i_{2,j}}$, the general element of $f_{4+q+j}(a)(V)$ has the form $(\alpha c_{i_{1,j}} + \beta,\alpha c_{i_{3,j}} + \beta c_{i_{2,j}},0,\ldots,0)  = (\alpha c_{i_{1,j}} + \beta, \alpha c_{i_{1,j}} c_{i_{2,j}} + \beta c_{i_{2,j}},0,\ldots,0) = (\alpha c_{i_{1,j}} + \beta) (1,c_{i_{2,j}},0,\ldots,0)$ so we have that $\rho_{4+q+j}=\Dim{f_{4+q+j}(a)(V)} = 1$.
 
 For $0 < j \leq t$, the general element of $f_{4 + q + k + j}(a)(V)$ has the form $(p_j \alpha - q_j \alpha a_j,0,\ldots,0) = {\mathbf 0}$.  
 That is, $\rho_{4+q+k+j} = \Dim{f_{4+q+k+j}(a)(V)} = 0$.
 
 Finally, if $0 < j \leq \ell$, then the general element of $f_{4+q+k+t+j}(a)(V)$ has the form $(\lambda_{0,j} \alpha + \sum_{i=1}^q \lambda_{i,j} \alpha c_i, 0, \ldots, 0) = {\mathbf 0}$ since $\lambda_{0,j} + \sum_{i=1}^q \lambda_{i,j} c_i = 0$. 
 So $\rho_{4+q+k+t+j} = \Dim{f_{4+q+k+t+j}(a)(V)}=0$.
 
 Thus, we have verified that if $a \in \QQ^t$ and there is some $b \in \QQ^{d-t}$ with $(a,b) \in V(S)(\QQ)$, then $f(a) \in E_{\nu;\rho,\rho_1,\ldots,\rho_\mu}$.
 
 Conversely, suppose that $a = (\frac{p_1}{q_1}, \ldots, \frac{p_t}{q_t}) \in \QQ^t$ and that $f(a) \in E_{\nu;\rho,\rho_1,\ldots,\rho_\mu}$, with the same $\mu,\nu,\rho,\rho_1,\ldots,\rho_\mu$. 
 Let  $V \subim \QQ^\nu$ have $\Dim{V} = \rho$ witnessing that $f(a) \in E_{\nu;\rho,\rho_1,\ldots,\rho_\mu}$.
 
 \begin{lemma} \label{claim1}
There are elements $g$ and $h$ in $V$ for which $g = (g_0,\ldots,g_q;0,\ldots,0)$, $h = (0,\ldots,0;h_0,\ldots,h_q)$, $g_0 = h_0 = 1$, and $\QQ g + \QQ h = V$.
\end{lemma}
\begin{proof}
 The following is an implementation of row reduction.
  Let the elements $d = (d_0,\ldots,d_q;d_0',\ldots,d_q')$ and $e = (e_0,\ldots,e_q;e_0',\ldots,e_q')$ be a basis for $V$.
 Since $\Dim{f_1(a)(V)} = 1$, at the cost of reversing $d$ and $e$ and multiplying by a scalar, we may assume that $d_0 = 1$.  
 Since $\Dim{f_3(a)(V)} = 1$, we may find a scalar $\gamma$ for which $(\gamma d_0, \ldots, \gamma d_q) = (e_0, \ldots, e_q)$.  
 Set $\widetilde{g} \ceq e - \gamma d$.    
 Write $\widetilde{g} = (0,\ldots,0,\widetilde{g}_0,\ldots,\widetilde{g}_q)$.
 Since $\widetilde{g}$ is linearly independent from $e$ and $\Dim{f_4(a)(V)} = 1$, we see that there is some scalar $\delta$ for which $(\delta \widetilde{g}_0, \ldots, \delta \widetilde{g}_q) = (d_0' , \ldots, d_q')$.  
 Set $h \ceq  d - \delta \widetilde{g}$.   
 Using the fact that $\Dim{f_2(a)(V)} = 1$ we see that $\widetilde{g}_0 \neq 0$.  
 Set $g \ceq \widetilde{g}_0^{-1} \widetilde{g}$. 
 \end{proof}
\begin{lemma}
 For $0 \leq j \leq q$ we have $g_j = h_j$.
\end{lemma}
\begin{proof}
We arranged $g_0 = 1 = h_0$ in Lemma~\ref{claim1}.
The general element of $V$ has the form $\alpha h + \beta g$ for some $(\alpha,\beta) \in \QQ^2$.
For $0 < j \leq q$, the general element of $f_{4+j}(a)(V)$ has the form $(\alpha h_0 - \beta h_0,\alpha h_j - \beta g_j, 0, \ldots, 0) = ( (\alpha - \beta) h_0, (\alpha - \beta) h_j + \beta (h_j - g_j), 0, \ldots, 0)$. 
Since $h_0 \neq 0$, if $h_j \neq g_j$, then this vector space would have dimension two, contrary to the requirement that $f(a) \in E_{\nu;\rho,\rho_1,\ldots,\rho_\mu}$. 
\end{proof}
\begin{lemma}
For $0 < j \leq t$ we have $h_j = a_j$.
\end{lemma}
\begin{proof}
The image of $\alpha h + \beta g$ under $f_{4 + q + k + j}(a)$ is $(p_j \alpha  - q_j \alpha h_j,0,\ldots,0)$ where $a_j = \frac{p_j}{q_j}$ in lowest terms.  
Since $\Dim{f_{4+q+k+j}(a)(V)} = 0$, we have $q_j h_j = p_j$.  
That is, $h_j = a_j$. 
\end{proof}
\begin{lemma}
For any $F \in S$, we have $F(h_1,\ldots,h_q) = 0$. 
\end{lemma}
\begin{proof}
If $F$ is an affine polynomial, that is, if $F = \lambda_{0,j} + \sum_{i=1}^q \lambda_{i,j} x_i$ for some $0 < j \leq \ell$, then because $\Dim{f_{4 + q + k + t + j}(a)(V)} = 0$, we have $\lambda_{0,j} + \sum_{i=1}^q \lambda_{i,j} h_i = 0$.    
On the other hand, if $F$ is a multiplicative relation, that is, if $F = x_{i_{1,j}} x_{i_{2,j}} - x_{i_{3,j}}$ for some $0< j \leq k$, then because $\Dim{f_{4+q+j}(a)(V)} = 1$ we see that there is some scalar $\gamma$ so that for any pair $(\alpha,\beta)$ we have $\gamma (\alpha h_{i_{1,j}} + \beta) = \alpha h_{i_{3,j}} + \beta h_{i_{2,j}}$.  
Specializing to $\beta = 0$ and $\alpha = 1$, we have $\gamma = \frac{h_{i_{3,j}}}{h_{i_{1,j}}}$ (unless both are zero, in which case the equality $h_{i_{1,j}} h_{i_{2,j}} = h_{i_{3,j}}$ holds anyway), which we substitute to obtain $\frac{h_{i_{3,j}}}{h_{i_{1,j}}} = h_{i_{2,j}}$, or $h_{i_{2,j}} h_{i_{1,j}} = h_{i_{3,j}}$. 
\end{proof}

Taking $b \ceq (h_{t+1},\ldots,h_{q})$ we see that $(a,b) \in V(S)(\QQ)$.  
\end{proof}

\subsection{Computing upper and lower bounds on $\shbl$}
\label{sec6.4}

%For a general program, where the decidability of writing down all the linear
%constraints \eqref{subcriticalhypothesis} determining the polytope $\PP$
%is unknown, it is of interest
%to effectively compute upper and lower bounds on the solution $\shbl$
%of the corresponding linear program.
%We show how to do so here. See also \sectn{\ref{sec:conclusions}} for a list
%of other special cases where we can effectively compute the linear constraints
%(these will appear in Part~2 of this paper).

\subsubsection{Bounding $\shbl$ below: Approximating $\PP$ by a superpolytope via sampling constraints} \label{sec:guessingsubgroups}
While the approach in \sectn{\ref{sec:computeP}} demonstrates that $\PP$ can be computed exactly, it follows from \eqref{tautologicalinclusion} that we can terminate the algorithm at any step $N$ and obtain a superpolytope $\PP_N \supseteq \PP$. 
Thus, the minimum sum $\shbla{N}$ of $s\in\PP_N$ is a lower bound on the desired $\shbl$.
Since our communication lower bound is inversely proportional to $M^{\shbl-1}$, $\shbla{N}$ may lead to an erroneous bound which is larger than the true lower bound. 
But the erroneous bound may still be attainable, possibly leading to a faster algorithm.

We note that, since $\shbla{N}$ is nondecreasing in $N$, the approximation can only improve as we take more steps of the algorithm.
The question remains of how to systematically choose a sequence of subspaces $W_1,\ldots,W_N$ of $V$ that are likely to yield the best estimate of $\shbl$, i.e., the fastest growing $\shbla{N}$. 
Furthermore, even if $\PP_N$ is a proper superpolytope, it may still be that $\shbla{N}=\shbl$, and we could stop the algorithm earlier.
In Part~2 of this work, we will discuss ways to improve the efficiency of the approach in \sectn{\ref{sec:computeP}}, including choosing the subspaces $W_i$ and detecting this early-termination condition (see \sectn{\ref{sec:conclusions}}).

%In the case of general array references, a practical approach to using Theorem~\ref{thm:1} is to simply choose a large number $H_1,\ldots,H_t$ of (carefully chosen) subgroups $H_i \subim \ZZ^d$, and use them as an approximation of hypothesis~\eqref{subcriticalhypothesis}.
%%
%In other words, we choose exponents $s_1,\ldots,s_m$ so that \eqref{subcriticalhypothesis} will be satisfied for $H_1,\ldots,H_t$ but not necessarily all subgroups $H \subim \ZZ^d$.
%%
%But suppose we miss an important (supercritical) subgroup $H$ from our list, for which \eqref{subcriticalhypothesis} is not satisfied.
%%
%What could go wrong?
%
%Recall that $\shbl$ is defined as the minimum possible value of $\sum_{j=1}^m s_j$ over all $s_j$ satisfying \eqref{subcriticalhypothesis}, a linear program. 
%%
%Missing a supercritical subgroup means omitting a constraint from the linear program, which means that minimizing $\shbl$ may result in a smaller value than the true optimum. 
%%

\subsubsection{Bounding $\shbl$ above: Approximating $\PP$ by a subpolytope via embedding into $\RR^d$} \label{sec:Tarskidecides}  
The approximation just presented in \sectn{\ref{sec:guessingsubgroups}} yielded an underestimate of $\shbl$.
Now we discuss a different approximation that yields an overestimate of $\shbl$. 
If the overestimate and underestimate agree, then we have a proof of optimality, and if they are close, it is also useful.

Recall in \sectn{\ref{sec5.1}}, we hoped to answer the following (possibly undecidable) question. 
\begin{quote}
\centering
``Given integers $0 \le r,r_1,\ldots,r_m \le d$, is there a subgroup $H \subim \ZZ^d$ such that $\Rank{H}=r$, $\Rank{\phi_1(H)}=r_1$, \ldots, $\Rank{\phi_m(H)}=r_m$?''
\end{quote}
If such an $H$ exists, we know the linear constraint $r \le \sum_{j=1}^m s_j r_j$ is one of the conditions \eqref{subcriticalhypothesis}, which define the polytope $\PP$ of feasible solutions $(s_1,\ldots,s_m)$. 
We now ask the related question,
\begin{quote}
\centering
``{\ldots}is there a subspace $V \subim \RR^d$ such that $\Dim{V}=r$, $\Dim{\phi_1(V)}=r_1$, \ldots, $\Dim{\phi_m(V)}=r_m$?''
\end{quote}
Here, we reinterpret $\phi_j$ as an $\RR$--linear map.
It turns out that computing the set of inequalities \eqref{subcriticalhypothesisfield} for the HBL datum $\tup{\RR^d,\tup{\RR^{d_j}},\tup{\phi_j}}$ is easier than in the previous case, in the sense of being Tarski-decidable.
Following the notation in \sectn{\ref{sec5.1}}, we define
\[ E_{d;r,r_1,\ldots,r_m}^\RR \ceq \dset{(\phi_1,\ldots,\phi_m) \in (\Mat{d}{\RR})^m : (\exists V \subim \RR^d)\; \Dim{V} = r \text{ and } \Dim{\phi_j(V)} = r_j, 1 \le j \le m }. \]
\begin{theorem} \label{thm:decidableoverR}
It is Tarski-decidable \cite{Tarski-book} to decide membership in the sets $E_{d;r,r_1,\ldots,r_m}^\RR$.
\end{theorem}
\begin{proof}
There is an effective algorithm (the cylindrical algebraic decomposition \cite{CAD}) to decide whether a system of real polynomial equations (e.g., that in Remark~\ref{rmk:minorpolynomials}) has a real solution.
\end{proof}

As opposed to the rational embedding considered in \sectn{\ref{sec5.1}}, this real embedding is a \emph{relaxation} of the original question.
That is, if $H \subim \ZZ^d$ exists, then $V$, the real vector subspace of $\RR^d$ generated by $H$, exists with $\Dim{V} = \Rank{H}$; however, there may be $V$ which do not correspond to any $H \subim \ZZ^d$.
In other words, the existence of $V$ is a necessary but not sufficient condition for the existence of such a subgroup $H \subim \ZZ^d$.
Since $\Dim{\phi_j(V)} = \Rank{\phi_j(H)}$, we see that the set of inequalities \eqref{subcriticalhypothesisfield} for the vector space HBL datum $\tup{\RR^d,\tup{\RR^{d_j}},\tup{\phi_j}}$ is a superset of the inequalities \eqref{subcriticalhypothesis} for the Abelian group HBL datum $\tup{\ZZ^d,\tup{\ZZ^{d_j}},\tup{\phi_j}}$ (or equivalently, for the rational embedding considered in \sectn{\ref{sec5.1}}).
In terms of the polytopes defined by these inequalities, we have
\[ \bar{\PP} \ceq \PP\dtup{\RR^d,\dtup{\RR^{d_j}},\dtup{\phi_j}} \subseteq \PP, \]
i.e., the constraints for the relaxed problem give a polytope $\bar{\PP}$ of feasible solutions $s=(s_1,\ldots,s_m)$, which is a subpolytope (subset) of $\PP$.
%
%
%Unfortunately, since $E_{d;r,r_1,\ldots,r_m} \subseteq E_{d;r,r_1,\ldots,r_m}^\RR$, membership in the latter sets is a necessary but not sufficient condition for membership in the former.
%%
%However, this does suggest an approach to approximate the linear constraints \eqref{subcriticalhypothesis}, which we discuss below in Section~\ref{sec:Tarskidecides}).
%%
%% We will not discuss the real case again this section.
%
% Whether $V$ exists is Tarski-decidable (see Remark~\ref{thm:decidableoverR}).
%
%Since tthe constraints for the relaxed problem give a superset of the constraints for the original problem.
%
If the relaxed constraints are a proper superset, then $\bar{\PP}$ is a proper subpolytope.
This means that the conclusion \eqref{mainconclusion} of Theorem~\ref{thm:1} is still valid for any $s \in \bar{\PP}$, but the upper bound may be too large.
Therefore, the communication lower bound that is inversely proportional to $M^{\barshbl-1}$, where $\barshbl \ceq \min_{s\in \bar{\PP}} \sum_{j=1}^m s_j$, is also still valid. 
However, it may be smaller than the lower bound that is inversely proportional to $M^{\shbl-1}$, where $\shbl \ceq \min_{s\in\PP} \sum_{j=1}^m s_j$.

In other words, it is Tarski-decidable to write down a communication lower bound, but it may be strictly smaller than the best lower bound implied by Theorem~\ref{thm:1}, and obtained by the approach in \sectn{\ref{sec:computeP}}.

\section{Easily computable communication lower bounds} \label{sec:bound}

%The good news is that despite the possible undecidability of determining all the constraints in \eqref{subcriticalhypothesis} (which lead to a communication lower bound), for many practical programs we can explicitly write down an equivalent linear program whose solution gives the same lower bound as the original (possibly undecidable) linear program.

As mentioned in \sectn{\ref{sec:undecidability}}, there may be many \emph{equivalent} linear programs to determine $\shbl$ that may not include the set of inequalities \eqref{subcriticalhypothesis}.
The approach in \sectn{\ref{sec:computeP}} shows that it is always possible to write down one such linear program.
In this section, we discuss three practical cases where we can write down another (equivalent) linear program without applying the approach in \sectn{\ref{sec:computeP}}.

First, in \sectn{\ref{sec6.1}}, we demonstrate some transformations in that let us simplify the input HBL datum $\tup{\ZZ^d,\tup{\ZZ^{d_j}},\tup{\phi_j}}$ without altering $\shbl$.
Then in \sectn{\ref{sec6.2}} we examine two trivial cases where we can immediately write down an equivalent linear program. 
In the first case (\sectn{\ref{sec:bound-nontrivialkernel}}), we show that we can always immediately recognize when the feasible region is empty; therefore no $\shbl$ exists and arbitrary data reuse is possible.
In the second case (\sectn{\ref{sec:bound-injectivemap}}), when one or more $\phi_j$ is an injection, we can immediately determine that $\shbl=1$, that is, no asymptotic data reuse is possible.
Lastly, in \sectn{\ref{sec:bound-prod}}, we present the most general case solved explicitly in this paper, when the subscripts of each array are a subset of the loop indices. 
We apply this to several examples: the direct $N$--body problem, database join, matrix-vector multiplication, and tensor contractions, as well as three examples seen earlier: matrix-matrix multiplication, ``complicated code'' (from \sectn{\ref{sec:intro}}), and computing matrix powers (from \sectn{\ref{sec:lb}}).

%Finally, \sectn{\ref{sec6.4}} deals with the general case of \eqref{subcriticalhypothesis},
%and shows that it is possible to effectively compute upper and lower bounds on the
%minimizing solution $\shbl = \sum_{j=1}^m s_j$ of the linear program,
%that determines the communication lower bound $\Omega(\niters/M^{\shbl-1})$.
%When these upper and lower bounds agree (or are close), then this of course
%tells us (much of) what we want to know.

There are variety of other situations in which we can effectively compute the communication lower bound, without applying the approach in \sectn{\ref{sec:computeP}}.
These are summarized in \sectn{\ref{sec:conclusions}}, and will appear in detail in Part~2 of this paper.

\subsection{Simplifying the linear constraints of \eqref{subcriticalhypothesis}}
\label{sec6.1}

If $\Rank{\phi_j(\ZZ^d)}=0$ (e.g., array $A_j$ is a scalar), we say subscript $\phi_j$ is a trivial map, since it maps all indices $\indx$ to the same subscript. 
The following remark shows that we can ignore trivial maps when computing a communication lower bound, without affecting $\shbl$. 
\begin{remark} \label{rmk:trivialmaps}
Suppose we are given the HBL datum $(\ZZ^d,(\ZZ^{d_j}), (\phi_j))$ which satisfies ${\Rank{\phi_j(\ZZ^d)}=0}$ for indices ${k < j \le m}$.
Then \eqref{subcriticalhypothesis} becomes 
\[ 
\Rank{H} \le \sum_{j=1}^k s_j \Rank{\phi_j(H)}\qquad \text{for all subgroups } H \subim \ZZ^d.
\]
So given feasible exponents $(s_j)_{j=1}^k$ for the HBL datum $(\ZZ^d,(\ZZ^{d_j})_{j=1}^k,(\phi_j)_{j=1}^k)$, we are free to pick $s_j \in [0,1]$ for ${k < j \le m}$.
Towards minimizing $\shbl$, we would, of course, pick $s_j=0$ for these indices.
So, we ignore trivial maps when computing a communication lower bound.
\end{remark}

It is also possible that $A_i$ and $A_j$ have the same subscript $\phi_i = \phi_j$.
The following remark shows that, as with trivial maps, we can ignore such duplicate subscripts, without affecting $\shbl$.
\begin{remark} \label{rmk:duplicatemaps}
Given the HBL datum $(\ZZ^d,(\ZZ^{d_j}), (\phi_j))$, suppose $\Kernel{\phi_k} = \cdots = \Kernel{\phi_m}$.
Then $\phi_i(H) \cong \phi_j(H)$ for any ${k \le i,j \le m}$.
Then the constraint \eqref{subcriticalhypothesis} becomes 
\[ 
\Rank{H} \le \sum_{j=1}^{k-1} s_j \Rank{\phi_j(H)} + \Rank{\phi_k(H)} \sum_{j=k}^m s_j = \sum_{j=1}^k t_j \Rank{\phi_j(H)} \qquad \text{for all subgroups } H \subim \ZZ^d.
\]
So given feasible exponents $t=(t_j)_{j=1}^k$ for the HBL datum  $(\ZZ^d,(\ZZ^{d_j})_{j=1}^k,(\phi_j)_{j=1}^k)$, we have a family of feasible exponents $\set{s = (s_j)_{j=1}^m \in [0,1]^m}$ for the original datum, satisfying $s_j = t_j$ for ${1 \le j < k}$ and $s_k + \cdots + s_m = t_k$. 
Our objective $\shbl$ only depends on the sum of the exponents (i.e., computing $t$ suffices); thus we ignore duplicate maps when computing a communication lower bound.
\end{remark}
So in the remainder of this section, we assume that the group homomorphisms $\phi_1,\ldots,\phi_m$ are distinct and nontrivial.
The arrays $A_j$ however, may overlap in memory addresses; this possibility does not affect our asymptotic lower bound, given our assumption from \sectn{\ref{sec:lb}} that $m$ is a constant, negligible compared to $M$.

\subsection{Trivial Cases}
\label{sec6.2}

There are a few trivial cases where no work is required to find the communication lower bound: 
when the homomorphisms' kernels have a nontrivial intersection, and when at least one homomorphism is an injection.
These cover the cases where $d=1$ or $m=1$.
(When $m=0$ there are no arrays, and when $d=0$ there are no loops, cases we ignore.)

% Of course, there is nothing to say when $m=0$; if this is the case because all maps are trivial 
% (thus ignored), we fall into one of the next two cases, depending on whether or not $d=0$.
% When $d=0$, \eqref{subcriticalhypothesis} holds for any exponents $s \in [0,1]^m$, so $\shbl=0$.
% In practice, we ignore this case, since any reordering of a loop `nest' with one iteration is optimal.
% Since $m \ll M$ by assumption, if operands may reside in cache before/after the program, 
% then no communication may be required---the lower bound is $0$ words.%
% \footnote{It is incorrect to conclude the lower bound $\Omega(MG/F)=\Omega(M)$ via \S\ref{sec:lb}, 
% since an assumption $G \gg F$ fails.}

The following lemma shows that $\shbl=0$ only arises in the trivial $d=0$ case; 
otherwise, when the constraints \eqref{subcriticalhypothesis} are feasible and $\shbl$ exists, 
then $\shbl \ge 1$.
\begin{lemma} \label{lem:sumofexponents}
If $d > 0$ and $s \in [0,1]^m$ satisfies \eqref{subcriticalhypothesis}, then $\sum_{j=1}^m s_j \ge 1$.
\end{lemma}
\begin{proof}
If not, then any nontrivial subgroup is supercritical with respect to $s$. 
\end{proof}

\subsubsection{Infeasible Case} \label{sec:bound-nontrivialkernel}

The following result, a companion to Theorem~\ref{thm:1}, gives a simpler necessary and sufficient condition for there to exist some exponent $s$ which satisfies \eqref{mainconclusion}.
\begin{lemma} \label{lem:nonemptyP}
There exists $s\in[0,1]^m$ for which \eqref{mainconclusion} holds if and only if
\begin{equation}
\bigcap_{j=1}^m \Kernel{\phi_j}=\set{0}.
\end{equation}
\end{lemma}
\begin{proof}
If $\bigcap_j \Kernel{\phi_j}=\set{0}$, then Lemma~\ref{lemma:exponentsallone} asserts that \eqref{mainconclusion} holds with all $s_j$ equal to $1$.
Conversely, if $H=\bigcap_j\Kernel{\phi_j}$ is nontrivial, then for any finite nonempty subset $E\subset H$, $\phi_j(E)=\set{0}$ for every index $j$, so $|\phi_j(E)|=1$. 
Thus the inequality \eqref{mainconclusion} fails to hold for any $E\subset H$ of cardinality $\ge 2$.
\end{proof}
So, if $\bigcap_j \Kernel{\phi_j} \ne \set{0}$, then there is no $\shbl$, and there are arbitrarily large sets $E$ that access at most $m \ll M$ operands; if operands may reside in cache before/after the program's execution, then a communication lower bound of $0$ is attainable (see \sectn{\ref{sec:attain-nontrivialkernel}}).

If $m=1$ or $d=1$, then either we fall into this case, or that of \sectn{\ref{sec:bound-injectivemap}}.

\begin{exAk}[Part~2/4]
Recall the `Original Code' for computing $B=A^k$ (for odd $k$) from part~1/4 of this example. 
The lines $C = A \cdot B$ and $B = A \cdot C$ are matrix multiplications, i.e., each is 3 nested loops (over $i_2,i_3,i_4$). 
As discussed in part~1/4, we ignore the second matrix multiplication, leaving only $C = A \cdot B$.
This yields 
$\phi_A(i_1,i_2,i_3,i_4)=(i_2,i_4)$,
$\phi_B(i_1,i_2,i_3,i_4)=(i_3,i_4)$, and
$\phi_C(i_1,i_2,i_3,i_4)=(i_2,i_3)$.
Observe that the subgroup 
$\Kernel{\phi_A} \cap \Kernel{\phi_B} \cap \Kernel{\phi_C} = \langle i_2 = i_3 = i_4 = 0 \rangle \subpr \ZZ^4$ 
is nontrivial, so by Lemma~\ref{lem:nonemptyP}, $\PP$ is empty and no $\shbl$ exists.
This tells us that there is no lower bound on communication that is proportional to the number of loop iterations $\lfloor k/2 \rfloor N^3$.
Indeed, if one considers an execution where the $i_1$ loop is made the innermost, then we can increase $k$ arbitrarily without additional communication.
(Recall that our bounds ignore data dependencies, and do not require correctness.)
So to derive a lower bound and corresponding optimal algorithm, we apply the approach of {\em imposing Reads and Writes} from \sectn{\ref{sec:lb}}.
We will continue this example in \sectn{\ref{sec:bound-prod}}, after introducing Theorem~\ref{thm:prod}.
\exend
\end{exAk}

\subsubsection{Injective Case} \label{sec:bound-injectivemap}

If one of the array references is an injection, then each loop iteration will require (at least) one unique operand, so at most $M$ iterations are possible with $M$ operands in cache.
This observation is confirmed by the following lemma.
\begin{lemma} \label{lem:injectivemap}
If $\phi_k$ is injective, then the minimal $\shbl=1$.
\end{lemma}
\begin{proof}
We see that $s \ceq e_k \in \PP$ by rewriting \eqref{subcriticalhypothesis} as
\[ 
1 \le s_k + \sum_{j \ne k} s_j \Rank{\phi_j(H)}/\Rank{H} \qquad \text{for all subgroups } H \subim \ZZ^d. 
\]
The sum $\shbl = 1$ is minimal due to Lemma~\ref{lem:sumofexponents}.  
\end{proof}
As anticipated, the argument in \sectn{\ref{sec:lb}} gives a lower bound $\Omega(M\cdot\niters/F)=\Omega(\niters)$, that is, no (asymptotic) data reuse is possible.

\begin{exMatvec}[Part~1/3]
The code is 
\begin{align*}
&\For{i_1}{1}{N} ,\quad \For{i_2}{1}{N} ,\\
&\qquad \quad y(i_1) \plusequals A(i_1,i_2) \cdot x(i_2)
\end{align*}
We get $\phi_y(i_1,i_2) = (i_1)$, $\phi_A(i_1,i_2) = (i_1,i_2)$, and $\phi_x(i_1,i_2) = (i_2)$.
Clearly, $\Rank{\phi_A(\ZZ^2)}=2$, i.e., $\phi_A$ is an injection, so by Lemma~\ref{lem:injectivemap} we have $\shbl=1$.
Each iteration $\indx=(i_1,i_2)$ requires a unique entry $A(i_1,i_2)$, so at most $M^{\shbl}=M$ iterations are possible with $M$ operands.
In part~2/3 of this example (in \sectn{\ref{sec:bound-prod}}, below) we arrive at the lower bound $\shbl=1$ in a different manner (via Theorem~\ref{thm:prod}).
\exend
\end{exMatvec}

\subsection{Product Case} \label{sec:bound-prod}
Recall the matrix multiplication example, which iterates over $\ZZ^3$ indexed by three loop indices $(i_1,i_2,i_3)$, with inner loop $C(i_1,i_2) \plusequals A(i_1,i_3)\cdot B(i_3,i_2)$, 
which gives rise to 3 homomorphisms
\begin{equation*}
\phi_A(i_1,i_2,i_3) = (i_1,i_3) ,\quad \phi_B(i_1,i_2,i_3) = (i_3,i_2) ,\quad \phi_C(i_1,i_2,i_3) = (i_1,j_2), 
\end{equation*}
all of which simply choose subsets of the loop indices $i_1,i_2,i_3$.
Similarly, for other linear algebra algorithms, tensor contractions, direct $N$--body simulations, and other examples discussed below, the $\phi_j$ simply choose subsets of loop indices.
In this case, one can straightforwardly write down a simple linear program for the optimal exponents of Theorem~\ref{thm:1}.
We present this result as Theorem~\ref{thm:prod} below (this is a special case, with a simpler proof, of the more general result in \cite[Proposition~7.1]{BCCT10}).
Using Theorem~\ref{thm:prod} we also give a number of concrete examples, some known (e.g., linear algebra) and some new.

We use the following notation. 
$\ZZ^d$ will be our $d$--dimensional set including all possible tuples of loop indices $\indx=(i_1,\ldots,i_d)$. 
The homomorphisms $\phi_1,\ldots,\phi_m$ all choose subsets of these indices $S_1,\ldots,S_m$, i.e., each $S_j \subseteq \set{1,\ldots,d}$ indicates which indices $\phi_j$ chooses.
If an array reference uses multiple copies of a subscript, e.g., $A_3(i_1,i_1,i_2)$, then the corresponding $S_3 = \set{1,2}$ has just one copy of each subscript.
\begin{theorem} \label{thm:prod} 
Suppose each $\phi_j$ chooses a subset $S_j$ of the loop indices, as described above.
Let $H_i \subim \ZZ^d$ denote the subgroup $H_i = \langle e_i \rangle$, i.e., nonzero only in the $i^\text{th}$ coordinate, for $i \in \set{1,\ldots,d}$.
Then, for any $s_j \in [0,1]$,
\begin{equation} \label{subcriticalhypothesis-product}
\Rank{H_i} \le \sum_{j=1}^m s_j\Rank{\phi_j(H_i)} \qquad \text{for all } i \in \set{1,2,\ldots,d}
\end{equation} 
if and only if
\begin{equation} \tag{\ref{subcriticalhypothesis}}
\Rank{H} \le \sum_{j=1}^m s_j\Rank{\phi_j(H)} \qquad \text{for all subgroups } H \subim \ZZ^d .
\end{equation}
\end{theorem} 
\begin{proof}
  Necessity follows from the fact that the constraints \eqref{subcriticalhypothesis-product} are a subset of \eqref{subcriticalhypothesis}.

  To show sufficiency, we rewrite hypothesis \eqref{subcriticalhypothesis-product} as
\[
\Rank{H_i} = 1 \leq \sum_{j=1}^m s_j \delta(j,i) \qquad \text{for all } i \in \set{1,2,\ldots,d},
\]
where $\delta(j,i) = 1$ if $i \in S_j$, and $\delta(j,i) = 0$ otherwise.
Let $H \subim \ZZ^d$ have $\Rank{H}=h$.
Express $H$ as the set spanned by integer linear combinations of columns of some rank $h$ integer matrix $M_H$ of dimension $d$--by--$h$. 
Pick some full rank $h$--by--$h$ submatrix $M'$ of $M_H$. 
Suppose $M'$ lies in rows of $M_H$ in the set $\mathcal{R} \subseteq \set{1,\ldots,d}$; note that $|\mathcal{R}|=h$.
Then
\[
\Rank{\phi_j(H)} \ge \sum_{r \in \mathcal{R}} \delta(j,r), 
\]
so
\[
\sum_{j=1}^{m} s_j \Rank{\phi_j(H)}
\ge \sum_{j=1}^{m} s_j \sum_{r \in \mathcal{R}} \delta(j,r) 
= \sum_{r \in\mathcal{R}}  \sum_{j=1}^{m}  s_j \delta(j,r)
\ge \sum_{r \in\mathcal{R}}  \Rank{H_r}
= \sum_{r \in\mathcal{R}} 1 = |\mathcal{R}| = \Rank{H}. 
\]
\end{proof}
 \begin{remark} \label{rmk:prod-equiv}
We can generalize this to the case where the $\phi_j$ do not choose subsets $S_j$ of the indices $i_1,\ldots,i_d$, but rather subsets of suitable independent linear combinations of the indices, for example subsets of $\set{i_1,i_1+i_2,i_3-2i_1}$ instead of subsets of $\set{i_1,i_2,i_3}$. 
We will discuss recognizing the product case in more general array references in Part~2 of this paper.
\end{remark}

% TODO: Macro for {\bf 1} = ones(d,1)
We now give a number of concrete examples. 
We introduce a common notation from linear programming that we will also use later. 
We start with the version of the hypothesis \eqref{subcriticalhypothesis-product} used in the proof, $1 \le \sum_{j=1}^m s_j \delta(j,i)$, and rewrite this as ${\bf 1}_d^T \le s^T \cdot \Delta$, where ${\bf 1}_d \in \RR^d$ is a column vector of all ones, $s = (s_1,\ldots,s_m)^T \in \RR^m$ is a column vector, and $\Delta \in \RR^{m \times d}$ is a matrix with $\Delta_{ji} = \delta(j,i)$. 
Recall that the communication lower bound in \sectn{\ref{sec:lb}} depends on $\sum_{j=1}^m s_j = {\bf 1}_m^T \cdot s$, and that the lower bound is tightest (maximal) when ${\bf 1}_m^T \cdot s$ is minimized. 
This leads to the linear program
\begin{equation} \label{LP1}
{\rm minimize} \;  {\bf 1}_m^T \cdot s \; {\rm subject\ to} \; 
{\bf 1}_d^T \le s^T \cdot \Delta, 
\end{equation}
where we refer to the optimal value of ${\bf 1}_m^T \cdot s$ as $\shbl$.
Note that \eqref{LP1} is not necessarily feasible, in which case $\shbl$ does not exist. 
This is good news, as the matrix powering example shows.
Since $\Delta$ has one row per $\phi_j$ (i.e., per array reference), and one column per loop index, we will correspondingly label the rows and columns of the $\Delta$ matrices below to make it easy to see how they arise.
These lower bounds are all attainable (see \sectn{\ref{sec:attain-prod}}).
\begin{exMatvec}[Part~2/3]
We continue the example from \sectn{\ref{sec:bound-injectivemap}}, with $\phi_y(i_1,i_2) = (i_1)$, 
$\phi_A(i_1,i_2) = (i_1,i_2)$, and $\phi_x(i_1,i_2) = (i_2)$, or
\[
\Delta = 
\bordermatrix{
  & i_1 & i_2  \cr
y & 1 & 0  \cr
A & 1 & 1  \cr
x & 0 & 1  
}, 
\]
and the linear program~\eqref{LP1} becomes minimizing $\shbl = s_y+s_A+s_x$ subject to $s_y+s_A \ge 1$ and $s_A+s_x \ge 1$.
The solution is $s_x=s_y=0$ and $s_A = 1 = \shbl$.
This yields a communication lower bound of $\Omega (N^2/M^{\shbl-1}) = \Omega(N^2)$.
This reproduces the well-known result that there is no opportunity to minimize communication in matrix-vector multiplication (or other so-called BLAS-2 operations) because each entry of $A$ is used only once.
As we will see in \sectn{\ref{sec:attain-injectivemap}}, this trivial lower bound (no data reuse) is easily attainable.
\exend
\end{exMatvec}
\begin{exNbody}[Part~1/3]
The (simplest) code that accumulates the force on each particle (body) due to all $N$ particles is
\begin{align*}
&\For{i_1}{1}{N} ,\quad \For{i_2}{1}{N} ,\\
&\qquad F(i_1) \plusequals \text{compute\_force}(P(i_1),P(i_2))
\end{align*}
We get $\phi_F(i_1,i_2) = (i_1)$, $\phi_{P_1}(i_1,i_2) = (i_1)$, and $\phi_{P_2}(i_1,i_2) = (i_2)$, or   
\[
\Delta = 
\bordermatrix{
   & i_1 & i_2  \cr
F  & 1 & 0  \cr
P_1 & 1 & 0  \cr
P_2 & 0 & 1  
}. 
\]
Now the solution of the linear program is not unique: $s_F = \alpha$, $s_{P_1} = 1-\alpha$ and $s_{P_2} = 1$ is a solution for any $0 \le \alpha \le 1$, all yielding $\shbl = 2$ and a communication lower bound of $\Omega(N^2/M)$.

Consider also the similar program
\begin{align*}
&\For{i_1}{1}{N_1} ,\quad \For{i_2}{1}{N_2} ,\\
&\qquad \If{\text{predicate}(R(i_1),S(i_2)) = \text{true}} ,\\
&\qquad\qquad \text{output}(i_1,i_2) = \text{func}(R(i_1),S(i_2))
\end{align*}
which represents a generic ``nested loop'' database join algorithm of the sets of tuples $R$ and $S$ with $|R|=N_1$ and $|S|=N_2$.
We have a data-dependent branch in the inner loop, so we split the iteration space $\iters \qec \iters_\text{true} + \iters_\text{false}$ depending on the value of the predicate (we still assume that the predicate is evaluated for every $(i_1,i_2)$).
This gives
 \[
 \Delta_\text{true} = 
\bordermatrix{
   & i_1 & i_2  \cr
R  & 1 & 0  \cr
S & 0 & 1 \cr
\text{output} & 1 & 1} 
\qquad\text{and}\qquad
\Delta_\text{false} = 
\bordermatrix{
   & i_1 & i_2  \cr
R  & 1 & 0  \cr
S & 0 & 1 },
\]
leading to $\shbla{\text{true}} = 1$ and $\shbla{\text{false}}=2$, respectively.
Writing $|\iters_\text{true}| = \alpha|\iters| = \alpha N_1N_2$, then we obtain lower bounds $\Omega(\alpha N_1 N_2)$ and $\Omega((1-\alpha)N_1N_2/M)$ for iterations $\iters_\text{true}$ and $\iters_\text{false}$, respectively.
%
%So, we can consider any reordering of $\iters$ as an interleaving of $\iters_\text{true}$ and $\iters_\text{false}$, and we may simply add the two lower bounds together to obtain a lower bound for the whole program.
So we can take the maximum of these two lower bounds to obtain a lower
bound for the whole program.
Altogether, this yields 
% $\Omega(N_1N_2(\alpha+(1-\alpha)/M))$, 
\[ 
\max\left(\Omega(\alpha N_1N_2), \Omega((1-\alpha)N_1N_2/M)\right) 
= \Omega(\max( \alpha N_1N_2, N_1N_2/M)),\]
an increasing function of $\alpha \in [0,1]$ (using the fact that $(M+1)/M = \Theta(1)$).
When $\alpha$ is close enough to zero, 
we have the `best' lower bound $\Omega(N_1N_2/M)$, 
and when $\alpha$ is close enough to 1, 
we have $\Omega(\alpha N_1N_2)$, the size of the output,
a lower bound for any computation.
\exend
\end{exNbody}
\begin{exMatmul}[Part~3/5]
We outlined this result already in part~2/5 of this example (see \sectn{\ref{sec:intro}}). 
%
% The code is 
%\[
 % \For{i_1}{1}{N} ,\quad \For{i_2}{1}{N} ,\quad \For{i_3}{1}{N} ,\quad C(i_1,i_2) \plusequals A(i_1,i_3)\cdot B(i_3,i_2)
%\]
We have $\phi_A(i_1,i_2,i_3)=(i_1,i_3)$, $\phi_B(i_1,i_2,i_3)=(i_3,i_2)$, and $\phi_C(i_1,i_2,i_3)=(i_1,i_2)$, or
\[
\Delta = 
\bordermatrix{
  & i_1 & i_2 & i_3  \cr
A & 1 & 0 & 1  \cr
B & 0 & 1 & 1  \cr
C & 1 & 1 & 0
}. 
\]
One may confirm that the corresponding linear program has solution $s_A = s_B = s_C = 1/2$ so $\shbl = 3/2$, yielding the well known communication lower bound of $\Omega(N^3/M^{1/2})$. 
Recall from \sectn{\ref{sec:model}} that the problem of matrix multiplication was previously analyzed in \cite{ITT04} using the Loomis-Whitney inequality (Theorem~\ref{thm_LoomisWhitney}, a special case of Theorem~\ref{thm:1}) and extended to many other linear algebra algorithms in \cite{BallardDemmelHoltzSchwartz11}.
\exend
\end{exMatmul}
\begin{exTensor}[Part~1/2]
Assuming $1 \le j < k-1 < d$, the code is 
\begin{align*}
 &\For{i_1}{1}{N} ,\quad \ldots,\quad \For{i_d}{1}{N} ,\\
 &\qquad C(i_1,\ldots,i_j,i_k,\ldots,i_d) \plusequals A(i_1,\ldots,i_{k-1}) \cdot B(i_{j+1},\ldots,i_d)
\end{align*}
We get
\[
\Delta = 
\bordermatrix{
  & i_1 & \cdots & i_j & i_{j+1} & \cdots & i_{k-1} & i_k & \cdots & i_d  \cr
A & 1   & \cdots &  1  &  1      & \cdots &   1     &  0  & \cdots &  0   \cr 
B & 0   & \cdots &  0  &  1      & \cdots &   1     &  1  & \cdots &  1   \cr 
C & 1   & \cdots &  1  &  0      & \cdots &   0     &  1  & \cdots &  1   
}.
\]
The linear program is identical to the linear program for matrix multiplication, so $\shbl = 3/2$, and communication lower bound is $\Omega(N^d/M^{1/2})$.  
\exend
\end{exTensor}
\begin{exComplicated}[Part~2/4]
We already outlined this example in part~1/4 (see \sectn{\ref{sec:intro}}).
%
%The following complicated code fragment is intended to demonstrate the generality of Theorem~\ref{thm:prod}.
%% TODO: code
%\begin{tabbing}
%junk \= junk \= junk \= junk \= \kill
%\> for $i_1 = 1$ to $N$, for $i_2 = 1$ to $N$, for $i_3 = 1$ to $N$, for $i_4 = 1$ to $N$, for $i_5 = 1$ to $N$, for $i_6 = 1$ to $N$, \\
%\> \> $A_1(i_1,i_3,i_6) \plusequals func1(A_2(i_1,i_2,i_4), A_3(i_2,i_3,i_5), A_4(i_3,i_4,i_6))$ \\
%\> \> $A_5(i_2,i_6) \plusequals func2(A_6(i_1,i_4,i_5), A_3(i_3,i_4,i_6))$
%\end{tabbing}
%
The code given in part~1/4 leads to
\[
\Delta = 
\bordermatrix{
  & i_1 & i_2 & i_3 & i_4 & i_5 & i_6 \cr
A_1 & 1 & 0 & 1 & 0 & 0 & 1 \cr
A_2 & 1 & 1 & 0 & 1 & 0 & 0 \cr
A_{3,1} & 0 & 1 & 1 & 0 & 1 & 0 \cr
A_4,A_{3,2} & 0 & 0 & 1 & 1 & 0 & 1 \cr
A_5 & 0 & 1 & 0 & 0 & 0 & 1 \cr
A_6 & 1 & 0 & 0 & 1 & 1 & 0},
\]
since $A_3$ is accessed twice, once with the same subscript as $A_4$, so we have ignored the duplicate subscript following Remark~\ref{rmk:duplicatemaps}.
We obtain $s=(2/7,1/7,3/7,2/7,3/7,4/7)^T$ as a solution, giving $\shbl = 15/7$, so the communication lower bound is $\Omega(N^6/M^{8/7})$.
\exend
\end{exComplicated}
%
%\paragraph{Example 5: ``Nonassociative tensor contraction.''} 
%
%Multiplying multiple matrices or tensors is done most efficiently by exploiting associativity, i.e.\ inserting parentheses, and multiplying pairs. 
%%
%This example uses a general function in the inner loop to make this otherwise natural optimization impossible.
%%
%The code is 
%\[
%  \For{i_1}{1}{N},\quad \ldots,\quad \For{i_6}{1}{N}, A_1(i_1,i_6) \plusequals \text{func}(A_2(i_1,i_2), A_3(i_2,i_3), A_4(i_3,i_4), A_5(i_4,i_5), A_6(i_5,i_6))
%\]
%We get
%\[
%\Delta = 
%\bordermatrix{
%    & i_1 & i_2 & i_3 & i_4 & i_5 & i_6 \cr
%A_1 &  1  &  0  &  0  &  0  &  0  &  1  \cr
%A_2 &  1  &  1  &  0  &  0  &  0  &  0  \cr
%A_3 &  0  &  1  &  1  &  0  &  0  &  0  \cr
%A_4 &  0  &  0  &  1  &  1  &  0  &  0  \cr
%A_5 &  0  &  0  &  0  &  1  &  1  &  0  \cr
%A_6 &  0  &  0  &  0  &  0  &  1  &  1  
%}
%\]
%The solution of the linear program yields all $s_{A_i} = 1/2$, $\shbl = 3$, and a communication lower bound of $\Omega(N^6/M^2)$.
%%
%
%
\begin{exAk}[Part~3/4]
We continue this example from part~2/4 (above), with $\phi_A$, $\phi_B$ and $\phi_C$ as given there.
% with $\phi_{A_1}\from(q,i,j,k)\mapsto(i,k)$ and $\phi_{A_2}\from(q,i,j,k)\mapsto(j,k)$, and $\phi_B\from(q,i,j,k)\mapsto(i,j)$.
%
So we have
\[
\Delta = 
\bordermatrix{
  & i_1 & i_2 & i_3 & i_4  \cr
A & 0   & 1   & 0   & 1    \cr
B & 0   & 0   & 1   & 1    \cr
C & 0   & 1   & 1   & 0
} 
\]
The first column of $\Delta$ gives us the infeasible constraint 
$0 \cdot s_A + 0 \cdot s_B + 0 \cdot s_C \geq 1$.
% $0 \cdot s_{A_1} + 0 \cdot s_{A_2} + 0 \cdot s_B \geq 1$.
%
As mentioned above, we can reorganize the loops so that we can increase $k$ arbitrarily without communication.
The issue is that the intermediate powers of $A$ do not occupy unique memory addresses, but rather overwrite the previous power (stored in the buffers $B$ and $C$). 
Following the strategy of imposing Reads and Writes, we apply Theorem~\ref{thm:prod} to the second code fragment in part~1/4 of this example (see \sectn{\ref{sec:model}}), yielding
\[
\Delta = 
\bordermatrix{
          & i_1 & i_2 & i_3 & i_4  \cr
\hat{A}_1 & 1   & 1   & 1   & 0  \cr
\hat{A}_2 & 0   & 1   & 0   & 1  \cr
\hat{A}_3 & 1   & 0   & 1   & 1
} 
\]
where the subscripts on $\hat{A}$ distinguish its three appearances.
The resulting linear program is feasible with $s_1 = s_2 = s_3 = 1/2$ and $\shbl = 3/2$, yielding a communication lower bound of $c k N^3/M^{1/2}$ for some constant $c>0$.
As mentioned in part~1/4, the number of imposed Reads and Writes is at most $2(k-2)N^2$, so subtracting we get a communication lower bound for the original algorithm of $ckN^3/M^{1/2} - 2(k-2)N^2$.
When $N =\omega(M^{1/2})$, the first dominates, and we see the complexity is the same as $k$ independent matrix multiplications.
When $N = O(M^{1/2})$, the second term may dominate, leading to a zero lower bound, which reflects the fact that one can read the entire matrix $A$ into fast memory, perform the entire algorithm while computing and discarding intermediate powers of $A$, and only write out the final result $A^k$.
In other words, any communication lower bound that is proportional to the number of loop iterations $\niters = (k-1)N^3$ must be zero.
\exend
\end{exAk}

\section{Communication-optimal algorithms} \label{sec:attain}
Our ultimate goal is to take a simple expression of an algorithm,
say by nested loops, and a target architecture, and automatically 
generate a communication optimal version, 
i.e., that attains the lower bound of Theorem~\ref{Thm4.1}, 
assuming the dependencies between loop iterations permit.
In this section we present some special cases where this
is possible, corresponding to the cases in 
\sectn{\ref{sec:bound}} where we could effectively compute the 
linear constraints \eqref{subcriticalhypothesis}.
%
% Beyond the special case treated in section~\ref{sec:attain-prod} below, this is still an open problem.
%
% Indeed, given the (un)decidability issues raised in 
%\sectn{\ref{sec:undecidability}}, 
%it is conceivable that there are situations where an 
%optimal algorithm may exist, without there being an 
% effective way to write it down.

The rest of this section is organized as follows. 
\sectn{\ref{sec:attain-trivial}} briefly addresses the infeasible and injective cases described in \sectn{\ref{sec6.2}}. 
\Sectn{\ref{sec:attain-prod}} presents communication-optimal sequential algorithms for the product case whose lower bound was presented in \sectn{\ref{sec:bound-prod}}. 
\Sectn{7.3} presents communication-optimal parallel algorithms for the same product case, including generalizations of recent algorithms like 2.5D matrix multiplication \cite{2.5D_EuroPar}, which use additional memory (larger $M$) to replicate data and reduce communication.
In Part~2 of this paper (outlined in \sectn{\ref{sec:conclusions}}) we will discuss attainability in general.
%other special cases in which optimal algorithms are known.

\subsection{Trivial Cases} \label{sec:attain-trivial}

\subsubsection{Infeasible Case} \label{sec:attain-nontrivialkernel}
As shown in Lemma~\ref{lem:nonemptyP}, 
if $H=\bigcap_{j=1}^m \Kernel{\phi_j}$ is nontrivial, then $\shbl$ doesn't exist.
That is, since $H$ is infinite, we can potentially perform an unbounded number 
of iterations $\indx \in H$ with the same $m$ operands 
(which fit in cache, since $m = o(M)$ by assumption 
--- see \sectn{\ref{sec:lb}}).

Recall the matrix powering example from the previous section (parts~2/4 and~3/4).
In part~2/4, we saw a lower bound of $0$ words moved was possible, provided the matrix dimension $N$ was bounded in terms of $M^{1/2}$.
We then showed in part~3/4 how to obtain a lower bound on the modified code with imposed Reads and Writes (from part~1/4, in \sectn{\ref{sec:lb}}) which is sometimes tighter, depending on the relative sizes of $N$ and $M$.
In part~4/4 below, we show that the tighter bound is also attainable, by reorganizing the modified code.

\subsubsection{Injective Case} \label{sec:attain-injectivemap}
As shown in Lemma~\ref{lem:injectivemap}, if there is an injective $\phi_j$, we have $\shbl=1$.
Thus no data reuse is possible asymptotically.
So, executing the loop nest in any order (of the iteration space) will attain the lower bound.
We revisit the matrix-vector multiplication example from the previous section (parts~1/3 and~2/3) in part~3/3 below.

\subsection{Product Case --- Sequential Algorithms} \label{sec:attain-prod}
Here we show how to 
find an optimal sequential algorithm in the special case discussed 
in Theorem~\ref{thm:prod}, where the subscripts $\phi_j$ 
all choose subsets $S_j$ of the loop indices $i_1,\ldots,i_d$.
Recalling the notation from \sectn{\ref{sec:bound-prod}}, \eqref{LP1} provided a linear program for $s = (s_1,\ldots,s_m)^T$ (that we call ``$\text{sLP}$'' here),
\[
  (\text{sLP}) \quad  \text{minimize } {\bf 1}_m^T \cdot s \text{ subject to } {\bf 1}_d^T \le s^T \cdot \Delta ,
\]
where $\Delta \in \RR^{m \times d}$ was a matrix with $\Delta_{ji} = 1$ if $i \in S_j$ and $0$ otherwise.
We let $\shbl$ denote the minimum value of ${\bf 1}_m^T \cdot s$.

For notational simplicity, we consider a computer program of the form
\begin{align*}
&\For{i_1}{1}{N} ,\quad \For{i_2}{1}{N} ,\quad \ldots,\quad \For{i_d}{1}{N} ,\\
&\qquad\text{inner\_loop}(i_1,i_2,\ldots,i_d)
\end{align*}
and ask how to reorganize it to attain the communication lower bound of $\Omega(N^d/M^{\shbl-1})$ in a sequential implementation.  
As explained in \sectn{\ref{sec:lb}}, explicit loop nests are not necessary, but it is easier to explain the result.

We consider a reorganization of the following ``blocked'' or ``tiled'' form:
\begin{align*}
&\ForStep{j_1}{1}{M^{x_1}}{N} ,\quad \ForStep{j_2}{1}{M^{x_2}}{N} ,\quad \ldots,\quad \ForStep{j_d}{1}{M^{x_d}}{N} ,\\
&\qquad\For{k_1}{0}{M^{x_1}-1} ,\quad \For{k_2}{0}{M^{x_2}-1} ,\quad \ldots,\quad \For{k_d}{0}{M^{x_d}-1} ,\\
&\qquad\qquad(i_1,i_2,\ldots,i_d) = (j_1,j_2,\ldots,j_d) + (k_1,k_2,\ldots,k_d) \\
&\qquad\qquad\text{inner\_loop}(i_1,i_2,\ldots,i_d)
\end{align*}
Here $M$ is the fast memory size, and $x_1,\ldots,x_d$ are parameters to be determined below. 

Here we show that simply solving the dual linear program of ($\text{sLP}$) provides the values of $x_1,\ldots,x_d$ that we need for 
the tiled code to be optimal, i.e., to attain the communication lower bound:
\begin{theorem} \label{thm6.1}
  Suppose that the linear program ($\text{sLP}$) for $s$ is feasible. 
  Then the dual linear program ``$\text{xLP}$'' for $x = (x_1,\ldots,x_d)^T$
\[
  (\text{xLP})\quad\text{maximize } {\bf 1}_d^T \cdot x \text{ subject to } {\bf 1}_m \ge \Delta \cdot x,
\]
is also feasible, and using these values of $x$ in the tiled code lets it attain the communication lower bound of $\Omega(N^d/M^{\shbl-1})$ words moved.
\end{theorem}
\begin{proof}
By duality the solution $x$ of ($\text{xLP}$) satisfies ${\bf 1}_d^T \cdot x = {\bf 1}_m^T \cdot s = \shbl$.
Let $E = \set{(i_1,\ldots,i_d)} \subset \ZZ^d$ be the set of lattice points traversed by the $d$ innermost loops of the tiled code, i.e., the loops over $k_1,\ldots,k_d$, for one value of $(j_1,\ldots,j_d)$.
It is easy to see that the size of $E$ is $|E| = \prod_{i=1}^d M^{x_i} = M^{\sum_{i=1}^d x_i} = M^{\shbl}$.
Furthermore, the constraint ${\bf 1}_m \geq \Delta \cdot x$ means that for each $j$, $1 \ge \sum_{i=1}^d \Delta_{ji} x_i = \sum_{i \in S_j} x_i$, with equality for at least one $j$.
Thus $|\phi_j (E)| = \prod_{i \in S_j} M^{x_i} = M^{\sum_{i \in S_j} x_i} \le M^1 = M$.
In other words, the number of words accessed by the $j^\text{th}$ array in the innermost $d$ loops is at most $M$, with equality at least once.
Thus the inner $d$ loops of the algorithm perform the maximum possible $M^{\shbl}$ iterations while accessing $\Theta(M)$ words of data, so the algorithm is optimal.
\end{proof}
\noindent We note that in practice, we may want to choose the block sizes $M^{x_i}$ to be a constant factor smaller, in order for all the data accessed in the innermost $d$ loops to fit in fast memory memory simultaneously.
We will look at these constants in Part~2 of this work.

We look again at the examples from \sectn{\ref{sec:bound-prod}}, using the notation introduced there.
\begin{exMatvec}[Part~3/3]
We write $x = (x_1,x_2)^T$ as the unknown in the dual linear program ($\text{xLP}$), so that $M^{x_1}$ is the tile size for variable $i_1$ and $M^{x_2}$ is the tile size for variable $i_2$.
Then ($\text{xLP}$) becomes maximizing $\shbl = x_1 + x_2$ subject to $x_1 \le 1$, $x_1+x_2 \le 1$, and $x_2 \le  1$.
The solution is $x_1 = \alpha$ and $x_2 = 1 - \alpha$ for any $0 \le \alpha \le 1$, and $\shbl = 1$. 
In other words, we can ``tile'' the matrix $A$ into any $M^{\alpha}$--by--$M^{1-\alpha}$ rectangular blocks of total size $M$.
Indeed, such flexibility is common when one of the subscripts is injective (see \sectn{\ref{sec:attain-injectivemap}}, above).

We can use our approach to further reduce communication,
potentially up to a factor of $2$, by choosing the optimal $\alpha$.
We simply ignore the array $A$, and apply the theory just to the
vectors $x$ and $y$, leading to the same analysis as the $N$--body problem,
which tells use to use blocks of size $M$ for both $x$ and $y$ (see next example).
We still read each entry of $A$ once (moving $N^2$ words), but reuse each
entry of $x$ and $y$ $M$ times.
\exend
\end{exMatvec}
\begin{exNbody}[Part~2/3]
We write $x = (x_1,x_2)^T$ as the unknown in the dual linear program ($\text{xLP}$), so that $M^{x_1}$ is the tile size for variable $i_1$ and $M^{x_2}$ is the tile size for variable $i_2$.
Then ($\text{xLP}$) becomes maximizing $\shbl = x_1 + x_2$ subject to $x_1 \le 1$ and $x_2 \le 1$. 
The solution is $x_1 = x_2 = 1$ and $\shbl = 2$, which tells us to to take $M$ particles indexed as $P(i_1)$, $M$ (usually different) particles indexed as $P(i_2)$, and compute all $M^2$ pairs of interactions.

For the database join example, we use the optimal blocking for the
iterations $\iters_\text{false}$, which also yields $x_1=x_2=1$, and
so compute ${\rm predicate}(R(i_1),S(i_2))$ $M^2$ times using
$M$ values each of $R_{i_1}$ and $S_{i_2}$. Whenever the predicate
is true, ${\rm output}(i_1,i_2)$ is immediately written to slow memory.
Thus there are $N_1N_2/M$ reads of $R(i_1)$ and $S(i_2)$,
and $|\iters_\text{true}| = \alpha N_1 N_2$ writes of 
${\rm output}(i_1,i_2)$, which attains the lower bound.
%
%For the database join example, we cannot always attain the lower bound $\Omega(\alpha N_1N_2+(1-\alpha)N_1N_2/M)$, because the lower bound $\Omega(\alpha N_1N_2)$ for iterations $\iters_\text{true}$ is actually loose (unattainable) by a lower-order term, and depending on $\alpha$, this term may comparable to the lower bound for iterations $\iters_\text{false}$.
%%
%That is, in reality we can perform $M^2$ iterations with no less than $M+M^2$ operands, so with $M$ operands we can tighten our upper bound $F=M^2/(M+1)<M$.
%Adding the modified bound for $\iters_\text{true}$ with the original bound for $\iters_\text{false}$, we obtain,
%%
% \[ \Omega\left(\alpha \frac{M+1}{M}N_1N_2\right) + \Omega\left((1-\alpha)\frac{N_1N_2}{M}\right) = \Omega\left(\frac{N_1N_2}{M}+\alpha N_1N_2\right)\]
% %
%This tighter lower bound is attainable by the same blocking as for the $N$-body case, except any time the predicate evaluates true, we compute the output and immediately write it to slow memory ($(O(\alpha N_1 N_2)$ writes). 
%
\exend

\end{exNbody}
\begin{exMatmul}[Part~4/5]
We write $x = (x_1,x_2,x_3)^T$ as the unknown in the dual linear program ($\text{xLP}$), so that $M^{x_1}$ is the tile size for variable $i_1$, $M^{x_2}$ is the tile size for variable $i_2$, and $M^{x_3}$ is the tile size for variable $i_3$.
Then ($\text{xLP}$) becomes maximizing $\shbl = x_1 + x_2 + x_3$ subject to $x_1+x_3 \le 1$, $x_1+x_2 \le 1$ and $x_2+x_3 \le 1$.
The solution is $x_1 = x_2 = x_3 = 1/2$ and $\shbl = 3/2$, which tells us the well-known result that each matrix $A$, $B$, and $C$ should be tiled into square blocks of size $M^{1/2}$--by--$M^{1/2}$.
When $M$ is the cache size on a sequential machine, this is a well-known algorithm, with a communication cost of $\Theta(N^3/M^{1/2})$ words moved between cache and main (slower) memory.
\exend
\end{exMatmul}
\begin{exTensor}[Part~2/2]
We write $x=(x_1,\ldots,x_d)^T$ as the unknown in the dual linear program ($\text{xLP}$), so that $M^{x_r}$ is the tile size for variable $i_r$.
Then ($\text{xLP}$) becomes maximizing $\shbl = \sum_{r=1}^d x_r$ subject to 
\[
  \sum_{r=1}^{k-1} x_r \le 1, \quad \sum_{r=j+1}^d x_r \le 1,\quad \text{and}\quad \sum_{r=1}^j x_r + \sum_{r=k}^d x_r \le 1.
\]
The solution is any set of values of $0 \le x_r \le 1$ satisfying 
\[
  \sum_{r=1}^{j} x_r = 1/2,\quad \sum_{r=j+1}^{k-1} x_r = 1/2,\quad\text{and}\quad \sum_{r=k}^d x_r = 1/2,
\]
so $\shbl = 3/2$.
For example, one could use
\[
  x_1 = \cdots = x_j = \frac{1}{2j},\quad x_{j+1} = \cdots = x_{k-1} = \frac{1}{2(k-j-1)},\quad\text{and}\quad x_k = \cdots = x_d = \frac{1}{2(d-k+1)}
\]
which is a natural generalization of matrix multiplication.
\exend
\end{exTensor}
\begin{exComplicated}[Part~3/4]
We write $x = (x_1,\ldots,x_6)^T$ as the unknown in the dual linear program ($\text{xLP}$), so that $M^{x_r}$ is the tile size for variable $i_r$.
The solution is \[ x=(2/7,3/7,1/7,2/7,3/7,4/7)^T \] (and $\shbl = 15/7$), leading to the block sizes $M^{2/7}$--by--$\cdots$--by--$M^{4/7}$. 
\exend
\end{exComplicated}
%
%\paragraph{Example 5: ``Nonassociative tensor contraction.''} \TODO{Replace with Jim's example}
%We write $x=[x_1,\ldots,x_6]^T$ as the unknown in the dual linear program ($\text{xLP}$), so that $M^{x_r}$ is the tile size for variable $i_r$.
%%
%Then ($\text{xLP}$) becomes maximizing $\shbl = \sum_{r=1}^6 x_r$ subject to $x_1+x_6 \le 1$, $x_1+x_2 \le 1$, $x_2+x_3 \le 1$, \ldots, $x_5+x_6 \le 1$.
%%
%The solution is $x_r = 1/2$ and $\shbl = 3$. 
%%
%Analogous to matrix multiplication, this tells us that each matrix $A_1,\ldots,A_6$ should tiled into square blocks of size $M^{1/2}$-by-$M^{1/2}$.
%%
%
%
\begin{exAk}[Part~4/4]
While ($\text{sLP}$) is infeasible, ($\text{xLP}$) has unbounded solutions.
As suggested above, the trivial lower bound of $0$ words moved may not always be attainable.
After imposing Reads and Writes, we saw in part~3/4 (see previous section) that when $N$ is sufficiently large compared to $M^{1/2}$, the communication lower bound (for the modified program) looks like $k$ independent matrix multiplications, and when $N = O(M^{1/2})$, the lower bound degenerates to $0$.
To attain the lower bound for the modified code with imposed Reads and Writes, we implement the matrix multiplication 
$\hat{A}(i_1,\cdot,\cdot) = \hat{A}(1,\cdot,\cdot) \cdot \hat{A}(i_1-1,\cdot,\cdot)$ 
in the inner loop in the optimized way as described in part~4/4 of the matrix multiplication example (above), and leave the outer loop (over $i_1$) unchanged.
\exend
\end{exAk}

\subsection{Product Case --- Parallel Algorithms}
\label{sec:attain-prod-par}
% \subsection{Product Case - parallel algorithms}
\label{sec7_3}

We again consider the special case discussed in Theorem~\ref{thm:prod},
and show how to attain a communication optimal parallel algorithm.
We again start with the unblocked code of the last section,
and show how to map it onto $P$ processors so that it is both
{\em load balanced}, i.e., every processor executes
$\text{inner\_loop}(i_1,i_2,\ldots,i_d)$
$O(N^d/P)$ times, and every processor attains the communication 
lower bound $\Omega ((N^d/P)/M^{\shbl-1})$, i.e., communicates at most
this many words to or from other processors. We need some assumption
about load balance, since otherwise one processor could 
store all the data and do all the work with no communication.
In contrast to the sequential case, $M$ is no longer fixed by the hardware
(e.g., to be the cache size), but can vary from the
minimum needed to store all the data in the arrays accessed by the algorithm
to larger values. Since the lower bound is a decreasing function of $M$,
we should ideally have a set of algorithms that communicate less as $M$
grows (up to some limit). There is in fact a literature
%\TODO{\cite{many}} 
% Tiskin pairwise: Tiskin07
% Demmel/Solomonik, 2.5D_EuroPar, SolomonikDemmel11
% McColl/Tiskin efficient matmul: McCollTiskin99
% N-Body: DGKSY12
on algorithms for linear algebra 
\cite{2.5D_EuroPar,SolomonikDemmel11,Tiskin07,McCollTiskin99}
and the $N$--body problem 
\cite{DGKSY12}
that do attain the lower bound as $M$ grows; here we will see that this is
possible in general (again, loop dependencies permitting).

In \sectn{\ref{sec:intro}}, we said that we would only count the number of words
moved between processors, not other costs like message latency, network
congestion, or costs associated with noncontiguous data.
We will also ignore dependencies.
This will greatly simplify our presentation, but leave significant work
to design more practical parallel algorithms; we leave this until Part~2
of this paper.

Assume for a moment that we have chosen $M$; we distribute the work
among the processors as follows: Supposing that the linear program
($\text{sLP}$) is feasible,
let $x=(x_1,\ldots,x_d)^T$ be the solution of the dual linear program
($\text{xLP}$) defined in Theorem~\ref{thm6.1}.
Just as in the sequential case, we take the lattice of $N^d$ points
representing inner loop iterations,
indexed from $\indx = (i_1,i_2,\ldots,i_d)=(0,0,\ldots,0)$ to $(N-1,N-1,\ldots,N-1)$,
and partition the axis indexed by $i_j$ into $N/M^{x_j}$ contiguous segments
of equal size $M^{x_j}$. (As before, for simplicity we assume all
fractions like $N/M^{x_j}$ divide evenly.)
This in turn divides the entire lattice into 
$\prod_{i=1}^d N/M^{x_i} = N^d/M^\shbl$ ``bricks''
of $M^{\shbl}$ lattice points each.
Let $E$ refer to the lattice points in one brick.
Just as in the proof of Theorem~\ref{thm6.1}, the amount of data
from array $A_j$ needed to execute the inner loop iterations
indexed by $E$ is bounded by $|\phi_j(E)| \le M$.
So, up to a constant multiple $m$ (which we ignore),
each processor has enough memory to store the data needed
to execute a brick. This yields one constraint from our
approach: to be load balanced, there need to be at least
$P$ bricks, i.e., $N^d/M^{\shbl} \ge P$, or
$M \le N^{d/\shbl}/P^{1/\shbl}$.
In other words, we should pick $M$ no larger than this upper
bound in order to use all available processors.
(Of course $M$ also cannot exceed the memory available on each processor.)
% \TODO{If there aren't enough blocks, why don't we simply reblock for a fraction $M \ceq M/C$?}

The next constraint arises from having enough memory
in all the processors to hold all the data. Recalling
the notation used to define ($\text{sLP}$), we let
$S_j$ be the subset of loop indices appearing in $\phi_j$, and we have $|S_j|=d_j=\Rank{\phi_j(\ZZ^d)}$.
This means that the total storage required by array $A_j$
is $N^{d_j}$, and the total storage required by
all arrays is $\sum_{j=1}^m N^{d_j}$, yielding the lower
bound $M \ge \sum_{j=1}^m N^{d_j} / P$, and so altogether
\begin{equation}\label{eqn_ParConstraints}
\frac{N^{d/\shbl}}{P^{1/\shbl-1}} 
\ge M
\ge \frac{\sum_{j=1}^m N^{d_j}}{P}
\ge \frac{N^{d_\text{max}}}{P}, 
\end{equation}
where $d_\text{max} \ceq \max_{j=1}^m d_j$.
Plugging these bounds on $M$ into the communication lower bound yields
\begin{equation}\label{eqn_ParComm}
\frac{N^{d/\shbl}}{P^{1/\shbl}} 
\le \frac{N^d/P}{M^{\shbl-1}} 
\le \frac{N^{d-d_\text{max}(\shbl-1)}}{P^{2-\shbl}} .
\end{equation}
The expression on the left is simply the memory-independent lower bound $\Omega((\niters/P)^{1/\shbl})$ discussed in \sectn{\ref{sec:gen-machine}}.
Note that when $M$ equals its upper bound, it also equals the communication lower bound, which means that the content of the memory only needs to be communicated once; this is related to the fact that the assumption $\niters/P = \omega(F)$ fails to hold (see \sectn{\ref{sec:gen-machine}}).

For example, for matrix multiplication ($d=3$, $\shbl=3/2$, and $d_\text{max} =2$) 
we get the familiar constraints 
$N^2/P^{2/3}
\ge M \ge 
N^2/P$
and $N^2/P^{2/3} \le (N^3/P)/M^{1/2} \le N^2/P^{1/2}$
% \TODO{\cite{various}}.
\cite{BDHLS12,AggarwalChandraSnir90,ITT04,2.5D_EuroPar}.
%  Strong scaling (local) BDHLS12
%  Snir paper AggarwalChandraSnir90

For $M$ in this range, we may express our parallel algorithm
most simply as follows:
\begin{align*}
&\text{while there are unexecuted bricks} \\
&\qquad\text{assign $P$ unexecuted bricks to $P$ processors} \\
&\qquad\text{in parallel, communicate the needed data to each processor} \\
&\qquad\text{in parallel, execute each brick}
% \\ &\text{end while}
\end{align*}

The algorithm is clearly load balanced (and oblivious to any
loop dependencies). It remains to compute the communication
complexity: this is simply the number of times each
processor executes a brick, times the communication required per brick.
Since there are $N^d/M^{\shbl}$ bricks shared evenly by $P$
processors, with $O(M)$ communication per brick, the result
is $O((N^d/P)/M^{\shbl-1})$, i.e., the lower bound is attained.
More formally, we could use an abstract model of computation like
LPRAM \cite{AggarwalChandraSnir90} or
BSPRAM \cite{Tiskin98,McCollTiskin99}
to describe this algorithm.
% \TODO{We should acknowledge that this is equivalent to using some known theoretical model (e.g.\ BSPRAM) since readers/reviewers from that community will want us to. 
%
% We should also claim that this PRAM sketch means that a ``real'' optimal algorithm will exist (at least asymptotically), in the bandwidth sense.}

Instead, by making more assumptions (satisfied by some of our running examples)
we may also describe the algorithms in more detail as to how to map them
to an actual machine (we will weaken these assumptions in Part~2).
We will assume that $\Delta \cdot x = {\bf 1}_m$, not just $\Delta \cdot x \le {\bf 1}_m$ as required in (xLP).
By choosing all $x_i = 1/d_\text{max}$, which satisfies $\Delta \cdot x \leq {\bf 1}_m$, we see that $\shbl \geq {\bf 1}_d^T \cdot x = d/d_\text{max}$.
We will in fact assume $\shbl = d/d_\text{max}$.
Finally, we will assume for convenience that certain intermediate expressions, like $P^{\shbl}$, are integers.
We will write $M = CN^{d_\text{max}}/P$, where $C$ is the number of copies of the data we will use; $C=1$ corresponds to the minimum memory necessary.  
Combining with the upper bound on $M$ from \eqref{eqn_ParConstraints}, we get $N^{d/\shbl}/P^{1/\shbl} \ge C N^{d_\text{max}}/P$, or $C \le P^{1 - 1/\shbl}$.

Next, we divide the $P$ processors into $C$ groups of $P/C$ processors each. 
The total number of bricks is $\prod_{i=1}^d N/M^{x_i} = N^d/M^{\shbl} = N^d/(CN^{d_\text{max}}/P)^{\shbl} = (P/C)^{\shbl}$. 
If we used just $P/C$ processors to process all these bricks ($P/C$ at a time in parallel), then it would take $(P/C)^{\shbl-1}$ parallel phases. 
But since there are $C$ groups of processors, only $P^{\shbl-1}/C^{\shbl}$ phases are needed.

Now we describe how the data is initially laid out in each group of $P/C$ processors. 
Since we assume there is enough memory $P\cdot M$ for $C$ copies of the data, each group of $P/C$ processors will have its own copy.
For each array $A_\mu$, $1 \le \mu \le m$, we index the processors by $(j_k \in \NN : 0 \le j_k < N/M^{x_k})_{k  \in S_\mu}$, so the number of processors indexed is $\prod_{k \in S_\mu} N/M^{x_k} = N^{d_\mu}/M^{\sum_{k \in S_\mu} x_k} = N^{d_\mu}/M$, since we have assumed $\Delta \cdot x = {\bf 1}_m$.
Finally, $N^{d_\mu}/M \le N^{d_\text{max}}/M = P/C$, so we index all the processors in the group when $d_\mu = d_\text{max}$.
(When $d_\mu < d_\text{max}$, i.e., the array $A_\mu$ has fewer than the maximum number of dimensions, then it is smaller and so needs fewer processors to store.)
Since $A_\mu$ has subscripts given by the indices in $S_\mu$, we may partition $A_\mu$ by blocking the subscripts $i_k$, for $k \in S_\mu$, into blocks of size $M^{x_k}$, so that each block contains $\prod_{k \in S_\mu} M^{x_k} = M$ entries, and so may be stored on a single processor. 
Then block $(j_k)_{k \in S_\mu}$ is stored on the processor with the same indices, for $0 \le j_k < N/M^{x_k}$. This indexing will make it easy for a processor to find the data it needs, given the brick it needs to execute.

Next we describe how to group bricks (of the iteration space), indexed by ${\cal J} = (j_k \in \NN : 0 \le j_k < N/M^{x_k})_{k=1}^d$, into groups of $P/C$ to be done in each phase. 
We define a ``superbrick'' to consist of $(N/M^{x_k})^{1/\shbl}$ contiguous bricks in each direction, containing a total of $\prod_{k=1}^d (N/M^{x_k})^{1/\shbl} = P/C$ bricks.
Thus one superbrick can be executed in one parallel phase by one group of $P/C$ processors, and $C$ superbricks can be executed in one parallel phase by all $P$.
Given the indexing described in the last paragraph, the index $\cal J$ of each brick in a superbrick indicates exactly which processor in a group owns the data necessary to execute it.

We still need to assign superbricks to groups of processors.
There are many ways to do this.
We consider two extreme cases, $C=1$ and $C=P^{1 - 1/\shbl}$.

When $C=1$, there is just one group of processors, and there are $P^{\shbl-1}$ superbricks, each containing $P$ bricks, with each brick assigned to a processor using the index $\cal J$.
Altogether there are $P^{\shbl-1}$ parallel phases to complete the execution.
We emphasize that this is just one way of many to accomplish parallelization; 
see the examples below. For example, for matrix multiplication,
this corresponds to the classical so-called `2D' algorithms, 
where the $P$ processors are laid out in a $P^{1/2}$--by--$P^{1/2}$ grid.
% \TODO{explain that this corresponds to a kind of 2D matmul, but not the usual SUMMA algorithm.}

When $C$ equals its maximum value $C=P^{1 - 1/\shbl}$, then there are just $C$ superbricks and the number of parallel phases is $P^{\shbl-1}/C^{\shbl} = 1$.
For example, for matrix multiplication 
this corresponds to a so-called `3D' algorithm,
where the $P$ processors are laid out in a 
$P^{1/3}$--by--$P^{1/3}$--by--$P^{1/3}$ grid (see \cite{AggarwalChandraSnir90}).
For the $N$--body problem, this corresponds to a `2D' algorithm (see \cite{DGKSY12}).
% 
% \TODO{explain that this corresponds to 3D Matmul and 2D $N$-body.}

Only two of our running examples satisfy the assumptions in this special case.
As we mentioned above, we will weaken these assumptions in Part~2 to address the general `product case,' as described in \sectns{\ref{sec:bound-prod}~and~\ref{sec:attain-prod}}.
\begin{exNbody}[Part~3/3]
Our assumptions in the preceding paragraphs apply to the $N$--body example, with $x_1=x_2=1$ and $\shbl = 2 = d / d_\text{max}$.
(We will not discuss the database join example here because it does not satisfy the assumptions above.)
\exend
\end{exNbody}
\begin{exMatmul}[Part~5/5]
Here our assumptions in the preceding paragraphs apply, with $x_1=x_2=x_3=1/2$ and $\shbl = 3/2 = d / d_\text{max}$.
\exend
\end{exMatmul}
\begin{exComplicated}[Part~4/4]
This example does not satisfy the assumptions of the preceding paragraphs, so we will instead consider the similar-looking code sample
\begin{align*}
&\For{i_1}{1}{N} ,\quad \For{i_2}{1}{N} ,\quad \For{i_3}{1}{N} ,\quad \For{i_4}{1}{N} ,\quad \For{i_5}{1}{N} ,\quad \For{i_6}{1}{N} ,\\
&\qquad A_1(i_1,i_3,i_6) \plusequals \text{func}_1(A_2(i_1,i_2,i_4), A_3(i_2,i_3,i_5), A_4(i_3,i_4,i_6)) \\
&\qquad A_5(i_2,i_5,i_6) \plusequals \text{func}_2(A_6(i_1,i_4,i_5), A_3(i_3,i_4,i_6))
\end{align*}
where we have added the additional subscript $i_5$ to the access of $A_5$, leading to
\[
\Delta = 
\bordermatrix{
  & i_1 & i_2 & i_3 & i_4 & i_5 & i_6 \cr
A_1 & 1 & 0 & 1 & 0 & 0 & 1 \cr
A_2 & 1 & 1 & 0 & 1 & 0 & 0 \cr
A_{3,1} & 0 & 1 & 1 & 0 & 1 & 0 \cr
A_4,A_{3,2} & 0 & 0 & 1 & 1 & 0 & 1 \cr
A_5 & 0 & 1 & 0 & 0 & 1 & 1 \cr
A_6 &  1 & 0 & 0 & 1 & 1 & 0}
\]
and the solution $x_1 = \cdots = x_6 = 1/3$, so now our previous assumptions are satisfied: $\Delta \cdot x = 1$, and $\shbl = 2 = d / d_\text{max} = 6/3$.
\exend
\end{exComplicated}

%% Previous text on 1.5D N-body:
%Parallel algorithms which attain this bound are described in \cite{DGKSY12}; like parallel matrix multiplication, it is possible to use more memory per processor to reduce communication below the lower bound based on using the least memory $\Theta(N/p)$ per processor.
%%
%The paper \cite{DGKSY12} also shows how to attain the lower bound for the common case of $N$-body with a ``cutoff distance,'' i.e.\ where interactions are computed only for pairs of particles that are sufficiently close together.
%
%
%% Previous text on 2.5D Matmul:
%But on a parallel machine with $p$ processors, each processor may do only ${1/p}^\text{th}$ of the work and store ${1/p}^\text{th}$ of the data. 
%%
%In this case $M$ could be $3N^2/p$, if each processor stored an equal fraction of the $A$, $B$ and $C$ matrices, but the lower bound is not restricted to $M$ this small. 
%%
%In other words, the lower bound holds out the possibility of using more memory (larger $M$) to further reduce the communication cost to $O((N^3/p)/M^{1/2})$.
%%
%This is in fact possible (see \cite{2.5D_EuroPar}) up to a limit on the largest $M$ that can be used \cite{BDHLS12}.
%%
%For a survey of communication-optimal linear algebra algorithms (dense and sparse, sequential and parallel), see \cite[\S6]{BallardDemmelHoltzSchwartz11}.
%
%
%% Previous text on 2.5D Tensor Contraction:
% Again, it may improve parallel performance to use more than 
% than minimum amount of memory per processor, if possible \cite{CTF12}.

%
\section{Conclusions} \label{sec:conclusions}

Our goal has been to extend the theory of communication lower bounds,
and to design algorithms that attain them. Motivated by
a geometric approach first used for the 3 nested loops
common in linear algebra, we extended these bounds to a very general 
class of algorithms, assuming only that that the data accessed 
can be modeled as arrays with subscripts that are linear functions 
of loop indices. This lets us use a recent result in functional analysis 
to bound the maximum amount of data reuse ($F$), which leads to a 
communication lower bound.
The bound is expressed in terms of a linear program that can be written down and solved decidably.
%Unfortunately, to be able effectively write it down for certain programs turns out to to be equivalent to Hilbert's Tenth Problem over the rational numbers, whose decidability has long been an open question. 
For many programs of practical interest, it it possible 
%not only to write down the communication lower bound explicitly, but 
to design a communication-optimal algorithm that attains the bound. This is true, for example,
when all the array subscripts are just subsets of the loop indices.
This includes, among other examples, linear algebra, direct $N$--body interactions,
database joins,
and tensor contractions.

Part~2 of this paper will both expand the classes of algorithms 
discussed in \sectn{\ref{sec:bound}}
for which the lower bound can be easily computed,
and expand the 
discussion of \sectn{\ref{sec:attain}}
on the attainability of our lower bounds.

Part~2 will also improve the algorithm in \sectn{\ref{sec:computeP}} to improve its efficiency. 
First, we will show that it suffices to check condition 
\eqref{subcriticalhypothesis} for a \emph{subset} $\LL$ 
of subgroups $H$ of $\ZZ^d$, rather than \emph{all} subgroups.
In this discrete extension of \cite[Theorem~1.8]{Valdimarsson07}, 
previously mentioned in Remark~\ref{rmk:V07}, %\sectn{\ref{sec:guessingsubgroups}}, 
$\LL$ is the lattice generated by the kernels of the homomorphisms, 
$\Kernel{\phi_1},\ldots,\Kernel{\phi_j}$.
% If $\LL$ is finite, then there are no decidability issues --- we simply enumerate the elements of $\LL$ and the corresponding inequalities \eqref{subcriticalhypothesis}.
We will discuss a few practical cases where $\LL$ is finite and the approach simplifies further: this 
includes the case when there are at most $m=3$ arrays 
in the inner loop, as well as a generalization of the 
product case from \sectn{\ref{sec:bound-prod}}.

Then, we discuss a few other cases where $\LL$ may be infinite, 
but we know in advance that we need only check a certain finite subset of subgroups $H$ from $\LL$.
This includes the case when all the arrays are $1$-- or $(d-1)$--dimensional, 
and when there are at most $d=4$ loops in the loop nest.

We then return to the question of attainability, i.e., finding 
an optimal algorithm.
We discuss two inequalities, \eqref{mainconclusion} and 
$|E| \le M^\shbl$, in which equality must be (asymptotically) attainable 
to find an optimal algorithm.
For the former inequality, we show that asymptotic equality is always 
attainable by a family of cubes in a critical subgroup.
For the latter inequality, we give a special case where asymptotic 
equality also holds, while leaving the general case open.
These results suggest how to reorganize code to obtain an optimal algorithm; 
we discuss practical concerns, including constants hidden in the 
asymptotic bounds.
We conclude by discussing how our lower bounds depend on the presence of 
(inter-iteration) data dependencies.

\section*{Acknowledgements}
We would like to thank Bernd Sturmfels, Grey Ballard, Benjamin Lipshitz, and Oded Schwartz for helpful discussions.

We acknowledge funding from Microsoft (award 024263) and Intel (award 024894), and matching funding by UC Discovery (award DIG07-10227), with additional support from ParLab affiliates National Instruments, Nokia, NVIDIA, Oracle, and Samsung, % Demmel, Yelick (Par Lab)
and support from MathWorks. % Demmel (MathWorks)
We also acknowledge the support of the US DOE (grants
DE-SC0003959,          % Demmel (EASI & CACHE)
DE-SC0004938,          % Demmel (MAGMA)
DE-SC0005136,          % Demmel (XSTACK)
DE-SC0008700,          % Demmel (DEGAS)
DE-AC02-05CH11231,     % Demmel, Yelick (overall LBL, includes FASTMATH, Yelick's DEGAS)
DE-FC02-06ER25753, and % Yelick (PModels)
DE-FC02-07ER25799),    % Yelick (CSCADS)
DARPA (award HR0011-12-2-0016),  % Demmel, (ASPIRE)
and NSF (grants
DMS-0901569 and % Christ
DMS-1001550).   % Scanlon
\bibliographystyle{plain}
{\small%
\bibliography{CDKSY13}}
\end{document}